\documentclass[12pt]{amsart}

\usepackage{amssymb,latexsym}
\usepackage{bm}
\usepackage{graphicx,calc}
\usepackage{psfrag}       
\usepackage{wrapfig}
\usepackage{color}
\usepackage{enumitem}
\setlist[enumerate]{parsep=0pt plus 4pt,topsep=0pt plus 4pt}
\usepackage{url}
\usepackage{hyperref}
\definecolor{darkblue}{RGB}{0,0,160}
\hypersetup{colorlinks,citecolor=black,filecolor=black,%
            linkcolor=darkblue,urlcolor=darkblue}
\definecolor{lightred}{rgb}{1,.3,.3}
\definecolor{darkpurple}{rgb}{.7,0,.7}

\newcommand\blu{\color{blue}}
\newcommand\pur{\color{darkpurple}}
\allowdisplaybreaks

\newcommand{\excise}[1]{}

\newtheorem{thm}{Theorem}[section]
\newtheorem{lemma}[thm]{Lemma}

\newtheorem{cor}[thm]{Corollary}
\newtheorem{prop}[thm]{Proposition}
\newtheorem{conj}[thm]{Conjecture}

\newtheorem{question}[thm]{Question}

\theoremstyle{definition}
\newtheorem{example}[thm]{Example}
\newtheorem{remark}[thm]{Remark}
\newtheorem{defn}[thm]{Definition}

\numberwithin{equation}{section}

\newcounter{separated}

\newcommand{\Ring}[1]{\ensuremath{\mathbb{#1}}}

\renewcommand\>{\rangle}
\newcommand\<{\langle}
\newcommand\0{\mathbf{0}}
\newcommand\NN{\Ring{N}}
\newcommand\OO{\mathcal{O}}
\newcommand\RR{{\mathbb R}}
\newcommand\ZZ{{\mathbb Z}}
\newcommand\bb{{\mathbf b}}
\newcommand\cc{{\mathbf c}}
\newcommand\ee{{\mathbf e}}
\newcommand\ff{{\mathbf f}}
\newcommand\ii{{\mathbf i}}
\newcommand\jj{{\mathbf j}}
\newcommand\kk{\Bbbk}
\newcommand\mm{{\mathfrak m}}
\newcommand\pp{{\mathfrak p}}
\newcommand\qq{{\mathbf q}}

\newcommand\uu{{\mathbf u}}
\newcommand\vv{{\mathbf v}}
\newcommand\ww{{\mathbf w}}
\newcommand\xx{{\mathbf x}}
\newcommand\zz{{\mathbf z}}
\newcommand\cA{\mathcal{A}}
\newcommand\cB{\mathcal{B}}
\newcommand\cC{\mathcal{C}}

\newcommand\cF{\mathcal{F}}
\newcommand\cH{H}

\newcommand\cM{M}
\newcommand\cN{N}
\newcommand\cP{P}
\newcommand\cQ{Q}
\newcommand\cS{\mathcal{S}}
\newcommand\cT{\mathcal{T}}

\newcommand\oD{\hspace{.3ex}\ol{\hspace{-.3ex}D\hspace{-.05ex}}\hspace{.05ex}}
\newcommand\oS{\hspace{.3ex}\ol{\hspace{-.3ex}\cS\hspace{-.05ex}}\hspace{.05ex}}

\newcommand\del{\partial}
\newcommand\sss{\mathbf{s}}

\renewcommand\aa{{\mathbf a}}
\renewcommand\phi{\varphi}

\newcommand\gt{\Gamma_{\!\tau}}

\newcommand\op{\mathrm{op}}

\newcommand\cfq{\cF_\cQ}
\newcommand\cfr{\cF_{\RR^2}}
\newcommand\cmr{\cM^{\rho\hspace{-.2ex}}}

\newcommand\cmt{\cM\hspace{-1pt}/\tau}
\newcommand\cnr{\cN\hspace{-1pt}/\rho}
\newcommand\cnt{\cN\hspace{-1pt}/\tau}
\newcommand\dda{\Delta^{\!D}{\hspace{-1.1ex}}_\aa}
\newcommand\ddb{\Delta^{\!D}_\bb}

\newcommand\dsd{\delta^\sigma\hspace{-.25ex} D}
\newcommand\dsm{\delta^\sigma\hspace{-.25ex}\cM}
\newcommand\dst{\delta^{\sigma/\tau}}
\newcommand\dzt{D/\hspace{.3ex}\ZZ\tau}
\newcommand\fqo{\cfq^\op}
\newcommand\fro{\cfr^{\hspace{.2ex}\op}}
\newcommand\mtd{\textstyle \max_{\hspace{.3ex}\tau\hspace{-.3ex}} D}
\newcommand\mti{\textstyle \max_{\hspace{.3ex}\tau\hspace{-.3ex}} I}

\newcommand\nda{\nabla_{\hspace{-.35ex}D}^\aa}

\newcommand\oas{\OO_{\hspace{-.05ex}\aa}^{\hspace{.09ex}\sigma}}
\newcommand\qpc{\cQ_+^\circ}
\newcommand\qrs{\cQ/\hspace{.2ex}\RR\sigma}
\newcommand\qrr{\cQ/\hspace{.2ex}\RR\rho}
\newcommand\qrt{\cQ/\hspace{.2ex}\RR\tau}

\newcommand\qzr{\cQ/\hspace{.2ex}\ZZ\rho}
\newcommand\qzt{\cQ/\hspace{.2ex}\ZZ\tau}
\newcommand\qrrp{\cQ_+/\hspace{.2ex}\RR\rho}
\newcommand\qrsp{\cQ_+/\hspace{.2ex}\RR\sigma}
\newcommand\qrtp{\cQ_+/\hspace{.2ex}\RR\tau}
\newcommand\qztp{\cQ_+/\hspace{.2ex}\ZZ\tau}

\newcommand\too{\longrightarrow}

\newcommand\from{\leftarrow}
\newcommand\into{\hookrightarrow}

\newcommand\kats{\kk_\sigma[\aa+\tau]}
\newcommand\kbrx{\kk_\xi[\bb+\rho]}
\newcommand\nabk{\nabla\hspace{-.2ex}\sigma_k}
\newcommand\nabr{\nabla\hspace{-.2ex}\rho}
\newcommand\nabs{\nabla\hspace{-.2ex}\sigma}
\newcommand\nabt{\nabla\hspace{-.2ex}\tau}
\newcommand\nabx{\nabla\hspace{-.2ex}\xi}
\newcommand\onto{\twoheadrightarrow}

\newcommand\spot{{\hbox{\raisebox{1pt}{\tiny$\scriptscriptstyle\bullet$}}}}

\newcommand\minus{\smallsetminus}
\newcommand\nabro{\Delta\rho}

\newcommand\simto{\mathrel{\!\ooalign{$\fillrightmap$\cr\raisebox{.75ex}{$\,\sim\ \hspace{.2ex}$}}}}

\newcommand\goesto{\rightsquigarrow}
\newcommand\dirlim{\varinjlim}
\newcommand\invlim{\varprojlim}

\newcommand\nothing{\varnothing}

\newcommand\fillrightmap{\mathord- \mkern-6mu
	\cleaders\hbox{$\mkern-2mu \mathord- \mkern-2mu$}\hfill
	\mkern-6mu \mathord\rightarrow}

\renewcommand\iff{\Leftrightarrow}
\renewcommand\epsilon{\varepsilon}
\renewcommand\implies{\Rightarrow}

\newcommand\dd[2][\!D]{\Delta^{#1}_{#2}}
\newcommand\dr[1][]{\delta_{\hspace{-.2ex}\rho\hspace{.15ex}}^{\hspace{.15ex}#1}}
\newcommand\ds[1][\ ]{\delta^{\sigma\hspace{-1.1ex}}{}_{#1\hspace{.2ex}}}
\newcommand\dt[1][\ ]{\operatorname{\delta_{\tau\hspace{-1ex}}{}^{#1}\hspace{-.2ex}}}
\newcommand\dx[1][]{\delta^{\hspace{.1ex}\xi}}
\newcommand\lr[1][]{\del_{\hspace{-.2ex}\rho\hspace{.2ex}}^{\hspace{.15ex}#1}}

\newcommand\lx[1][\xi]{\del^{\hspace{.15ex}#1}}
\newcommand\nd[1][\ ]{\nabla_{\!D}^{#1}}
\newcommand\ol[1]{{\overline{#1}}}

\newcommand\wt[1]{{\widetilde{#1}}}
\newcommand\ats[1][\sigma]{\operatorname{\mathit{A}}_\tau^{#1}\hspace{-.2ex}}

\newcommand\csp[1][\tau']{\operatorname{\mathcal S}_{#1\!}}
\newcommand\cst[1][]{\operatorname{\mathcal S}_{\tau\!}^{#1\hspace{-.3ex}}}
\newcommand\ctr{\mathcal{T}_{\!\rho\,}}
\newcommand\dsp{\delta^{\sigma'}\hspace{-.09ex}}

\newcommand\qnk{\cQ_{\nabk}}
\newcommand\qnp{\cQ_{\nabs'}}
\newcommand\qns[1][]{\cQ_{\nabs#1}}
\newcommand\qnt{\cQ_{\nabt}}
\newcommand\qnx[1][]{\cQ_{\nabx#1}}
\newcommand\qny[1][]{\cQ_{\naby[#1]}}
\newcommand\sts[1][\sigma]{\operatorname{\mathit{S}}_\tau^{#1}\hspace{-.2ex}}
\newcommand\naby[1][y]{\nabla\hspace{-.2ex}#1}

\newcommand\socc[1][]{\ol{\mathrm{soc}}_{#1\,}} 

\newcommand\socp[1][\tau']{\operatorname{soc}_{\hspace{.02ex}#1\hspace{-.15ex}}}
\newcommand\socr[1][]{\operatorname{soc}_{\hspace{.02ex}\rho}^{#1\!\!}\hspace{-.15ex}}
\newcommand\soct[1][]{\operatorname{soc}_{\hspace{.02ex}\tau}^{#1\!\!}\hspace{-.15ex}}
\newcommand\topc[1][]{\ol{\mathrm{top}}_{#1\,}} 

\newcommand\topr[1][]{\operatorname{top}_{\hspace{-.1ex}\rho}^{\hspace{.15ex}#1\!\!}}

\newcommand\socct[1][\tau]{\ol{\operatorname{soc}}_{\hspace{.02ex}#1}^{\,}\hspace{.1ex}}

\newcommand\topcr[1][\ ]{\ol{\mathrm{top}}{}_{\rho\!}^{\hspace{.15ex}#1}} %
\newcommand\cpsoc{\cP\hbox{\rm-}\socc}
\newcommand\cptop{\cP\hbox{\rm-}\topc}
\newcommand\fqsoc{\fqo\hspace{-.1ex}\hbox{\rm-}\hspace{.1ex}\socc}
\newcommand\nrsoc{\nabr\hspace{.25ex}\hbox{\rm-}\hspace{.1ex}\socc}
\newcommand\nrtop{\nabro\hspace{.25ex}\hbox{\rm-}\hspace{.1ex}\topc}
\newcommand\ntsoc{\nabt\hbox{\rm-}\hspace{.1ex}\socc}

\newcommand\posoc{\cP^\op\hbox{\rm-}\socc}

\newcommand\qtsoc{(\qrt)\hspace{-.05ex}\hbox{\rm-}\hspace{.1ex}\socc}

\newcommand\rnsoc{\cQ\hspace{-.1ex}\hbox{\rm-}\hspace{.1ex}\socc}

\renewcommand\top{\operatorname{top}}

\newcommand{\aoverb}[2]{{\genfrac{}{}{0pt}{1}{#1}{#2}}}
\def\twoline#1#2{\aoverb{\scriptstyle {#1}}{\scriptstyle {#2}}}

\DeclareMathOperator\ass{Ass} 
\DeclareMathOperator\att{Att} 
\DeclareMathOperator\Hom{Hom} 
\DeclareMathOperator\soc{soc} 
\DeclareMathOperator\hhom{
	\hspace{.6pt}{\underline{\hspace{-.6pt}{\rm Hom}\hspace{-.6pt}}\hspace{1pt}}}

\DeclareMathOperator\soco{soc_{\hspace{.1ex}\0}} 
\DeclareMathOperator\image{im} 

\newcommand\noheight[1]{\raisebox{0pt}[0pt][0pt]{#1}}

\begin{document}

\mbox{}
\vspace{-4.1ex}
\title[Essential graded algebra over polynomial rings with real exponents]%
      {Essential graded algebra over\\polynomial rings with real exponents}
\author{\vspace{-1.1ex}Ezra Miller}
\address{Mathematics Department\\Duke University\\Durham, NC 27708}
\urladdr{\url{http://math.duke.edu/people/ezra-miller}}

\makeatletter
  \@namedef{subjclassname@2020}{\textup{2020} Mathematics Subject Classification}
\makeatother

\date{5 July 2025; \textit{Declarations of interest}: none;
funded in part by NSF DMS-1702395}

\begin{abstract}\vspace{-.5ex}
The geometric and algebraic theory of monomial ideals and multigraded
modules is initiated over real-exponent polynomial rings and, more
generally, monoid algebras for real polyhedral cones.  The main
results include the generalization of Nakayama's lemma; complete
theories of minimal and dense primary, secondary, and irreducible
decomposition, including associated and attached faces; socles and
tops; minimality and density for downset hulls, upset covers, and
fringe presentations; Matlis duality; and geometric analysis of
staircases.  Modules that are semialgebraic or piecewise-linear (PL)
have the relevant property preserved by functorial constructions as
well as by minimal primary and secondary decompositions.  And when the
modules in question are subquotients of the group itself, such as
monomial ideals and quotients modulo them, minimal primary and
secondary decompositions are canonical, as are irreducible
decompositions up to the new real-exponent notion of density.
\end{abstract}
\maketitle

\vspace{-4.22ex}
\setcounter{tocdepth}{2}
\tableofcontents

\vspace{-4ex}
\section{Introduction}\label{s:intro}

\subsection*{Overview}\label{b:overview}

Little is known about the algebra of rings of polynomials whose
exponents are allowed to be nonnegative real numbers instead of
integers.  The extreme failure of the noetherian condition---ideals
can be uncountably generated---and nontrivial topology on exponent
sets present daunting technical difficulties.  The small amount of
existing literature proceeds by restricting to monomial ideals that
are finitely generated, in an appropriate sense
\cite{ingebretson--sather-wagstaff2013,
andersen--sather-wagstaff2015}, or to multigraded modules that are
finitely presented \cite{lesnick-interleav2015}.  Other work can be
viewed as touching on the continuous nature of exponent sets via
nondiscrete-monoid algebras
\cite{anderson-coykendall-hill-zafrullah2007}.  But the general
behavior of modules over real-exponent polynomial rings remains wide
open, even in the special case of monomial ideals.  Beyond its
intrinsic value, the issue has risen to prominence because modules
over polynomial and semigroup rings with dense or continuous real
exponents emerge in quantum noncommutative toric geometry
\cite{katzarkov-lupercio-meersseman-verjovsky2020} and applied
topology \cite{multiparamPH} (see also \cite{fruitFlyModuli}), which
both focus on the multigraded~setting.

This paper breaks ground on the earnest study of modules over
real-exponent polynomial rings, including the usual setting where
exponents lie in a right-angled nonnegative orthant but also real
analogues of affine semigroup rings, with exponents in arbitrary
pointed polyhedral cones.  This first step of the investigation
concerns monomial ideals and multigraded modules, which are general
enough to exhibit the starkly different behavior resulting from
continuous exponents and deviation from noetherianity but have enough
combinatorial structure to allow complete treatment of basic theory,
such as primary decomposition, Nakayama's lemma, and minimal
presentations.

The algebraic development hinges on a number of foundations whose
elementary versions for finitely generated or noetherian modules fail
when straightforwardly generalized to real exponents but nonetheless
admit fully functioning analogues when appropriately enhanced.  Most
importantly, detecting
\begin{itemize}
\item%
an injective homomorphism of modules by checking at all associated
primes or
\item%
a surjective homomorphism by Nakayama's lemma (Krull--Azumaya theorem)
\end{itemize}
falter at the outset: modules need not contain copies of quotients by
prime ideals, so the notion of associated prime requires serious
thought; and dually, modules---even as simple as monomial ideals---do
not have minimal generators \cite{ingebretson--sather-wagstaff2013},
so considerations surrounding Nakayama's lemma require just as much
attention.

The solutions developed here to tackle these problems construct a
topological framework for concepts of minimality, in the form of dense
generator and cogenerator functors.  For example, with definitions
made properly, arbitrary real-exponent monomial ideals have canonical
monomial primary decompositions that are minimal in a strong sense,
generalizing the situation for polynomial and other affine semigroup
rings.  And these decompositions are similarly derived from canonical
irreducible decompositions.  But the notion of ``irredundant'' for
irreducible decomposition must be revised: components can be omitted
as long as those that remain are dense in the sense developed~here.

Density engages with continuity of exponent sets in a way that gives
hope of lifting lessons learned here for monomial ideals and
multigraded modules to arbitrary ideals and modules.  In the meantime,
the results here open the door to the vast literature surrounding
monomial ideals and multigraded modules, from Stanley--Reisner theory
to homological algebra, in the uncharted context of real exponents.
These developments are important by virtue of the central roles of
real multigraded algebra independently in quantum noncommutative toric
geometry \cite{katzarkov-lupercio-meersseman-verjovsky2020}, which
rests on dense finitely generated additive subgroups of~$\RR^n$
(already there are payoffs; see \cite{gldim}, which identifies smooth
cases), and in the rapidly developing field of topological data
analysis, where real parameters have been present since the
introduction of multiple parameters to persistent homology
\cite{multiparamPH} but mathemtical foundations have been sorely
lacking.%
\enlargethispage*{.5ex}%

\addtocontents{toc}{\protect\setcounter{tocdepth}{2}}\vspace{-.65ex}
\subsection*{Acknowledgements}

Justin Curry provided feedback after listening for hours about face
posets infinitesimally near real persistence parameters; he enhanced
the functorial viewpoint and provided references as well as insight on
topics from sheaf theory to real algebraic geometry.  Ashleigh Thomas
played a crucial role in the genesis of this algebraic theory of real
multipersistence.  Ville Puuska read an inchoate version of this
manuscript \cite[\S6--14]{qr-codes} extremely carefully; among his
valuable comments, he indicated a need for certain hypotheses in
Matlis~duality.  A~referee~made~\mbox{excellent}~\mbox{comments}.

\subsection*{Main results}\label{b:main}

The remainder of this Introduction provides an account of the main
results, along the way comparing and contrasting them with the usual
noetherian case.  Readers for whom these comparisons may not be useful
can skip to Section~\ref{b:pogroups} and then to
Section~\ref{s:staircases}, where the novel content begins, although
it might be worth skimming Section~\ref{s:intro} to get a feel for
phenomena unique to the real-exponent case.  Readers skipping to
Section~\ref{s:staircases} should refer back to
Section~\ref{s:pogroups} for background as needed; a~guide to which
parts of Section~\ref{s:pogroups} are used where is included in its
opening paragraphs.

\subsection{Real exponent issues}\label{b:trouble}

\begin{example}\label{e:maximal}
Let $\RR_+$ denote the nonnegative real numbers, so $\RR^n_+$ is the
nonnegative orthant in~$\RR^n$.  In the real-exponent polynomial ring
$\kk[\RR^n_+]$, a \emph{monomial ideal} is an ideal generated by
monomials: $I = \<\xx^\aa \mid \aa \in {A}\>$ for some $A \subseteq
\RR^n_+$, where $\xx^\aa = x_1^{a_1}\!\cdots x_n^{a_n}$.  For example,
the maximal graded ideal $\mm = \<x_1^{b_1},\dots,x_n^{b_n} \mid b_i >
0\ \text{for}\ i = 1,\dots,n\>$, which is the monomial ideal generated
by all positive powers of the individual variables, has the exponent
set%
\vspace{-2ex}%
$$
  \psfrag{x}{\footnotesize$x$}
  \psfrag{y}{\footnotesize$y$}
  \includegraphics[height=20ex]{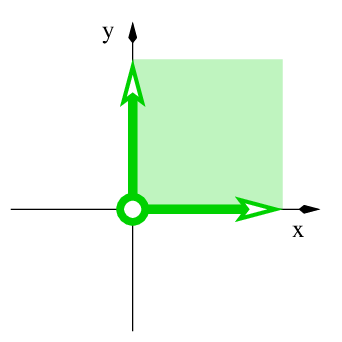}
\vspace{-1ex}%
$$
which is the nonnegative orthant with the origin missing.  This ideal
is not finitely generated, although it is countably generated by the
powers $x_1^{b_1},\dots,x_n^{b_n}$ of individual variables for any
sequence of strictly positive vectors $\bb = (b_1,\dots,b_n)$
converging to~$\0$.  This ideal also has no minimal generating set,
because deleting any finite subset of a sequence converging to~$\0$
still results in a sequence converging to~$\0$.  This phenomenon was
observed in \cite{ingebretson--sather-wagstaff2013}.
\end{example}

\begin{example}\label{e:trouble}
The monomial ideal $I = \<\xx^\aa \mid a_1+\dots+a_n = 1\>
\hspace{-1pt}\subseteq\hspace{-1pt} \kk[\RR^n_+]$ has exponent set
depicted on the left:
\vspace{-1ex}%
$$
  \psfrag{x}{\footnotesize$x$}
  \psfrag{y}{\footnotesize$y$}
  \includegraphics[height=20ex]{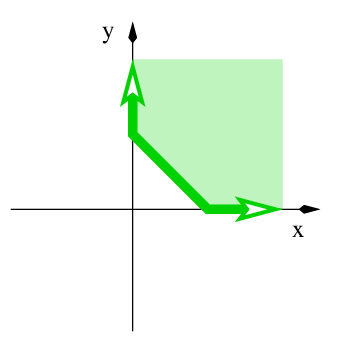}
  \qquad\qquad\qquad
  \includegraphics[height=20ex]{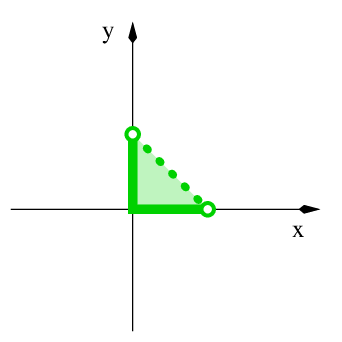}
\vspace{-1ex}%
$$
This ideal is uncountably generated, with the given generators forming
the unique minimal monomial generating set.  The quotient module $M =
\kk[\RR^n_+]/I$, depicted on the right, has open upper boundary.  This
module would appear to be primary to the maximal ideal~$\mm$, but as
$\mm^d = \mm$ for all positive integers~$d$, the module~$M$ is not
annihilated by any power of~$\mm$.  Worse, $M$ has no elements
annihilated by the maximal ideal: $M$ has no simple submodules, so
$\mm$ is not an associated prime in the usual sense of $M$ containing
a copy of~$\kk = \kk[\RR^n_+]/\mm$.  In that usual sense, the set of
such copies, namely the socle $\Hom(\kk,M)$, would be an essential
submodule of~$M$, which by definition intersects every nonzero
submodule of~$M$ nontrivially.  But in this picture the $\Hom$
vanishes.  Moreover, an essential submodule requires containing a
strip of locally positive width near the upper boundary, but the
intersection of all essential~submodules~is~$0$.
\end{example}

\pagebreak[3]

\begin{question}\label{q:nakayama}
In Example~\ref{e:maximal}, what should Nakayama's lemma say?
\end{question}

\begin{question}\label{q:prim-decomp}
In Example~\ref{e:trouble}, what should the statement,
``A~homomorphism $M \to N$ is injective if and only if $M_\pp \into
N_\pp$ is injective for all associated primes~$\pp$ of~$M$'' say?
\end{question}

\noindent
These questions are precisely dual to each other when considered from
a \mbox{judicious}~\mbox{angle}.

\subsection{Nakayama's lemma and primary decomposition}\label{b:summary}\mbox{}

\medskip
\noindent
The first answers to Questions~\ref{q:nakayama}
and~\ref{q:prim-decomp} are two major contributions of this~paper,
\begin{itemize}
\item%
Theorem~\ref{t:surjection-RR} on detecting surjectivity by generator
functors (tops),
\end{itemize}
and the Matlis dual (Section~\ref{b:matlis}) from which it is
deduced,
\begin{itemize}
\item%
Theorem~\ref{t:injection} on detecting injectivity by cogenerator
functors (socles).
\end{itemize}
However, the final answers reflect the observations in
Examples~\ref{e:maximal} and~\ref{e:trouble} that generating sets and
essential submodules in the setting of real exponents maintain those
properties when replaced with dense approximations.  The precise
formulation of this deeply non-discrete observation yields two
additional major contributions,
\begin{itemize}
\item%
the density enhancements in Theorem~\ref{t:dense-top} and
Theorem~\ref{t:dense-subfunctor}.
\end{itemize}

Flowing from these foundational results, particularly from the socle
injectivity criterion that is Theorem~\ref{t:injection}, are other
staples of commutative algebra:
\begin{itemize}
\item%
primary decompositions, minimal in a strong sense
(Theorems~\ref{t:interval-hull} and~\ref{t:minimal-primary});

\item%
canonical minimal primary decompositions of monomial ideals
(Theorem~\ref{t:hull-I});

\item%
irreducible decomposition for monomial ideals that are canonical and
irredundant up to taking dense subsets (Theorem~\ref{t:interval=union}
and Corollary~\ref{c:interval=union});

\item%
duals of all these for secondary decomposition and attached primes
(\mbox{Section}~\ref{s:tops}).
\end{itemize}

\begin{example}\label{e:min-primary}
The upper boundary of the interval in~$\RR^2$ at the left end of the
display
$$
\psfrag{x}{\tiny$x$}
\psfrag{y}{\tiny$y$}
  \kk\!
  \left[
  \begin{array}{@{\!}c@{\!\!}}\includegraphics[height=29mm]{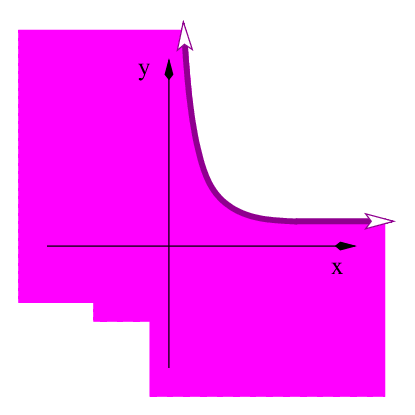}\end{array}
  \right]
\:\into\ 
  \kk\!
  \left[
  \begin{array}{@{\!}c@{\!\!}}\includegraphics[height=29mm]{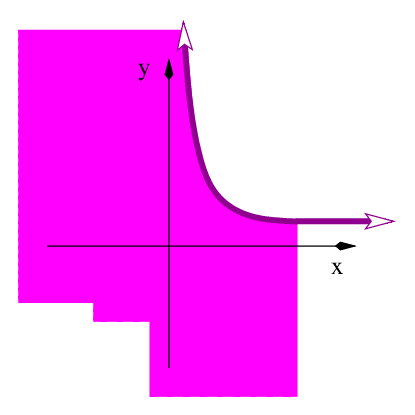}\end{array}
  \right]
\oplus\,
  \kk\!
  \left[
  \begin{array}{@{\!}c@{\!\!}}\includegraphics[height=29mm]{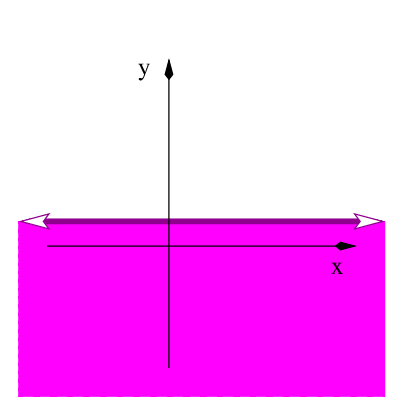}\end{array}
  \right]
$$
has the vertical axis as an asymptote, whereas the horizontal axis is
exactly parallel to the positive horizontal end of the upper boundary.
The corresponding interval module has the indicated canonical minimal
primary decomposition by Theorem~\ref{t:hull-I}.
\end{example}

The minimality in Theorem~\ref{t:minimal-primary} is a requirement
that the socle of the module should map isomorphically to the direct
sum of the socles of the quotients modulo its primary components
(Definition~\ref{d:minimal-primary}).  This isomorphism is stronger
than usually proposed in noetherian commutative algebra.  When applied
to an injective hull or irreducible decomposition in noetherian
situations, socle-minimality as in Definition~\ref{d:minimal-primary}
is equivalent to there being a minimal number of indecomposable
summands.  In contrast, minimal primary decompositions in noetherian
commutative algebra do not require socle-minimality in any sense; they
stipulate only minimal numbers of summands, with no conditions on
socles.  This has unfortunate consequences: even in noetherian
settings, different choices of primary components for an embedded
(i.e., nonminimal) prime can strictly contain one another, for
example.  Requiring socle-minimality as in
Definition~\ref{d:minimal-primary} recovers a modicum of uniqueness
over arbitrary noetherian rings, as socles are functorial even if
primary components themselves need not be.  This tack is more commonly
taken in combinatorial commutative algebra, typically involving
objects such as monomial or binomial ideals.  In particular, the
``witnessed'' forms of minimality for mesoprimary decomposition
\cite[Definition~13.1 and Theorem~13.2]{mesoprimary} and irreducible
decomposition of binomial ideals \cite{soccular} serve as models for
the type of minimality in primary decompositions considered here.

In ordinary noetherian commutative algebra socle-minimal primary
decompositions are anyway automatically produced by the usual
existence proof, which leverages the noetherian hypothesis to create
an irreducible decomposition.  Indeed, a noetherian primary
decomposition is socle-minimal if and only if each primary component
is obtained by gathering some of the components in a minimal
irreducible decomposition.  When real exponents enter, truly minimal
irreducible decompositions are impossible by
Theorem~\ref{t:interval=union} and Corollary~\ref{c:interval=union},
which force us to settle for irredundancy up to taking dense subsets.
Nonetheless, the primary component formed by gathering all irreducible
components with a given associated face is well defined, regardless of
which dense subset of irreducible components was present.  That is how
uniqueness of the minimal primary decomposition in
Theorem~\ref{t:hull-I} arises even from nonunique
irreducible~decomposition.

Secondary decomposition is lesser known, even to algebraists, than its
Matlis dual, primary decomposition, but secondary decomposition has
been in the literature for decades \cite{kirby1973,
macdonald-secondary-rep1973,northcott-gen-koszul1972} (see
\cite[Section~1]{sharp-secondary1976} for a brief summary of the main
concepts).  The unfamiliarity of secondary decomposition and its
related functors is a primary reason why the bulk of the technical
development over real exponents is carried out in terms of
cogenerators and socles instead of generators and tops.

Minimal primary and dense irreducible decomposition owe their
existence to definitions tailored to real exponents.  These include
especially
\begin{itemize}
\item%
a definition of associated prime by socle nonvanishing
(Definition~\ref{d:associated}) that yields

\item%
characterizations of coprimary modules as those with only one
associated prime (Proposition~\ref{p:elementary-coprimary} and
Theorem~\ref{t:coprimary}).
\end{itemize}
These, in turn, rely on the heart of the matter regarding density,
which draws, at the most fundamental level, on the topological algebra
surrounding real exponents:
\begin{itemize}
\item%
the characterization of essential submodules by socle inclusion
(Theorem~\ref{t:essential-submodule}).
\end{itemize}

\subsection{Socles, cogenerators, and staircases}\label{b:staircases}\mbox{}

\medskip
\noindent
\enlargethispage*{0ex}%
The results discussed thus far all rest on the main socle injectivity
criterion in Theorem~\ref{t:injection}.  As such, the entire edifice
is built on socles.%
	\footnote{This is appropriate to the English definition of
	\emph{socle}: the base of a column.}
Identifying the right definition of socle in Section~\ref{s:socle} to
account for the departure from discrete exponents is the most subtle
and difficult aspect of the theory.  But the answer turns out to be
pretty and, as luck would have it, finite.

The problem to be overcome is seen in Example~\ref{e:trouble}: the
socle of the quotient module~$M$ there should lie along its upper%
	\footnote{It is an accident of history that in illustrations,
	socles lie along upper boundaries instead of along lower
	boundaries at the bottoms of pictures, where tops quite
	unfortunately reside.\vspace{-30pt}\mbox{}}
boundary,
but the upper boundary is~missing: $M$~is zero in the corresponding
$\RR^n$-graded degrees.  The solution is to keep track of the
directions in which limits must be taken to reach the missing boundary
points.  The finiteness of the answer comes down to the fact that it
matters only which of the finitely many faces of the exponent cone the
limits are taken along.
The main product of Section~\ref{s:socle} is not a theorem, but
nonetheless a major contribution, namely
\begin{itemize}
\item%
the notion of socle in Definition~\ref{d:soct}.\ref{i:global-soc-tau}
as well as
\item%
cogenerator and nadir in Definition~\ref{d:soct}.\ref{i:global-cogen}.
\end{itemize}
Readers from persistent homology should view these as functorializing
the notion of ``closed or open right endpoint of an interval'' and
generalizations to more parameters.

\begin{example}\label{e:soc-min-primary}
Consider the module
$\kk[\raisebox{-.6ex}[0pt][0pt]{\includegraphics[height=2.8ex]{decomp}}]$
at the left-hand end of Example~\ref{e:min-primary}.  Any point along
the curved portion of its upper boundary---that is, along the upper
boundary of the middle illustration---represents a usual (``closed'')
socle element of
$\kk[\raisebox{-.6ex}[0pt][0pt]{\includegraphics[height=2.8ex]{decomp}}]$
(Definition~\ref{d:socc}), because such an element is annihilated by
moving up in any direction, including straight vertically or
horizontally, so it yields an injection $\kk \into
\kk[\raisebox{-.6ex}[0pt][0pt]{\includegraphics[height=2.8ex]{decomp}}]$
in the relevant multigraded degree.  In contrast, the horizontal ray
in the upper boundary represents a closed socle element of
$\kk[\raisebox{-.6ex}[0pt][0pt]{\includegraphics[height=2.8ex]{decomp}}]$
along the $x$-axis (Definition~\ref{d:socct}), because it (i)~extends
infinitely far to the right, so it yields an injection
$\kk[x\text{-axis}_+] \into
\kk[\raisebox{-.6ex}[0pt][0pt]{\includegraphics[height=2.8ex]{decomp}}]$
but (ii)~is annihilated upon moving upward in any direction, notably
the vertical direction.
\end{example}

\begin{example}\label{e:trouble-soc}
In the right-hand illustration from Example~\ref{e:trouble}, the
module is~$0$ at any point along the antidiagonal upper boundary line
segment.  However, any such point can be approached within the
interior of the triangle from below or from the left:%
\vspace{-1.5ex}%
$$
  \psfrag{x}{\footnotesize$x$}
  \psfrag{y}{\footnotesize$y$}
  \includegraphics[height=13.25ex]{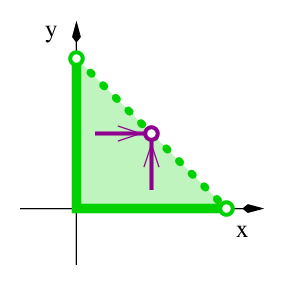}
\vspace{-2.5ex}%
$$
This open boundary point represents an element in the socle
(Definition~\ref{d:soc}).  Limits can be taken along any nonzero face
of the cone $\RR^2_+$ to reach this point, but the two minimal such
faces, namely the positive $x$-axis and the positive $y$-axis, are the
two nadirs of this cogenerator
(Definition~\ref{d:soct}.\ref{i:global-cogen}) along the face $\tau =
\{\0\}$ comprising~the~origin.
\end{example}

\pagebreak[3]

\begin{example}\label{e:nadir-x}
As the upper boundary in Example~\ref{e:min-primary} is closed, its
has no open cogenerators: its socle is closed.  So modify the interval
there by omitting the horizontal ray and leaving the rest of the
points as they are.
$$
\psfrag{x}{\tiny$x$}
\psfrag{y}{\tiny$y$}
  \kk\!
  \left[
  \begin{array}{@{\!}c@{\!\!}}\includegraphics[height=29mm]{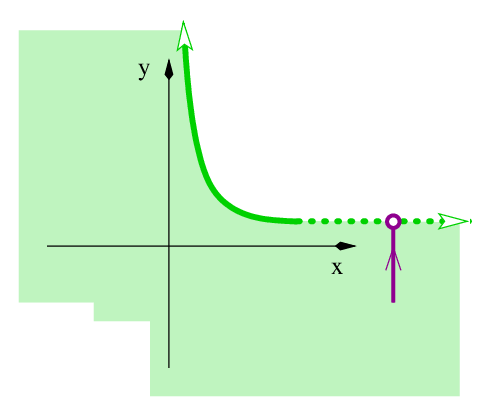}\end{array}
  \right]
\quad
\stackrel{\text{localization}}{\:\goesto\ }
\quad
  \kk\!
  \left[
  \begin{array}{@{\!}c@{\!\!}}\includegraphics[height=29mm]{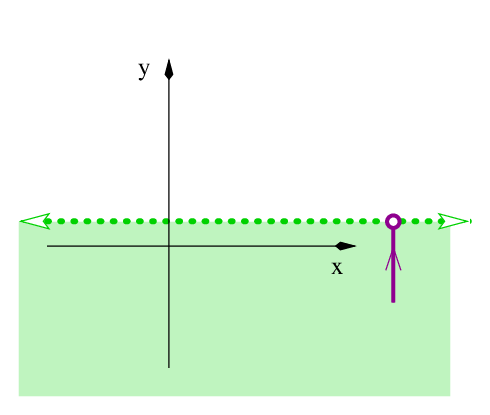}\end{array}
  \right]
\quad
\stackrel{\text{quotient-restriction}}{\:\goesto\ }
\quad
  \kk\!
  \left[
  \begin{array}{@{\!}c@{\!\!}}\includegraphics[height=29mm]{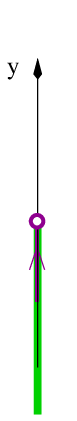}\end{array}
  \right]
$$
The missing horizontal boundary ray represents an element in the socle
along the~\mbox{$x$-axis}.  In more detail, the ray yields a summand
of the upper closure (Definition~\ref{d:upper-closure}) that contains
a copy of~$\kk[x\text{-axis}_+]$.  Localizing along the $x$-axis
yields the translation-invariant version in the middle picture, which
can be reached by vertical limits from the interior of the lower
half-plane.  The quotient-restriction
(Definition~\ref{d:quotient-restriction}), thought of either as
modding out by horizontal translation or by restricting to a vertical
slice, yields a situation analogous to Example~\ref{e:trouble-soc}.
The nadir (Definition~\ref{d:soct}.\ref{i:global-cogen}) along the
face $\tau = x\text{-axis}_+$ is the open interior face of~$\RR^2_+$.
Upon quotient-restriction along~$\tau$, this nadir maps to the
vertical axis depicted at right in the figure, but the nadir itself is
formally labeled by the preimage face, which by definition
contains~$\tau$.
\end{example}

The socle along a face~$\tau$ of the exponent cone is a module over
the quotient $\qrt$ of the ambient real vector space~$\cQ$ modulo its
subspace generated by~$\tau$, via quotient-restriction
(Definition~\ref{d:quotient-restriction}).  That aspect is not novel
to real exponents; it is the analogue of the socle
$\Hom_R(R/\pp,M)_\pp$ at a prime~$\pp$ of positive dimension in a
ring~$R$ being a module over the local ring~$R_\pp$.  What is novel,
however, is the set of nadirs, as in the examples.  A~nadir of
positive dimension indicates an ``open'' cogenerator of~$M$, which is
not an element of~$M$ but an element in its upper closure~$\delta\cM$
(Definition~\ref{d:upper-closure}).  The upper closure is a module
not merely over a real-exponent polynomial ring, but also over the
face poset of the exponent cone.  That is the crucial ingredient
entailed by real exponents: every point of the grading group~$\cQ$
gets replaced by an infinitesimal copy of the face poset of the
exponent cone.

This treatment of socles highlights the price to pay for real
exponents.  First, generalizing standard constructions from noetherian
commutative algebra demands care.  Notably, for instance, localization
fails to commute with $\Hom$, materially complicating proofs; see
Remark~\ref{r:soc-vs-supp}, which explains how this failure to commute
is not an artifact of the proofs but rather an intrinsic facet of the
real-exponent theory.  Second, socles of modules over real exponents
are not submodules, but instead are functorially manufactured from
auxiliary modules derived from~$\cM$, namely upper closure
modules~$\delta\cM$.  Honest submodules must be reconstructed from
cogenerators (this is done in Section~\ref{s:density}).  Finally, it
would have been nice to develop module theory over real exponents
entirely within the language of monoid algebras, but the infinitesimal
structure of real-exponent polynomial rings is unavoidably
poset-theoretic in nature, the poset being the face lattice of the
positive cone of exponents.  Hence this paper is phrased in terms of
modules over posets \cite{hom-alg-poset-mods, prim-decomp-pogroup},
whose theory is reviewed~in~Section~\ref{s:pogroups}.

The geometry and combinatorics that rules the construction of socles,
especially the entrance of face posets, is that of downsets and their
boundaries, otherwise known as \emph{staircases}
\cite[\S2]{graphMinRes}.  These are rich objects at the interface of
geometry, algebra, and combinatorics with connections to other areas
of mathematics and science; see \cite{okounkov16}, for example, where
they are Ising crystals at zero temperature, or
\cite{berkouk-petit2019}, where the modules called ``ephemeral'' in
topological data analysis are those with no upper boundary along the
interior of the exponent cone.
The functorial viewpoint on staircases in Section~\ref{s:staircases},
particularly
\begin{itemize}
\item%
the upper closure functors (Definition~\ref{d:upper-closure}),

\item%
what it means to divide an upper closure element
(Definition~\ref{d:divides}), and

\item%
computations of downset upper boundaries
(Lemma~\ref{l:downset-upper-closure} and
Proposition~\ref{p:downset-upper-closure})
\end{itemize}
translate topological limits in partially ordered real vector spaces
into algebraic
colimits on modules.  Taking ordinary socles of the upper closure of
a module~$M$, to get a module over the real exponent ring and the face
poset of the exponent cone, completes the detection of missing
boundary points that the ordinary socle of~$M$ itself misses.

Section~\ref{s:semialg} assures that functorial constructions
surrounding socles preserve additional semialgebraic or piecewise
linear structure when they are present in the input.  To wit,
\begin{itemize}
\item%
Theorem~\ref{t:endofunctor} says that left-exact functors with
predictable actions on quotients modulo monomial ideals preserve
additional geometric structure, and

\item%
Theorem~\ref{t:soct-tame} verifies the hypotheses to draw
this conclusion for the cogenerator functors, which take socles.
\end{itemize}
This conclusion is reasonable, because socles take each downset to a
well behaved subset of its closure
(Lemma~\ref{l:downset-upper-closure} and
Proposition~\ref{p:downset-upper-closure}).  Preserving additional
geometric structure is particularly crucial for algorithms: it is
hopeless for a computer to manipulate an arbitrary real-exponent
monomial ideal, for its staircase could be missing a Cantor set or
some more arbitrary, unfathomable antichain.  Multigraded modules that
arise in practice---such as from persistent homology---come from
finite, computable procedures.  Constraints from linear inequalities
or comparisons among (squared) distances between points or other
simple geometric objects yield PL or \mbox{semialgebraic} modules.

The entire theory for tame modules over real exponents has a simpler
analogue over affine semigroup rings.  It is barely new, being based
on more elementary foundations, but it is worth phrasing precisely and
collecting the results for the record (Section~\ref{s:discrete}).

The dual discrete theory surrounding generator functors is
interspersed with the corresponding real-exponent theory of tops in
Sections~\ref{s:gen-functors}--\ref{s:tops}, which is Matlis duality
applied to earlier sections; see especially
\begin{itemize}
\item%
\hspace{-1.1pt}Theorem~\ref{t:top=socvee}: socle and top duality over
real-exponent polynomials, and

\item%
\hspace{-1.1pt}Theorem~\ref{t:topc=socc^vee}: closed socle and top
duality over partially ordered abelian~\mbox{groups}.
\end{itemize}
To locate results for affine semigroup rings, look for results over
arbitrary partially ordered abelian groups, such as
Theorem~\ref{t:topc=socc^vee}, or look for the keywords ``discrete
polyhedral group''; these indicate the same context as ``affine
semigroup'' (Definition~\ref{d:polyhedral}) but refer to modules over
posets, which is the language adopted (of necessity) for real
exponents.

The paper closes with its final major results, in a discussion
(Section~\ref{s:min}) of
\begin{itemize}
\item%
minimal presentations, including downset, upset, and fringe
presentations; and

\item%
resolutions, including conjectures about minimal lengths of
resolutions as well as lengths of minimal resolutions along the lines
of the Hilbert syzygy theorem.
\end{itemize}

\section{Algebra over partially ordered abelian groups}\label{s:pogroups}

The algebra surrounding localization, support, primary decomposition,
and Matlis duality works over a broad class of partially ordered
abelian groups with finitely many faces, appropriately defined
\cite{prim-decomp-pogroup}, and sometimes in more generality.  The
main settings of this paper restrict primarily to the continuous case
of real vector spaces but also secondarily the discrete case of
finitely generated free abelian groups.  Nonetheless, since some of
the surrounding algebra works for all partially ordered abelian
groups, this section reviews the basic setup, always indicating the
allowed generality.  For reference, the definitions and claims in
Sections~\ref{b:pogroups}, \ref{b:tame}, and the start
of~\ref{b:localization} are taken from \cite[\S2, \S3.1, \S4.1,
\S4.5]{hom-alg-poset-mods} and \cite[\S5]{prim-decomp-pogroup}; those
in Section~\ref{b:support} are taken from \cite[\S4,
\S5]{prim-decomp-pogroup}.
%
The expositions in Section~\ref{b:matlis} and the remainder of
Section~\ref{b:localization} do not review specific sources, as
their levels of generality are likely new, though they build on
indicated well known material in straightforward~ways.

Notation and concepts surrounding partially ordered abelian groups
(Section~\ref{b:pogroups}) are used freely throughout, so all readers
should begin there before proceeding to Section~\ref{s:staircases}.
Localization (Section~\ref{b:localization}) is not used until socles
along faces of positive dimension are introduced in
Section~\ref{b:socc+}, where localization becomes fundamental for the
duration.  Global support (Section~\ref{b:support}) is essential for
interactions with primary decomposition and minimality starting in
Section~\ref{s:minimality}, specifically
Proposition~\ref{p:sigma-nbd-cogen}, and continuing through
Section~\ref{s:hulls}; see Lemma~\ref{l:nearby} for a particularly
tight encapsulation of the connection.  Matlis duality
(Section~\ref{b:matlis})
translates statements about cogenerators and socles to statements
about generators and tops starting in Section~\ref{s:gen-functors}.

\subsection{Real and discrete polyhedral groups}\label{b:pogroups}\mbox{}

\medskip
\noindent
The modules in this paper are families of vector spaces indexed by
partially ordered sets that are also vector spaces.  Basic notions and
notations surrounding those are introduced here, including faces,
upsets, and downsets.

\begin{defn}\label{d:poset-module}
Let $Q$ be a partially ordered set (\emph{poset}) and~$\preceq$ its
partial order.  A \emph{module over~$Q$} (or a \emph{$Q$-module})
is
\begin{itemize}
\item%
a $Q$-graded vector space $M = \bigoplus_{q\in Q} M_q$ with
\item%
a homomorphism $M_q \to M_{q'}$ whenever $q \preceq q'$ in~$Q$
such that
\item%
$M_q \to M_{q''}$ equals the composite $M_q \to M_{q'} \to
M_{q''}$ whenever $q \preceq q' \preceq q''$.
\end{itemize}
A \emph{homomorphism} $M \to \cN$ of $Q$-modules is a
degree-preserving linear map, or equivalently a collection of vector
space homomorphisms $M_q \to \cN_q$, that commute with the structure
homomorphisms $M_q \to M_{q'}$ and $\cN_q \to \cN_{q'}$.
\end{defn}

The posets of interest in the paper are the following, primarily the
real case in Definition~\ref{d:polyhedral}, although some results are
naturally stated in the generality of Definition~\ref{d:pogroup}.

\begin{defn}\label{d:pogroup}
An abelian group~$Q$ is \emph{partially ordered} if it is generated
by a submonoid~$Q_+$, called the \emph{positive cone}, that has
trivial unit group.  The partial order is: $q \preceq q' \iff q' - q
\in Q_+$.
\end{defn}

\begin{defn}\label{d:polyhedral}
A \emph{real polyhedral group}~$Q$ is a real vector space of finite
dimension partially ordered so that its positive cone~$\cQ_+$ is an
intersection of finitely many closed half-spaces.  The notation
$\RR^n$ and especially $\RR^n_+$ is reserved for the case where the
positive cone is the nonnegative orthant, so the partial order is
componentwise comparison.
\end{defn}

\begin{example}\label{e:polyhedral-discrete}
A \emph{discrete polyhedral group} is a finitely generated free
abelian group partially ordered so that its positive cone is a
finitely generated submonoid.  (Equivalently a discrete polyhedral
group is the Grothendieck group of an affine semigroup with trivial
unit group.)  The notation $\ZZ^n$ is reserved for the special case
where the positive cone is the nonnegative orthant~$\NN^n$, so the
partial order is componentwise~comparison.
\end{example}

\begin{defn}\label{d:face}
A \emph{face} of a partially ordered abelian group~$Q$ is a submonoid
$F \subseteq \cQ_+$ of the positive cone whose complement in~$\cQ_+$
is an ideal of~$\cQ_+$.  The face~$F$ can also be referred to as a
face of~$\cQ_+$ rather than a face of~$\cQ$.
\end{defn}

\begin{remark}\label{r:polyhedral}
Definition~\ref{d:polyhedral} and~\ref{e:polyhedral-discrete} are
instances of the class of \emph{polyhedral partially ordered groups},
introduced in \cite[Definition~2.8]{prim-decomp-pogroup}, which have
finitely many faces.
\end{remark}

\begin{defn}\label{d:indicator}
Fix a poset~$Q$.  The vector space $\kk[Q] = \bigoplus_{q\in Q} \kk$
that assigns $\kk$ to every point of~$Q$ is a $Q$-module with identity
maps on~$\kk$.  More generally,
\begin{enumerate}
\item\label{i:upset}%
an \emph{upset} (also called a \emph{dual order ideal}) $U \subseteq
Q$, meaning a subset closed \mbox{under} going upward in~$Q$ (so $U +
Q_+ = U$, when $Q$ is
an abelian group) determines an \emph{indicator submodule} or
\emph{upset module} $\kk[U] \subseteq \kk[Q]$;
\item\label{i:downset}%
dually, a \emph{downset} (also called an \emph{order ideal}) $D
\subseteq Q$, meaning a subset closed under going downward in~$Q$ (so
$D - Q_+ = D$, when $Q$ is
an abelian group) determines an \emph{indicator quotient module} or
\emph{downset module} $\kk[Q] \onto \kk[D]$; and
\item\label{i:interval}%
an \emph{interval} $I \subseteq Q$, meaning the intersection of an
upset and a downset, determines an \emph{interval module} $\kk[I]$
that is a subquotient of~$\kk[Q]$: if $I = U \cap D$ then $\kk[I]
\into \kk[D]$ and $\kk[U] \onto \kk[I]$, since $I$ is an upset in~$D$
and a downset in~$U$.
\end{enumerate}
\end{defn}

\begin{remark}\label{r:Q-graded}
For polyhedral groups the language of \mbox{$Q$-modules} is equivalent
to that of $Q$-graded $\kk[Q_+]$-modules.  Indeed, a module over any
partially ordered abelian group~$Q$ is the same thing as a
$Q$-graded\/ module over the monoid algebra~$\kk[Q_+]$ of the positive
cone \cite[Lemma~2.6]{prim-decomp-pogroup}.  When $Q = \ZZ^n$ and $Q_+
= \NN^n$, the relevant monoid algebra is the polynomial ring
$\kk[\NN^n] = \kk[\xx]$, where $\xx = x_1,\dots,x_n$ is a sequence of
$n$ variables.  This is the classical case; see \cite[\S8.1]{cca}, for
instance.  When $Q = \RR^n$ and $Q_+ = \RR^n_+$, the relevant monoid
algebra is the \emph{real-exponent polynomial ring} $\kk[\RR^n_+]$,
whose elements are polynomials in $\xx = x_1,\dots,x_n$ with exponents
that are nonnegative real numbers.
\end{remark}

\subsection{Localization and restriction}\label{b:localization}\mbox{}

\medskip
\noindent
One of the main reasons to work over partially ordered abelian groups
instead of over arbitrary posets is that groups admit translations.
Localization of a set or module along a face forces---in a universal
way---translation along the group generated by that face to be a
symmetry.  Subsequently taking a quotient modulo the translation
action has the same effect as slicing by a transverse subspace.  The
overall effect is to reduce the dimension of the given face to~$0$.
Algebraic manifestations of this process culminate with exactness in
Lemma~\ref{l:exact-qr}.

\begin{defn}\label{d:localization}
Fix a face~$\tau$ of a partially ordered abelian group~$Q$.  The
\emph{localization} of a $Q$-module~$M$ \emph{along~$\tau$} is the
tensor product
$$
  M_\tau = M \otimes_{\kk[Q_{+\!}]} \kk[Q_{+\!} + \ZZ \tau],
$$
viewing~$M$ as a $Q$-graded $\kk[Q_{+\!}]$-module.
\end{defn}

\begin{example}\label{e:localization}
Depicted on the left is an interval in~$\RR^2$
(Definition~\ref{d:indicator}.\ref{i:interval}).
\begin{center}
\psfrag{x}{\raisebox{-.5ex}{\tiny$\scriptscriptstyle x$}}
\psfrag{y}{\tiny$\scriptscriptstyle \!\!y$}
\raisebox{-.7\height}{\includegraphics[width=18ex]{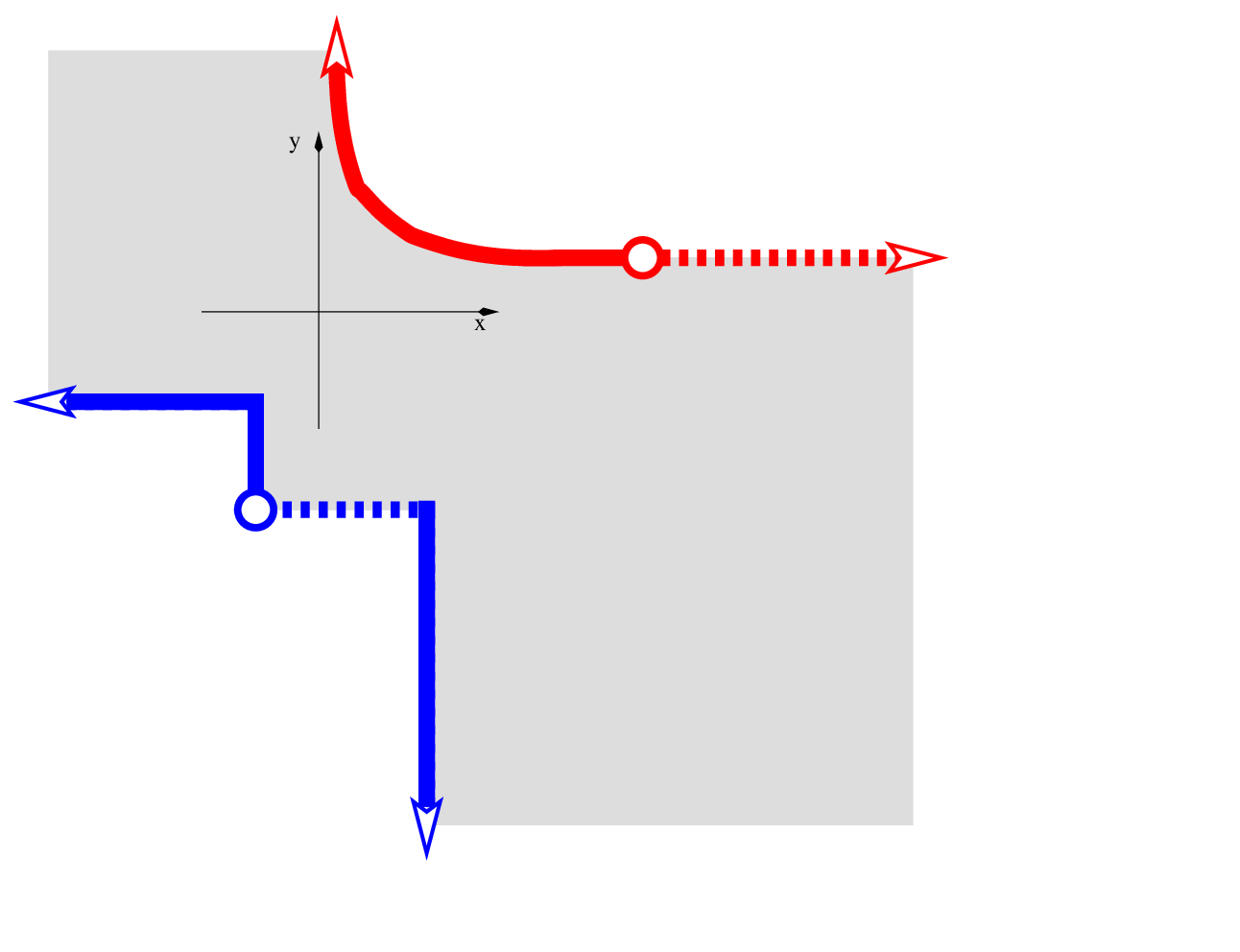}}
\quad\hspace{-4ex}$\overset{\text{localize along }x\text{-axis}}\goesto$\quad
\raisebox{-.7\height}{\includegraphics[width=18ex]{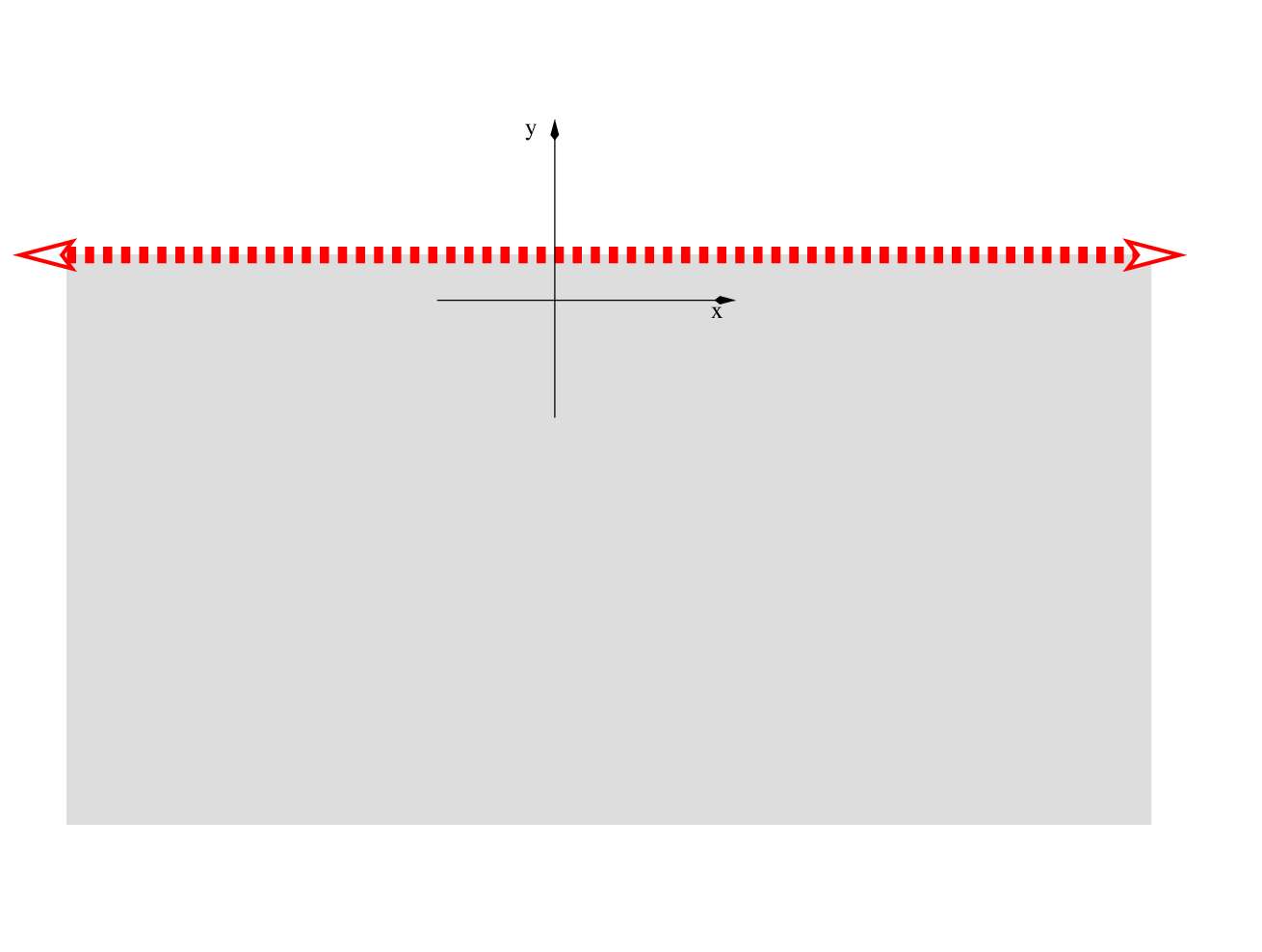}}\hspace{-5ex}
\end{center}
The corresponding interval module
(Definition~\ref{d:indicator}.\ref{i:interval}) generated along the
blue path has localization along the $x$-axis that is also an interval
\pagebreak[3]%
module, namely for the interval depicted on the right.  The process of
localization here can be thought of as moving the original image
infinitely far to the left---that is, negatively along the $x$-axis.
\end{example}

\begin{defn}\label{d:localize-downset}
The \emph{localization} $D_\tau$ of a downset~$D$ is the downset with
\mbox{$\kk[D]_\tau \hspace{-2.17134pt}=\hspace{-2pt} \kk[D_\tau]$}.
\end{defn}

\begin{example}\label{e:localize-downset}
The interval depicted on the left side of Example~\ref{e:localization}
is not a downset.  However, deleting the (blue) lower boundary results
in a downset whose upper boundary is the (red) upper curve.  The
localization of that downset along the $x$-axis is depicted on the
right side of Example~\ref{e:localization}.  Why does localization
along the $x$-axis not care about the (blue) lower generating curve?
Because any interval with this (red) upper boundary curve yields the
depicted localization along the $x$-axis as long as the interval
extends infinitely downward.
\end{example}

The remainder of Section~\ref{b:localization} is new in this
generality, although $\ZZ^n$-graded versions date back to
\cite[\S4]{alexdual1998} and \cite[\S3.6]{alexdual2000}

\begin{defn}\label{d:qrt}
For a partially ordered abelian group~$\cQ$ and a face~$\tau$
of~$\cQ_+$, write $\qzt$ for the quotient of~$\cQ$ modulo the subgroup
generated by~$\tau$.  If $\cQ$ is a real polyhedral group then write
$\qrt = \qzt$.
\end{defn}

\begin{remark}\label{r:qrt}
The image $\qztp$ of~$\cQ_+\!$~in~$\qzt$ is a submonoid that generates
$\qzt$, but $\qztp$ can have units, so $\qzt$ need not be partially
ordered in a natural way.  For example, when $\cQ = \ZZ^2$ and the
columns of $\big[\twoline {2\ 1\ 0}{0\ 1\ 1}\big]$ generate~$\cQ_+$,
taking $\tau = \big\<\big[\twoline 20\big]\big\>$ to be the face along
the $x$-axis yields a quotient monoid $\cQ_+/\ZZ\tau \cong \ZZ/2\ZZ
\oplus \NN$ with torsion.  However, if $\cQ$ is a real polyhedral
group then the group of units (lineality space) of the cone $\cQ_+ \!+
\RR\tau$ is just~$\RR\tau$ itself, because $\cQ_+$ is pointed, so
$\qrt$ is a real polyhedral group whose positive cone $(\qrt)_+ =
\qrtp$ is the image of~$\cQ_+$.  Similar reasoning applies to the
intersection of the real polyhedral situation with any subgroup
of~$\cQ$; this includes the case of normal affine semigroups, where
the subgroup~of~$\cQ$~is~discrete.
\end{remark}

\begin{lemma}\label{l:quotient-restriction}
The subgroup $\ZZ\tau \subseteq \cQ$ of a partially ordered abelian
group~$\cQ$ acts freely on~the localization $\cM_\tau$ of any
$\cQ$-module~$\cM$ along a face~$\tau$.  Consequently, if $I_\tau
\subseteq \kk[\cQ_+]$ is the augmentation ideal $\<m - 1 \mid m \in
\kk[\tau]$ is a monomial\/$\>$, then the $\qzt$-graded module $\cmt =
\cM/I_\tau\cM$ over the monoid algebra $\kk[\qztp]$ satisfies
$$
  \cM_\tau \cong \bigoplus_{q \,\mapsto\, \wt q} (\cM/\tau)_{\wt q}.
$$
\end{lemma}
\begin{proof}
The monomials of $\kk[\cQ_+]$ corresponding to elements of~$\tau$ are
units on~$\cM_\tau$ acting as translations along~$\tau$.  Since the
augmentation ideal sets every monomial equal to~$1$, the quotient $\cM
\to \cmt$ factors through~$\cM_\tau$.
\end{proof}

\begin{defn}\label{d:quotient-restriction}
The $\kk[\qzt]$-module $\cmt$ in Lemma~\ref{l:quotient-restriction} is
the \emph{quotient-restriction} of~$\cM$ along~$\tau$.
\end{defn}

\begin{example}\label{e:quotient-restriction}
In the situation of Example~\ref{e:localization}, the
quotient-restriction along the $x$-axis is described in detail and
illustrated in Example~\ref{e:nadir-x}.
\end{example}

\begin{remark}\label{r:quotient-functor}
Over a real polyhedral group~$\cQ$, or over any subgroup of~$\cQ$, the
functor $\cM_\tau \mapsto\nolinebreak \cmt$ has a ``section'' $\cmt
\mapsto \cM_\tau|_\tau{}_{{}^{\!\perp}}$, where $\cN|_{\tau^\perp} =
\bigoplus_{\aa\in\tau^\perp} \cN_\aa$ is the \emph{restriction}
of~$\cN$ to any linear subspace~$\tau^\perp$ complementary
to~$\RR\tau$.  (When $\cQ_+ = \RR^n_+$, a complement is canonically
spanned by the face orthogonal to~$\tau$, or equivalently, the unique
maximal face of~$\RR^n_+$ intersecting~$\tau$ trivially.)  The
restriction is a module over the real polyhedral group $\tau^\perp$
with positive cone $(\cQ_+ \!+ \RR\tau) \cap \RR\tau^\perp$, which
projects isomorphically to the positive cone of~$\qrt$.  Thus the
quotient-restriction is both a quotient and a restriction
of~$\cM_\tau$.  While a section can exist over polyhedral partially
ordered groups that are not real, it need not.  For example, when the
quotient monoid $\cQ_+/\ZZ\tau$ has torsion, as in the case detailed
in Remark~\ref{r:qrt}, the torsion prevents the functor
$\kk[\cQ_+]_\tau \mapsto \kk[\cQ_+]/\tau$ from having a section to any
category of modules over a subgroup of~$\cQ$.
\end{remark}

\begin{lemma}\label{l:exact-qr}
The quotient-restriction functors $\cM \mapsto \cM/\tau$ are exact.
\end{lemma}
\begin{proof}
Localizing along~$\tau$ is exact because the localization $\kk[\cQ_+\!
+\nolinebreak \ZZ\tau]$ of~$\kk[\cQ_+]$ is flat as a
$\kk[\cQ_+]$-module.  The exactness of the functor that takes each
$\kk[\qztp]$-module $\cM_\tau$ to~$\cmt$ can be checked on each
$\qzt$-degree individually.
\end{proof}

\subsection{Support and primary decomposition}\label{b:support}\mbox{}

\medskip
\noindent
Elements in modules over partially ordered groups can persist
indefinitely or eventually become~$0$ in each direction.  The relevant
directions can always be viewed as occurring along faces of the
positive cone.  Making these notions precise gives way to primary
decomposition, the goal here in Definition~\ref{d:primDecomp}.

\begin{defn}\label{d:support}
Fix a face~$\tau$ of a partially ordered abelian group~$Q$.  The
submodule of~$M$ \emph{globally supported on~$\tau$}~is
$$
  \gt M
  =
  \bigcap_{\tau' \not\subseteq \tau}\bigl(\ker(M \to M_{\tau'})\bigr)
  =
  \ker \bigl(M \to \prod_{\tau' \not\subseteq \tau} M_{\tau'}\bigr).
$$
Fix a $Q$-module $M$ for a polyhedral partially ordered group~$Q$.
The \emph{local $\tau$-support} of~$M$ is the module $\gt M_\tau$ of
elements globally supported on~$\tau$ in the localization~$M_\tau$, or
equivalently \cite[Proposition~4.6]{prim-decomp-pogroup} the
localization along~$\tau$ of the submodule of~$M$ globally supported
on~$\tau$.
\end{defn}

\begin{example}\label{e:global-support}
The global supports of the interval module for the interval
in~$\RR^2$~on~the
$$
\psfrag{x}{\tiny$x$}
\psfrag{y}{\tiny$y$}
\!\!\!\!
  \begin{array}{@{}c@{}l@{}}
	{\pur .}\qquad\quad\ \\[-1.7ex]
	{\pur .}\qquad\quad\ \\[-1.7ex]
	{\pur .}\qquad\quad\ \\[-1ex]
	\includegraphics[height=29mm]{decomp}
	&\raisebox{3ex}{\pur$\!\cdot\!\cdot\!\cdot$}
	\end{array}\!\!
\ \ \goesto\ \
  \begin{array}{@{}c@{}}
	{\pur .}\quad\,\\[-1.7ex]
	{\pur .}\quad\,\\[-1.7ex]
	\makebox[0pt][l]{\quad$\tau = \{\0\}$}
	{\pur .}\quad\,\\[-1ex]
	\includegraphics[height=29mm]{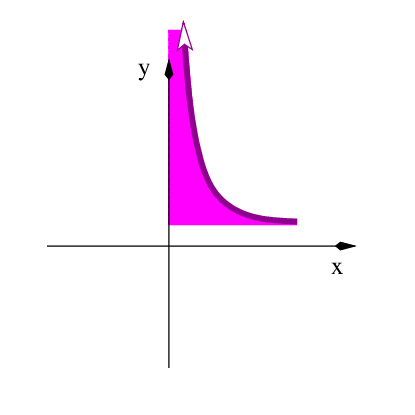}
	\end{array}
\ ,\,\ \
  \begin{array}{@{}c@{}l@{}}
	\phantom{.}\qquad\quad\ \ \\[-1.7ex]
	\phantom{.}\qquad\quad\ \ \\[-1.7ex]
	\makebox[0pt][l]{\qquad$\tau = x$-axis}
	\phantom{.}\qquad\quad\ \ \\[-1ex]
	\includegraphics[height=29mm]{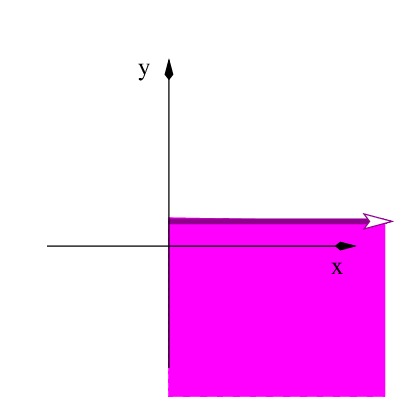}
	&\raisebox{3ex}{\pur$\!\cdot\!\cdot\!\cdot$}
	\end{array}
\ ,\,\ \
  \begin{array}{@{}c@{}}
	{\pur .}\qquad\quad\ \ \\[-1.7ex]
	{\pur .}\qquad\quad\ \ \\[-1.7ex]
	\makebox[0pt][l]{\qquad$\tau = y$-axis}
	{\pur .}\qquad\quad\ \ \\[-1ex]
	\includegraphics[height=29mm]{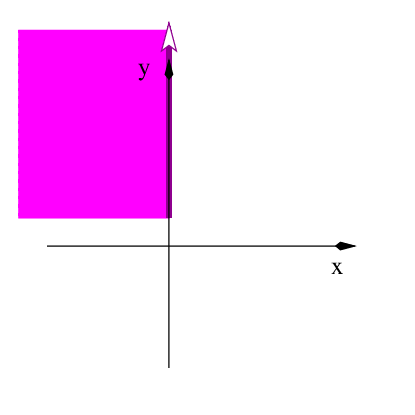}
	\end{array}
$$
left-hand side of this display (the same interval as in
Example~\ref{e:min-primary}) are the interval modules for the
intervals on the right-hand side, each labeled by the face~$\tau$ on
which the support is taken.  On the indicator module for a downset,
the kernel of localization $M \to M_{\tau'}$ records the part of the
downset that disappears under the ``infinite negative motion'' from
Example~\ref{e:localization}.  To reside in the intersection kernels
(or the kernel of the map to the product of localizations) a downset
element must disappear under all of the relevant localizations.  Thus,
for instance, the global support along $\tau = \{\0\}$ is nonzero only
above the horizontal boundary ray and to the right of the vertical
asymptote because any point beneath the horizontal ray or to the left
of the vertical asymptote survives one of the localizations, namely
along the $x$-axis and $y$-axis, respectively.  Note that extending
the first two of these global supports modules downward, so they
become quotients instead of submodules of the original interval
module, yields the
components in Example~\ref{e:min-primary}.
\end{example}

\begin{defn}\label{d:coprimary}%
A module $M$ over a polyhedral partially ordered group
is~\mbox{\emph{coprimary}} if for some face~$\tau$, the localization
map $M \into M_\tau$ is injective and $\gt M_\tau$ is an essential
submodule of~$M_\tau$, meaning every nonzero submodule
of~$M_\tau$\hspace{-1pt} intersects \mbox{$\gt M_\tau$\hspace{-1.4pt}
nontrivially}.
\end{defn}

\begin{example}\label{e:coprimary}
Of the four modules (whose degree sets are) depicted in
Example~\ref{e:global-support}, the three on the right-hand side are
coprimary.  Verifying Definition~\ref{d:coprimary} explicitly, the
localization map $M \to M_{\{\0\}}$ is the identity, and the depicted
$\tau = \{\0\}$ component equals its global support along $\tau =
\{\0\}$ by construction.  For the $\tau = x$-axis component,
localizing along the $x$-axis yields the module on the right-hand side
of Example~\ref{e:localization}; the map from the $x$-axis component
to this localization is injective, and the image is an essential
submodule since every element can be moved to the right (positively in
the $x$-direction) to land in the given $x$-axis component.  In
contrast, the module on the left side of
Example~\ref{e:global-support} is not coprimary because the
localization morphisms along the $x$-axis and $y$-axis are not
injective, and although localization along $\tau = \{\0\}$ is always
injective, the global support along $\tau = \{\0\}$ is not an
essential submodule because any nonzero element sufficiently far to
the right (that is, whose degree has sufficiently positive
$x$-coordinate) generates a submodule whose intersection with the
$\tau = \{\0\}$ component is trivial.
\end{example}

\begin{defn}\label{d:elementary-coprimary}
Fix a face~$\tau$ of the positive cone~$Q_+$ in a partially ordered
abelian group~$Q$.  A homogeneous element $z \in M_q$ in a
$Q$-module~$M$ is
\begin{enumerate}
\item%
\emph{$\tau$-persistent} if it has nonzero image in $M_{q'}$ for all
$q' \in q + \tau$;

\item%
\emph{$\tau$-transient} if, for each $f \in Q_+ \minus \tau$, the
image of~$z$ vanishes in $M_{q'}$ whenever $q' = q + \lambda f$ for
$\lambda \gg 0$;

\item%
\emph{$\tau$-coprimary} if it is $\tau$-persistent and
$\tau$-transient.
\end{enumerate}
\end{defn}

\begin{example}\label{e:elementary-coprimary}
In the left-hand module from Example~\ref{e:global-support}, any
element whose degree lies below or a little above the $x$-axis is
$x$-axis-persistent.  Elements to the right of the $y$-axis die when
pushed far enough in any direction with a positive $y$-component and
hence are $x$-axis-transient.  The intersection of these conditions
yields the $x$-axis component in Example~\ref{e:global-support}.
\end{example}

\begin{prop}[{\cite[Theorem~4.13]{prim-decomp-pogroup}}]\label{p:elementary-coprimary}
Fix a face~$\tau$ of the positive cone~$Q_+$ in a real or discrete
polyhedral partially ordered group~$Q$.  A~$Q$-module~$M$ is
$\tau$-coprimary if and only if every homogeneous element divides a
$\tau$-coprimary element, where $z \in M_q$ \emph{divides} $z' \in
M_{q'}$ if $q \preceq q'$ and $z$ has image~$z'$ under the structure
morphism $M_q \to M_{q'}$.
\end{prop}

\begin{example}\label{e:hyperbola}
In the indicator module for the downset
\begin{center}
\psfrag{x}{\footnotesize$x$}
\psfrag{y}{\footnotesize$y$}
\includegraphics[height=30mm]{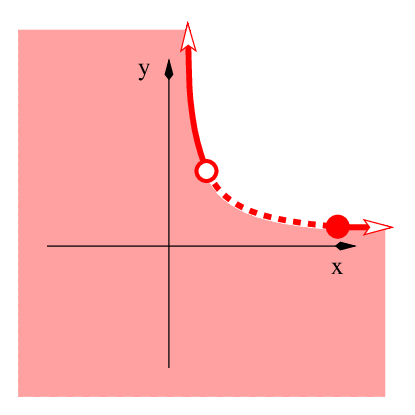}
\end{center}
beneath a hyperbola, not every homogeneous element is
$\{\0\}$-coprimary.  Indeed, only the elements whose degrees lie in
the strictly positive quadrant are $\{\0\}$-coprimary.  However, this
downset module is nonetheless $\{\0\}$-coprimary because every
homogeneous element divides an element in the strictly positive
quadrant.  This analysis does not depend on the upper boundary
hyperbola being present and closed, or not present, or anywhere in
between.
\end{example}

\begin{defn}\label{d:primDecomp}
Fix a $Q$-module $M$ over a polyhedral partially ordered group~$Q$.
A~\emph{primary decomposition} of~$M$ is an injection $M \into
\bigoplus_{i=1}^r M/M_i$ into a direct sum of coprimary quotients
$M/M_i$, called \emph{components} of the decomposition.
\end{defn}

\begin{example}\label{e:primDecomp}
The interval modules on the right-hand side of
Example~\ref{e:min-primary} are coprimary by
Proposition~\ref{p:elementary-coprimary}, as justified by
Example~\ref{e:coprimary}, because the operation of extending a region
from Example~\ref{e:global-support} downward causes every element to
divide an element in the original region.  The homomorphism in
Example~\ref{e:min-primary} is injective because the interval on the
left-hand side is contained in the union of the intervals on the
right-hand side, so every element on the left has nonzero image in at
least one component on the right.  Therefore the injection in
Example~\ref{e:min-primary} constitutes a primary decomposition.
\end{example}

\begin{example}\label{e:interval}
When $M = \kk[I]$ in Definition~\ref{d:primDecomp} is an interval
module, a primary decomposition $\kk[I] \into \bigoplus_{i=1}^r
\kk[I_i]$ may also be expressed as a \emph{primary decomposition}
of~$I$ itself: $I = \bigcup I_i$, where each $I_i$ is a
\emph{coprimary interval} in~$I$.  That is, $I_i$ is an interval
in~$Q$ but a downset in the subposet~$I$.  See
Example~\ref{e:min-primary}, where the interval on the left side is
the union of the two intervals on the right side.
\end{example}

\subsection{Tame, semialgebraic, and PL modules and morphisms}\label{b:tame}\mbox{}

\medskip
\noindent
The lack of a useful noetherian hypothesis for modules over real
polyhedral groups demands alternative finiteness conditions, which are
covered here.  The general idea is that, given a module~$\cM$, the
poset should decompose into finitely many regions on which $\cM$ is
constant.  Subtleties and variations on this notion are covered here,
ending with the concept of downset hull and downset-finiteness because
the theory in later sections is initially developed in terms of
cogenerators and socles rather than the more usual generators and
tops.

\begin{defn}\label{d:constant-subdivision}
A \emph{constant subdivision} of a poset~$Q$ \emph{subordinate} to a
$Q$-module~$M$ is a partition of~$Q$ into \emph{constant regions} such
that for each constant region~$I$ there is a single vector space~$M_I$
with an isomorphism $M_I \to M_\ii$ for all $\ii \in I$ that \emph{has
no monodromy}: if $J$ is some (perhaps different) constant region,
then all comparable pairs $\ii \preceq \jj$ with $\ii \in I$ and $\jj
\in J$ induce the same composite homomorphism $M_I \to M_\ii \to M_\jj
\to M_J$.
\end{defn}


\begin{defn}\label{d:tame}
Fix a poset~$Q$ and a $Q$-module~$M$.
\mbox{}
\begin{enumerate}
\item\label{i:finite-constant-subdiv}%
A constant subdivision of~$Q$ is \emph{finite} if it has finitely
many constant regions.

\item\label{i:Q-finite}%
The $Q$-module~$M$ is \emph{$Q$-finite} if its components $M_q$
have finite dimension over~$\kk$.

\item\label{i:tame}%
The $Q$-module~$M$ is \emph{tame} if it is $Q$-finite and $Q$
admits a finite constant subdivision subordinate to~$M$.
\end{enumerate}
\end{defn}

\begin{defn}\label{d:auxiliary-hypotheses}
Fix a subposet~$\cQ$ of a partially ordered real vector space (e.g.,~a
real polyhedral group).  A partition of~$\cQ$ into subsets~is
\begin{enumerate}
\item\label{i:semialgebraic}%
\emph{semialgebraic} if the subsets are real semialgebraic varieties;

\item\label{i:PL}%
\emph{piecewise linear (PL)} if the subsets are finite unions of
convex polyhedra, where a \emph{convex polyhedron} is an intersection
of finitely many closed or open half-spaces.
\end{enumerate}
A module over~$\cQ$ is \emph{semialgebraic} or \emph{PL} if $\cQ_+$ is
and the module is tame via a subordinate finite constant subdivision
of the corresponding type.
\end{defn}

\begin{defn}\label{d:encoding}
Fix a poset~$\cQ$.  An \emph{encoding} of a $\cQ$-module $\cM$ by a
poset~$\cP$ is a poset morphism $\pi: \cQ \to \cP$ together with a
$\cP$-module $\cH$ such that $\cM \cong \pi^*\cH =
\bigoplus_{q\in\cQ}H_{\pi(q)}$, the \emph{pullback of~$\cH$
along~$\pi$}, which is naturally a $\cQ$-module.  The encoding is
\emph{finite} if
\begin{enumerate}
\item%
the poset $\cP$ is finite, and
\item%
the vector space $H_p$ has finite dimension for all $p \in \cP$.
\end{enumerate}
\end{defn}

\begin{defn}\label{d:subordinate-encoding}
\vbox{
Fix a poset~$\cQ$ and a $\cQ$-module~$\cM$.
\begin{enumerate}
\item\label{i:morphism}%
A poset morphism $\pi: \cQ \to \cP$ or an encoding of a $\cQ$-module
(perhaps different from~$\cM$) is \emph{subordinate} to~$\cM$ if there
is a $\cP$-module~$\cH$ such that~$\cM \cong\nolinebreak \pi^*\cH$.

\item\label{i:auxiliary-encoding}%
When $\cQ$ is a subposet of a partially ordered real vector space, an
encoding of~$\cM$ by~$\pi$ is \emph{semialgebraic} or \emph{PL} if the
constant subdivision of~$\cQ$ formed by the fibers of~$\pi$
\cite[Theorem~4.22]{hom-alg-poset-mods} is of the corresponding type
(Definition~\ref{d:auxiliary-hypotheses}).
\end{enumerate}
}%
\end{defn}

\begin{defn}\label{d:tame-morphism}
A homomorphism $\phi: M \to N$ of modules over a poset~$\cQ$ is
\emph{tame} if $\cQ$ admits a finite constant subdivision subordinate
to both $M$ and~$N$ such that for each constant region~$I$ the
composite
$M_I \to M_\ii \to N_\ii \to N_I$ does not depend~on~$\ii \in I$.
\begin{enumerate}
\item\label{i:subordinate}%
This constant subdivision is \emph{subordinate} to the
morphism~$\phi$.

\item\label{i:dominates}%
The morphism $\phi$ \emph{dominates} a constant subdivision or poset
encoding if the subdivision or encoding is subordinate to~$\phi$.

\item\label{i:semialg-morphism}%
The morphism~$\phi$ is \emph{semialgebraic} or \emph{PL} if it
dominates a constant subdivision of the corresponding type.
\end{enumerate}
\end{defn}

\begin{defn}\label{d:abelian-category}
The \emph{category of tame modules} over a poset~$\cQ$ is the
subcategory of $\cQ$-modules whose objects are the tame modules and
whose morphisms are the tame homomorphisms.  The subcategories of
\emph{semialgebraic modules} and \emph{PL~modules} have the
correspondingly restricted objects and tame morphisms.
\end{defn}

\begin{prop}[{\cite[Proposition~4.28]{hom-alg-poset-mods}}]\label{p:abelian-category}
Over any poset~$\cQ$, the kernel and cokernel of any tame homomorphism
of $\cQ$-modules are tame, finite direct sums of tame modules are
tame, and the set
of tame morphisms from $M$ to~$N$ is an abelian subgroup of
$\Hom(M,N)$.  If~$\cQ$ is a subposet of a partially ordered real
vector space, then the same is true for semialgebraic and PL modules.
\end{prop}

\begin{remark}\label{r:less-than-tameness}
Since socles involve essential submodules, which are divorced from
generators, the theory can often get by with less than tameness, which
requires finite upset covers (Definitions~\ref{d:minimal-cover}) as
well as finite downset hulls as in the next definition.  In the
pictures, only phenomena near upper boundaries matter for socles, not
anything near lower boundaries; see
\cite[Remark~5.6]{prim-decomp-pogroup} for discussion.
\end{remark}

\begin{defn}\label{d:downset-hull}
A \emph{downset hull} of a module~$M$ over an arbitrary poset is an
injection $M \into \bigoplus_{j \in J} E_j$ with each $E_j$ being a
downset module.  The hull is \emph{finite} if $J$ is~finite.  The
module~$M$ is \emph{downset-finite} if it admits a finite downset
hull.
\end{defn}

\subsection{Matlis duality}\label{b:matlis}\mbox{}

\medskip
\noindent
Matlis duality takes vector space duals degree by degree while turning
the poset upside down.  It interchanges generators and cogenerators as
well as flat and injective modules.  Details, notations, and basic
isomorphisms are reviewed here for use in
Sections~\ref{s:gen-functors}--\ref{s:min}.

\begin{defn}\label{d:self-dual}
A poset~$\cQ$ is \emph{self-dual} if it is given a poset isomorphism
$\cQ \simto \cQ^\op$ with its opposite poset.  On elements denote this
isomorphism by $q \mapsto -q$.
\end{defn}

\begin{example}\label{e:self-dual}
Inversion makes partially ordered abelian groups self-dual as posets.
\end{example}

\begin{defn}\label{d:matlis}
Fix a poset~$\cQ$ with opposite poset~$\cQ^\op$.  The \emph{Matlis
dual} of a $\cQ$-module~$\cM$ is the $\cQ^\op$-module~$\cM^\vee$
defined by $(\cM^\vee)_q = \Hom_\kk(\cM_q,\kk)$.  When $\cQ \simto
\cQ^\op$ is a self-duality, then
$$
  (\cM^\vee)_q = \Hom_\kk(\cM_{-q},\kk),
$$
so the homomorphism $(\cM^\vee)_q \to (M^\vee)_{q'}$ is transpose to
$\cM_{-q'} \to \cM_{-q}$.
\end{defn}

\begin{example}\label{e:matlis}
\hspace{-1.5177pt}The Matlis dual over a partially ordered abelian
group~$\cQ$ is \mbox{equivalently}
$$
  \cM^\vee
  =
  \hhom_\cQ(\cM,\kk[\cQ_+]^\vee)
$$
where $\hhom_\cQ(\cM,\cN) = \bigoplus_{q\in\cQ}
\Hom\bigl(\cM,\cN(q)\bigr)$ is the direct sum of all degree-preserving
homomorphisms from~$\cM$ to $\cQ$-graded translates of~$\cN$, i.e.,
$N(q)_a = N_{a+q}$.  This is proved using the adjuntion between $\Hom$
and~$\otimes$; see \cite[Lemma~11.16]{cca}, noting that the nature of
the grading group is immaterial.  And as in \cite[Lemma~11.16]{cca},
$$
  \hhom_\cQ(\cM,\cN^\vee) = (\cM \otimes_\cQ \cN)^\vee.
$$
\end{example}

\begin{example}\label{e:dual-of-localization}
It is instructive to compute the Matlis dual of localization along a
face~$\tau$ over a partially ordered abelian group: the Matlis dual
of~$\cM_\tau$ is
$$
  (\cM_\tau)^\vee
  =
  \hhom\bigl(\kk[\cQ_+]_\tau\otimes\cM,\kk\bigr)
  =
  \hhom\bigl(\kk[\cQ_+]_\tau,\hhom(\cM,\kk)\bigr)
  =
  \hhom\bigl(\kk[\cQ_+]_\tau,\cM^\vee\bigr),
$$
the module of homomorphisms from a localization of~$\kk[\cQ_+]$
into~$\cM^\vee$.  The unfamiliarity of this functor is one of the
reasons for developing most of the theory in this paper in terms of
socles and cogenerators instead of tops and generators.
\end{example}

\begin{lemma}\label{l:vee-vee}
$(\cM^\vee)^\vee\!$ is canonically isomorphic to~$\cM$ if $\cM$ is
$\cQ$-finite (Def.\,\ref{d:tame}.\ref{i:Q-finite}).\qed
\end{lemma}

\begin{remark}\label{r:flat}
Every $\cQ$-finite injective module over a discrete polyhedral
group~$\cQ$ is, by \cite[Theorem~11.30]{cca}, isomorphic to a direct
sum of downset modules~$\kk[D]$ for downsets of the form $D = \aa +
\tau - \cQ_+$ for some vector $\aa \in Q$, said to be
\emph{cogenerated} by~$\aa$ \emph{along} the face~$\tau$.  (The
noetherian hypothesis in \cite[Theorem~11.30]{cca} is satisfied by the
finitely generated assumption in Example~\ref{e:polyhedral-discrete}.)
Taking Matlis duals, every $\cQ$-finite flat module over a discrete
polyhedral group~$\cQ$ is isomorphic to a direct sum of upset
modules~$\kk[U]$ for upsets of the form $U = \bb + \ZZ\tau + \cQ_+$
for some vector $\bb \in Q$, said to be \emph{generated} by~$\bb$
\emph{along} the face~$\tau$.  These upset modules are the graded
translates of localizations of~$\kk[\cQ_+]$~along~faces.
\end{remark}

\begin{lemma}\label{l:matlis-pair}
$\hhom\bigl(\kk[\cQ_+]_\tau,(-)^\vee\bigr)$ is exact for all
faces~$\tau$ of any partially ordered~abel\-ian group~$\cQ$.  As a
result, $\hhom(\kk[\cQ_+]_\tau,-)$ is exact on the category of
$\cQ$-finite~\mbox{modules}.
\end{lemma}
\begin{proof}
Localization is exact and so is Matlis duality, so the first sentence
follows from Example~\ref{e:dual-of-localization}.  The consequence
comes from Lemma~\ref{l:vee-vee}: for $\cQ$-finite modules,
$\hhom(\kk[\cQ_+]_\tau,-)$ is the composite $(-)^\vee$ followed by
$\hhom\bigl(\kk[\cQ_+]_\tau,(-)^\vee\bigr)$.
\end{proof}

\begin{remark}\label{r:matlis-pair}
What really drives the lemma is the observation that while the
opposite notion to injective is projective (reverse all of the arrows
in the definition), the adjoint notion to injective is flat.  That is,
a module is flat if and only if its Matlis dual is injective.  This is
an instance of a rather general phenomenon that can be phrased in
terms of a monoidal abelian category~$\cC$ possessing a \emph{Matlis
object}~$E$ for a \emph{Matlis dual pair} of subcategories~$\cA$
and~$\cB$ such that $\Hom(-,E)$ restricts to exact contravariant
functors $\cA \to \cB$ and $\cB \to \cA$ that are inverse to one
another.  The idea is to set $\cM^\vee = \Hom(\cM,E)$, the
\emph{Matlis dual} of any object~$\cM$ of~$\cC$, and define an object
of~$\cC$ to be \emph{$\cB$-flat} if $F \otimes -$ is an exact functor
on~$\cB$.  Then an object $F$ of~$\cA$ is $\cB$-flat if and only if
$\Hom(F,-)$ is an exact functor on~$\cA$.  Examples of this situation
include artinian and noetherian modules over a complete local ring;
modules of finite length over any local ring (in both cases $E =
E(R/\mm)$ is the injective hull of the residue field); and of course
$\cQ$-finite modules over a partially ordered abelian group~$\cQ$.
The latter two examples feature a single Matlis self-dual subcategory.
\end{remark}

\begin{example}\label{e:not-exact-Q-infinite}
It is important to note that $\hhom(\kk[\cQ_+]_\tau,-)$ is not exact
on the category of all---that is, not necessarily
$\cQ$-finite---modules over a partially ordered abelian group~$\cQ$.
Indeed, let $F \onto \kk[\cQ]$ be any free cover of the localization
of~$\kk[\cQ_+]$ along the maximal face.  (When $\cQ_+ = \NN$, this
writes the module $\kk[\ZZ]$ of Laurent polynomials as a quotient of a
graded free module~$F$ over the ordinary polynomial ring $\kk[\NN]$ in
one variable.)  Then $\kk[\cQ] = \kk[\cQ_+]_\tau$ for $\tau = \cQ$
itself, and applying $\hhom(\kk[\cQ], -)$ to the surjection $F \onto
\kk[\cQ]$ yields the homomorphism $0 \to \kk[\cQ]$, which is not
surjective.
\end{example}

\section{Geometry of real staircases}\label{s:staircases}

The difference between ordinary noetherian commutative algebra and
algebra over real polyhedral groups begins with the local geometry of
downsets near their boundaries (Section~\ref{b:tangent}) and the
functorial view of this geometry (Section~\ref{b:upper-closure}).

\subsection{Tangent cones of downsets}\label{b:tangent}\mbox{}

\medskip
\noindent
Locally speaking, the geometry of a downset near a boundary point is
rather rigid, in the sense that only finitely many geometries are
possible (Proposition~\ref{p:shape}): the negative relative interior
of each face of the positive cone is either filled or does not appear
at all, and as soon as some face appears, all bigger faces are forced
to appear, as well.  Getting there requires some definitions and limit
computations.

\begin{defn}\label{d:tangent-cone}
The \emph{tangent cone} $T_\aa D$ of a downset $D$ in a real
polyhedral group~$\cQ$ (Definition~\ref{d:polyhedral}) at a point $\aa
\in \cQ$ is the set of vectors $\vv \in -\cQ_+$ such that $\aa +
\epsilon\vv \in D$ for all sufficiently small (hence all)~$\epsilon >
0$.
\end{defn}

\begin{example}\label{e:tangent-cones}
The three depicted points on the upper boundary of the downset
\begin{center}
\includegraphics[height=30mm]{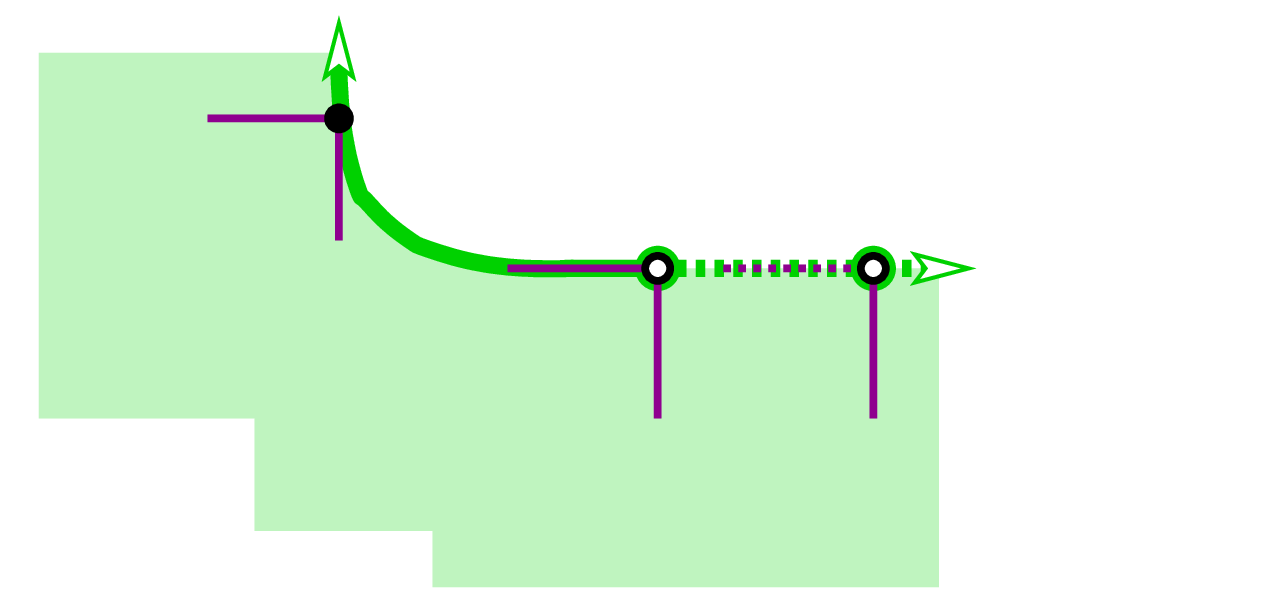}
\end{center}
have different tangent cones.  The leftmost point lies in the downset
itself, so its tangent cone equals the entire closed negative
\pagebreak[2]
quadrant.  The middle point does not itself lie in the downset, but
all points strictly to its left do, and all points strictly beneath it
do, so its tangent cone is the punctured negative quadrant, with only
the origin missing.  For the rightmost point, no points nearby and at
the same height lie in the downset, but all points with strictly lower
height lie in the downset, so the tangent cone is the negative
quadrant with its horizontal axis deleted.
\end{example}

\begin{remark}\label{r:tangent-cone}
Since the real number $\epsilon$ in Definition~\ref{d:tangent-cone} is
strictly positive, the vector $\vv = \0$ lies in $T_\aa D$ if and only
if $\aa$ itself lies in~$D$, and in that case $T_\aa D = -\cQ_+$.
\end{remark}

\begin{example}\label{e:tangent-cone}
The tangent cone defined here is not the tangent cone of~$D$ as a
stratified space, because the cone here only considers vectors
in~$-\cQ_+$.  A specific simple example to see the difference is the
closed half-plane $D$ beneath the line $y = -x$ in~$\RR^2$, where the
usual tangent cone at any point along the boundary line is the
half-plane, whereas $T_\aa D = -\RR^2_+$.  Furthermore, $T_\aa D$ can
be nonempty for a point~$\aa$ in the boundary of~$D$ even if $\aa$
does not lie in~$D$ itself.  For an example of that, take $D^\circ$ to
be the interior of~$D$; then $T_\aa D^\circ = -\RR^2_+ \minus \{\0\}$
for any $\aa$ on the boundary line.
\end{example}

The most important conclusion concerning tangent cones at points of
downsets, Proposition~\ref{p:shape}, says that such cones are certain
unions of relative interiors of faces.  Some definitions and
preliminary results are required.

\begin{defn}\label{d:cocomplex}
Fix a real polyhedral group $\cQ$.
\begin{enumerate}
\item%
For any face $\sigma$ of the positive cone~$\cQ_+$, write $\sigma^\circ$
for the relative interior of~$\sigma$.
\item%
For any set~$\nabla$ of faces of~$\cQ_+$, write $\cQ_\nabla =
\bigcup_{\sigma \in \nabla} \sigma^\circ$, the \emph{cone of
shape~$\nabla$}.
\item\label{i:cocomplex}%
A \emph{cocomplex} in~$\cQ_+$ is an upset in the face poset $\cfq$
of~$\cQ_+$, where $\sigma \preceq \tau$ if~$\sigma \subseteq \tau$.
\end{enumerate}
\end{defn}

\begin{example}\label{e:nabla}
The cocomplex $\nabs = \{\text{faces }\sigma'\text{ of }\cQ_+ \mid
\sigma' \supseteq \sigma\}$ is the \emph{open star} of the
face~$\sigma$.  It determines the cone $\qns$ of shape~$\nabs$, which
plays an important role.
\end{example}

\begin{remark}\label{r:quantum}
The next proposition is the reason for specializing this section to
real polyhedral groups instead of arbitrary polyhedral partially
ordered groups, where limits might not be meaningful.  For example,
although limits make formal sense in the integer lattice~$\ZZ^n$ with
the usual discrete topology, it is impossible for a sequence of points
in the relative interior of a face to converge to the origin of the
face.  This quantum separation has genuine finiteness consequences for
the algebra of $\ZZ^n$-modules that usually do not hold for
$\RR^n$-modules.
\end{remark}

\begin{prop}\label{p:<<}
If $\{\aa_k\}_{k\in\NN}$ is any sequence converging to~$\aa$, then	
$\bigcup_{k=0}^\infty (\aa_k - \cQ_+) \supseteq \aa - \qpc$.  If the
sequence is contained in $\aa - \sigma^\circ$, then the union equals
$\aa - \qns$.
\end{prop}
\begin{proof}
For each point $\bb \in \aa - \qpc$, every linear function $\ell:
\RR^n \to \RR$ that is nonnegative on~$\cQ_+$ eventually takes values
on the sequence that are bigger than~$\ell(\bb)$; thus $\bb$ lies in
the union as claimed.  When the sequence is contained in $\aa -
\sigma^\circ$, the union is contained in $\aa - \sigma^\circ - \cQ_+$
by hypothesis, but the union contains $\aa - \sigma^\circ$ by the
first claim applied with $\sigma$ in place of~$\cQ_+$.  The union
therefore equals $\aa - \sigma^\circ - \cQ_+$ because it is a downset.
The next lemma completes the proof.
\end{proof}

\begin{lemma}\label{l:<<}
If $\sigma$ is any face of the positive cone~$\cQ_+$ then $\sigma^\circ
+ \cQ_+ = \qns$.
\end{lemma}
\begin{proof}
Fix $\ff + \bb \in \sigma^\circ + \cQ_+$.  If $\ell(\ff + \bb) = 0$
for some linear function $\ell: \RR^n \to \RR$ that is nonnegative
on~$\cQ_+$, then $\ell(\ff) = 0$, too.  Therefore the support face
of~\mbox{$\ff + \bb$} (the smallest face in which it lies)
contains~$\sigma$.  On the other hand, suppose $\bb$ lies interior to
some face of~$\cQ_+$ that contains~$\sigma$.  Then pick any $\ff \in
\sigma^\circ$.  If~$\ell(\bb) = 0$ then also $\ell(\ff) = 0$, because
the support face of~$\bb$ contains~$\sigma$.  But if $\ell(\bb) > 0$,
then $\ell(\bb) > \ell(\epsilon \ff)$ for any sufficiently small
positive~$\epsilon$.  As~$\cQ_+$ is an intersection of only finitely many
closed half-spaces, a single $\epsilon$ works for all relevant~$\ell$, and
then $\bb = \epsilon \ff + (\bb - \epsilon \ff) \in \sigma^\circ + \cQ_+$.
\end{proof}

For $\cQ_+ = \RR_+^n$ the following is essentially
\cite[Lemma~5.1]{madden-mcguire2015}.

\begin{cor}\label{c:<<}
If $D \subseteq \cQ$ is a downset in a real polyhedral group, then
$\aa$ lies in the closure $\oD$ if and only if $D$ contains the
interior $\aa - \qpc$ of the negative cone with~origin~$\aa$.
\end{cor}

\begin{prop}\label{p:shape}
If $\aa \in \oD$ for a downset $D$, then $T_\aa D = -\cQ_\nabla$ is
the negative cone of shape~$\nabla$ for some cocomplex~$\nabla$
in~$\cQ_+$.  In this case $\nabla = \nda$ is the \emph{shape} of~$D$
at\/~$\aa$.
\end{prop}
\begin{proof}
The result is true when $n = 1$ because there are only three
possibilities for $a \in \RR$: either $a \in D$, in which case $T_a D
= -\RR_+ = \cQ_\nabla$ for $\nabla = \cfq$
(Remark~\ref{r:tangent-cone}); or $a \not\in D$ but $a$ lies in the
closure of~$D$, in which case $T_a D = \cQ_\nabla$ for $\nabla =
\{\qpc\} \subseteq \cfq$; or $a$ is separated from~$D$ by a nonzero
distance, in which case $T_a D = \cQ_{\{\}}$ is empty.

Write $D_\sigma$ for the intersection of~$D$ with the $\aa$-translate
of the linear span of~$\sigma$:
$$
  D_\sigma = D \cap (\aa + \RR \sigma).
$$
If $\sigma \subsetneqq \cQ_+$, then $T_\aa D_\sigma = \sigma_\nabla$
for some upset $\nabla \subseteq \cF_\sigma$ by induction on the
dimension of~$\sigma$.  In actuality, only the case $\dim \sigma =
n-1$ is needed, as the face posets $\cF_\sigma$ for $\dim \sigma =
n-1$ almost cover~$\cfq$: only the open maximal face $\qpc$ itself
lies outside of their union, and that case is dealt with by
Corollary~\ref{c:<<}.
\end{proof}

\begin{example}\label{e:shape}
For the three upper boundary points in Example~\ref{e:tangent-cones},
the shapes are
\begin{itemize}
\item%
left: the entire four-element poset of all faces of~$\RR^2_+$;
\item%
center: the cocomplex of faces that are not the origin~$\0$;
\item%
right: the cocomplex whose unique minimal face is the $y$-axis.
\end{itemize}
\end{example}

\subsection{Upper closure functors}\label{b:upper-closure}\mbox{}

\medskip
\noindent
Geometric closures append points along boundaries.  Functorial
closures append categorical limits in
Definition~\ref{d:upper-closure}.  Specific calculations for downset
modules in Proposition~\ref{p:downset-upper-closure} render make the
theory concrete and usable.

\begin{defn}\label{d:atop-sigma}
For a module~$\cM$ over a real polyhedral group~$\cQ$, a face~$\sigma$
of~$\cQ_+$, and a degree $\aa \in \cQ$, the \emph{upper closure
atop~$\sigma$} at~$\aa$ in~$\cM$ is the vector space
$$
  (\dsm)_\aa
  =
  \cM_{\aa-\sigma}
  =
  \dirlim_{\aa' \in \aa - \sigma^\circ} \cM_{\aa'}.
$$
\end{defn}

\begin{lemma}\label{l:exact-delta}
The functor $\cM \mapsto \dsm = \bigoplus_{\aa\in\cQ} (\dsm)_\aa$
is exact.
\end{lemma}
\begin{proof}
Direct limits are exact in categories of vector spaces (or modules
over rings).
\end{proof}

\begin{lemma}\label{l:natural}
The structure homomorphisms of~$\cM$ as a $\cQ$-module induce natural
homomorphisms $\cM_{\aa-\sigma} \to \cM_{\bb-\tau}$ for $\aa \preceq
\bb$ in~$\cQ$ and faces $\sigma \supseteq \tau$ of~$\cQ_+$.
\end{lemma}
\begin{proof}
The natural homomorphisms come from the universal property of
colimits.  First a natural homomorphism $\cM_{\aa-\sigma} \to
\cM_{\bb-\sigma}$ is induced by the composite homomorphisms $\cM_\cc
\to \cM_{\cc+\bb-\aa} \to \cM_{\bb - \sigma}$ for $\cc \in \aa -
\sigma^\circ$ because adding $\bb - \aa$ takes $\aa - \sigma^\circ$ to
\mbox{$\bb - \sigma^\circ$}.  For $\cM_{\bb-\sigma} \to
\cM_{\bb-\tau}$ the argument is similar, except that existence of
natural homomorphisms $\cM_\cc \to \cM_{\bb-\tau}$ for $\cc \in \bb -
\sigma^\circ$ requires Proposition~\ref{p:<<} and Lemma~\ref{l:<<}.
\end{proof}

\begin{remark}\label{r:semilattice=monoid}
The face poset~$\cfq$ of the positive cone~$\cQ_+$ can be made into a
commutative monoid in which faces $\sigma$ and~$\tau$ of~$\cQ_+$ have
sum
$$
  \sigma + \tau = \sigma \cap \tau.
$$
Indeed, these monoid axioms use only that $\bigl(\cfq,\cap\bigr)$ is a
bounded meet semilattice, the monoid unit element being the maximal
semilattice element---in this case, $\cQ_+$ itself.  When $\cfq$ is
considered as a monoid in this way, the partial order on it has
$\sigma \preceq \tau$ if $\sigma \supseteq \tau$, which is the
opposite of the default partial order on the faces of a polyhedral
cone.  For utmost clarity, and because both monoid partial orders are
relevant (see Remark~\ref{r:semilattice=monoid'}), $\fqo$ is written
when this monoid partial order is intended.
\end{remark}

\begin{defn}\label{d:upper-closure}
Fix a module~$\cM$ over a real polyhedral group~$\cQ$ and a degree
$\aa \in \cQ$.  The \emph{upper closure functor} takes~$\cM$ to the
$\cQ \times \fqo$-module $\delta\cM$ whose fiber over $\aa \in \cQ$ is
the $\fqo$-module
$$
  (\delta\cM)_\aa
  =
  \bigoplus_{\sigma \in \cfq} \cM_{\aa-\sigma}
  =
  \bigoplus_{\sigma \in \cfq} (\dsm)_\aa.
$$
The fiber of $\delta\cM$ over $\sigma \in \fqo$ is the \emph{upper
closure~$\dsm\!$ of~$\cM$ atop~$\sigma$}.
\end{defn}

\begin{example}\label{e:closure-module}
The upper closure of the module~$M$ in Examples~\ref{e:trouble}
and~\ref{e:trouble-soc} with triangular degree set has vector spaces
of dimension either~$0$ or~$1$ in every graded component indexed
by~\mbox{$\RR^2_+ \times \fro$}.  For each $\aa \in \RR^2$, depict the
corresponding $\fro$-module using a solid dot for a vector space of
dimension~$1$ and no solid dot for a vector space of dimension~$0$.
Then $\delta\cM$ is drawn at left and $\fro$ is drawn at right:%
\vspace{-1ex}%
$$
  \psfrag{0}{\footnotesize$\0$}
  \psfrag{x}{\footnotesize$x$}
  \psfrag{y}{\footnotesize$y$}
  \psfrag{1}{\footnotesize$\RR^2_+$}
  \psfrag{F}{$\fro$}
  \includegraphics[height=30ex]{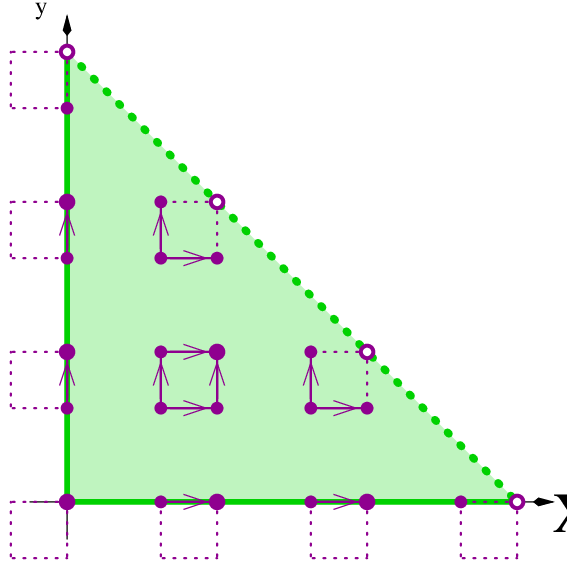}
  \qquad\qquad\qquad
  \psfrag{x}{\footnotesize$y$}
  \psfrag{y}{\footnotesize$x$}
  \includegraphics[height=30ex]{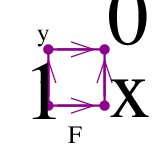}
\vspace{-1ex}%
$$
The $\fro$-module at~$\aa$ is drawn with $\aa$ at the upper-right corner,
to convey the idea that one should stand at~$\aa$ and see what direct
limits result as~$\aa$ is approached from below along the various faces.
When $M_\aa = 0$, the point~$\aa$ is drawn as an empty dot.
\end{example}

\begin{remark}\label{r:ephemeral}
Upper boundaries contain the later%
	\footnote{Upper closure functors were introduced in
	\cite{qr-codes}, where they were called upper boundary
	functors.}
notion of ephemeral modules \cite{berkouk-petit2019}:
an $\RR^n$ module~$\cM$ is ephemeral if its upper closure $\ds\cM$
atop the interior~$\sigma$ of~$\RR^n_+$ vanishes.  This notion is key
to the difference between poset module theory
\cite{hom-alg-poset-mods} and the formulation of persistent homology
via constructible sheaves \cite{kashiwara-schapira2018}, as detailed
in \cite{strat-conical}.
\end{remark}

\begin{remark}\label{r:upper-frontier}
The face of~$\cQ_+$ that contains only the origin~$\0$ is an absorbing
element: it acts like infinity, in the sense that $\sigma + \{\0\} =
\{\0\}$ in the monoid~$\fqo$ for all faces~$\sigma$.  Adding the
absorbing element~$\0$ in the $\fqo$ component therefore induces a
natural\/ $\cQ \times \fqo$-module projection from the upper closure
$\delta\cM$ to~$\cM$.  At a degree $\aa \in \cQ$, this projection is
$\cM_{\aa-\sigma} \to \cM_{\aa - \0} = \cM_\aa$.  Interestingly, the
\emph{frontier} of a downset~$D$---those points in the topological
closure but outside of~$D$---is the set of nonzero degrees of a
functor, namely $\ker(\dsm \to \cM)$ for $\sigma = \qpc$.  The proof
is by Corollary~\ref{c:<<}.
\end{remark}

There is no natural map $\cM \to \dsm$ when $\sigma \neq \{\0\}$ has
positive dimension, because an element of degree~$\aa$ in~$\cM$ comes
from elements of~$\dsm$ in degrees less than~$\aa$.  However, that
leaves a way for Lemma~\ref{l:natural} to afford a notion of
divisibility of upper closure elements by elements of~$\cM$.

\begin{defn}\label{d:divides}
An element $y \in \cM_\bb$ \emph{divides} $x \in (\dsm)_\aa$ if $\bb
\in \aa - \qns = \aa - \sigma^\circ - \cQ_+\!$ (Lemma~\ref{l:<<}) and
$y \mapsto x$ under the natural map $\cM_\bb \to \cM_{\aa-\sigma}$
(Lemma~\ref{l:natural}).  The element $y$ is said to
\emph{$\sigma$-divide}~$x$ if, more restrictively, $\bb \in \aa -
\sigma^\circ$.
\end{defn}

Now come the fundamental calculations of upper closure functors, in
the next lemma and proposition, that drive all of the results in the
rest of the paper.

\begin{lemma}\label{l:downset-upper-closure}
If $\sigma \in \cfq$ and $D$ is a downset in~$\cQ$ then $\ds\,\kk[D] =
\kk[\dsd]$,~where
$$
  \dsd
  =
  \bigcup_{\xx\in\cQ} \ol{D \cap (\xx + \RR\sigma)}
  =
  D \cup \bigcup_{\xx\in\del D} \ol{D \cap (\xx + \RR\sigma)}
$$
It suffices to take the middle union over $\xx$ in any subspace
complement to~$\RR\sigma$.
\end{lemma}
\begin{proof}
For the second displayed equality, observe that the middle union
contains the right-hand union because the middle one contains~$D$.
For the other containment, if $\xx + \RR\sigma$ contains no boundary
point of~$D$, then $D \cap (\xx + \RR\sigma) = \oD \cap (\xx +
\RR\sigma)$ is already closed, so the contribution of $D \cap (\xx +
\RR\sigma)$ to the middle union is contained in~$D$.

For the other equality, $\ds\kk[D]$ is nonzero in degree $\aa$ if and
only if $\aa - \sigma^\circ \subseteq D$.  That condition is
equivalent to $\aa - \sigma^\circ \subseteq D \cap (\aa + \RR\sigma)$
because $\aa - \sigma^\circ \subseteq \aa + \RR\sigma$.  Translating
$D \cap (\aa + \RR\sigma)$ back by~$\aa$ yields a downset in the real
polyhedral group~$\RR\sigma$, with $(\RR\sigma)_+ = \sigma$, thereby
reducing to Corollary~\ref{c:<<}.

The final sentence follows because $\xx + \RR\sigma = \xx' +
\RR\sigma$ when $\xx - \xx' \in \RR\sigma$.
\end{proof}

\begin{prop}\label{p:downset-upper-closure}
If $D$ is a downset in a real polyhedral group~$Q$ and $\sigma \in
\cfq$, then $\ds\,\kk[D] = \kk[\dsd]$ is the indicator quotient for a
downset $\dsd$ that satisfies
\begin{enumerate}
\item\label{i:containment}%
$D \subseteq \dsd \subseteq \oD$ and
\item\label{i:nabla}%
$\dsd = \{\aa\in \oD \mid \sigma \in \nda\}$.
\end{enumerate}
\end{prop}
\begin{proof}
Item~\ref{i:containment} follows from item~\ref{i:nabla}.  What
remains to show is that $\dsd$ is a downset in~$\oD$ characterized by
item~\ref{i:nabla} and that it is semialgebraic if $D$~is.

First, $\sigma \in \nda$ means that $\aa - \sigma^\circ \subseteq D$,
which immediately implies that $\aa \in \dsd$ by
Lemma~\ref{l:downset-upper-closure}.  Conversely, suppose $\aa \in
\dsd$.  Lemma~\ref{l:downset-upper-closure} and
Corollary~\ref{c:<<}, the latter applied to the downset $-\aa + D \cap
(\aa + \RR\sigma)$ in~$\RR\sigma$, imply that $\aa - \sigma^\circ
\subseteq D$, and hence $\sigma \in \nda$ by definition, proving
item~\ref{i:nabla}.  Given that $\aa - \sigma^\circ \subseteq D$,
Proposition~\ref{p:<<} yields $D \cap (\aa - \cQ_+) \supseteq \aa -
\qns$.  Consequently, if $\bb \in \aa - \cQ_+$ then $D \cap (\bb -
\sigma) \supseteq \bb - \sigma^\circ$, whence $\bb \in \dsd$.  Thus
$\dsd$ is a downset.
\end{proof}

\section{Socles and cogenerators}\label{s:socle}

The main contribution of this section is Definition~\ref{d:soct},
which introduces the notions of cogenerator functor and socle along a
face with a given nadir.  Its concomitant foundations include ways to
decompose the cogenerator functors into continuous and discrete parts
(Proposition~\ref{p:either-order}), interactions with localization
(Proposition~\ref{p:local-vs-global}), and left-exactness
(Proposition~\ref{p:left-exact-tau}), along with the crucial
calculation of socles in the simplest case, namely the indicator
module of a single face (Example~\ref{e:soct-k[tau]}).  The theory is
built step by step, starting with ordinary (closed) socles over
arbitrary posets in Section~\ref{b:socc} and proceeding through
Section~\ref{b:soc-along} to cogenerator functors and socles with
increasing levels of complexity and (necessarily) decreasing freedom
regarding~the~poset.

\subsection{Closed socles and closed cogenerator functors}\label{b:socc}\mbox{}

\medskip
\noindent
In commutative algebra, the socle of a module over a local ring is the
set of elements annihilated by the maximal ideal.  The socle is a
vector space over the residue field~$\kk$ that can alternately be
characterized by taking homomorphisms from~$\kk$ into the module.
Either characterization works for modules over partially ordered
abelian groups, but only the latter generalizes readily to modules
over arbitrary posets.  Note that socles over face posets of polyhedra
occur naturally in the theory over real polyhedral~groups.

\begin{defn}\label{d:socc}
Fix an arbitrary poset~$\cP$.  The \emph{skyscraper} $\cP$-module
$\kk_p$ at $p \in \cP$ has $\kk$ in degree~$p$ and $0$ in all other
degrees.  The \emph{closed cogenerator functor} $\hhom_\cP(\kk,-)$ takes
each $\cP$-module $\cM$ to its \emph{closed socle}: the $\cP$-submodule
$$
  \socc\cM = \hhom_\cP(\kk,\cM) = \bigoplus_{p\in\cP} \Hom_\cP(\kk_p,\cM).
$$
When it is important to specify the poset, the notation $\cpsoc$ is
used instead of~$\socc\!$.  A \emph{closed cogenerator} of
\emph{degree}~$p \in P$ is a nonzero element in $(\socc\cM)_p$.
\end{defn}

\begin{remark}\label{r:overbar}
The bar over ``$\soc$'' is meant to evoke the notion of closure or
``closed''.  The bar under $\Hom$ is the usual one in multigraded
commutative algebra for the direct sum of homogeneous homomorphisms of
all degrees (see \cite[Section~I.2]{GWii} or
\cite[Definition~11.14]{cca}, for example).
\end{remark}

\begin{example}\label{e:socc}
The closed socle of~$\cM$ consists of the elements that are
annihilated by moving up in any direction, or that have maximal
degree.  In particular, the interval module $\kk[I]$ for any interval
$I \subseteq \cP$ (Definition~\ref{d:indicator}.\ref{i:interval}) has
closed socle
$$
  \socc \kk[I] = \kk[\max I],
$$
the interval module with degree set the elements of~$I$ that are
maximal in~$I$.
\end{example}

\begin{lemma}\label{l:left-exact-socc}
The closed cogenerator functor over any poset is left-exact.
\end{lemma}
\begin{proof}
It suffices to observe that the category of $\cP$-modules is abelian,
so $\Hom$ from a fixed source is automatically left-exact.  There are
a number of well known reasons that the category is abelian.  For
example, the definition of $\cP$-module is equivalent to that of a
module over the path algebra of (the Hasse diagram of)~$\cP$ with
relations to impose commutativity, namely equality of the morphism
induced by $p < p''$ and the composite morphism induced by $p < p' <
p''$.  Alternatively, a~$\cP$-module is the same as a sheaf on~$\cP$
in which the topology comprises the upsets of~$\cP$.  (See
\cite{yuzvinsky1981} as well as
\cite[\S4.2]{curry-thesis} and \cite{curry2019} for discussions of
these and references.)
\end{proof}

\subsection{Socles and cogenerator functors}\label{b:soc}\mbox{}

\medskip
\noindent
Socle elements are obtained by pushing arbitrary module elements up in
the poset as far as possible, so that pushing farther in any direction
annihilates the element.  In a continuous poset, elements might remain
nonzero in all neighborhoods of a boundary degree but vanish at the
degree itself.  When in this manner there is no actual module element
at a boundary point, the upper closure functors from
Section~\ref{b:upper-closure} manufacture all possible candidates from
which to select socle elements.  Definitions, algebraic properties,
and computations for downset modules occupy this section.

\begin{defn}\label{d:soc}
The \emph{cogenerator functor} takes a module over a real polyhedral
group to its \emph{socle}:
$$
  \soco\cM = \socc\delta\cM,
$$
which is the closed socle, computed over the poset $\cQ \times \fqo$
(see Remark~\ref{r:semilattice=monoid}) of the upper closure
module~of~$\cM$ (Definition~\ref{d:upper-closure}).
\end{defn}

\begin{example}\label{e:soc}
For the module~$\cM$ in Example~\ref{e:closure-module}, the socle
$\socc\cM$ is computable using Example~\ref{e:socc}, because the upper
closure $\delta\cM$ is an interval module over $\RR^2_+ \times \fro$.  Let $L$ be
the reltively open hypotenuse in Example~\ref{e:closure-module}, with
endpoints $X$ and~$Y$ along the $x$- and~$y$-axes, respectively.  The
degree set of the closed socle of~$\delta\cM$ is the~union~of
\begin{itemize}
\item%
$L \times \{x\}$
\item%
$L \times \{y\}$
\item%
$X \times \{x\}$
\item%
$Y \times \{y\}$.
\end{itemize}
Standing at any other point in $\RR^2_+ \times \fro$ where the module $\delta\cM$
does not already vanish, moving up in the $\RR^2$ direction keeps $\delta\cM$
nonzero if the $\RR^2$ coordinate lies strictly beneath~$L$, and moving
up in the $\fro$ direction keeps $\delta\cM$ nonzero if the $\fro$
coordinate is the lower-left $\RR^2_+$ corner.
\end{example}

\begin{remark}\label{r:notation-soc-bar}
Notationally, the lack of a bar over ``$\soco$'' serves as a visual
cue that the functor is over a real polyhedral group, as the upper
closure $\delta$ is not defined in more generality.  This visual cue
persists throughout the more general notions of socle: bar means
arbitrary poset, and no bar means real polyhedral group.  The
subscript~``$\0$'' refers to the minimal face of the positive cone of
the real polyhedral group; to be precise, it is the $\tau = \0$
special case of Definition~\ref{d:soct}.
\end{remark}

\begin{lemma}\label{l:soco-left-exact}
The cogenerator functor $\cM \mapsto \soco\cM$ is left-exact.
\end{lemma}
\begin{proof}
Use exactness of upper boundaries atop~$\sigma$
(Lemma~\ref{l:exact-delta}), exactness of the direct sums forming
$\delta\cM$ from $\dsm$, and left-exactness of closed socles
(Lemma~\ref{l:left-exact-socc}).
\end{proof}
  
Sometimes is it useful to apply the closed socle functor
to~$\delta\cM$ over $\cQ \times \fqo$ in two steps, first over one
poset and then over the other.  These yield the same result.

\begin{lemma}\label{l:either-order}
The functors $\rnsoc\!$ and $\fqsoc$ commute.  In particular,
$$
  \fqsoc(\rnsoc\delta\cM)
  \cong
  \soco\cM\!
  \cong
  \rnsoc(\fqsoc\delta\cM).
$$
\end{lemma}
\begin{proof}
By taking direct sums over $\aa$ and~$\sigma$, this follows from the
natural isomorphisms
$
  \Hom_{\fqo}\bigl(\kk_\sigma,\Hom_{\cQ}(\kk_\aa,-)\bigr)
  \cong
  \Hom_{\cQ\times\fqo}(\kk_{\aa,\sigma},-)
  \cong
  \Hom_\cQ\bigl(\kk_\aa,\Hom_{\fqo}(\kk_\sigma,-)\bigr).
$
\end{proof}

The fundamental examples---indicator quotients for downsets---require
a notation.

\begin{defn}\label{d:del-nabla}
In the situation of Definition~\ref{d:cocomplex}, write $\del\nabla$
for the antichain of faces of~$\cQ_+$ that are minimal under inclusion
in~$\nabla$.
\end{defn}

\begin{example}\label{e:del-nabla}
For the module~$\cM$ in Example~\ref{e:trouble-soc}, the shape
of~$\cM$ (Proposition~\ref{p:shape}) at any point along the relative
interior of the hypotenuse is the cocomplex
(Definition~\ref{d:cocomplex}.\ref{i:cocomplex})~$\nabla$ comprising
the faces of $\RR^2_+$ labeled as $x$, $y$, and $\RR^2_+$ itself, so
$\del\nabla = \{x, y\}$.
\end{example}

\begin{remark}\label{r:del-nda}
The reason to write $\del\nabla$ instead of $\max$ or~$\min$ is that
it would be ambiguous either way, as both $\cfq$ and~$\fqo$ are
natural here.  Taking the ``op'' perspective
(Remark~\ref{r:semilattice=monoid}), the $\fqo$-module
$\kk[\del\nabla]$ with basis~$\del\nabla$ resulting from
Definition~\ref{d:del-nabla} is really just a $\fqo$-graded vector
space: the antichain condition ensures that every element of~$\fqo$
not acting as the identity on a summand of~$\kk[\del\nabla]$ acts
by~$0$.  More precisely, if the summand is indexed by a face~$\tau$,
then $\sigma \in \fqo$ acts as the identity on the summand if $\sigma
\supseteq \tau$ (because then $\sigma + \tau = \tau$ as in
Remark~\ref{r:semilattice=monoid}) and $\sigma$ acts by~$0$ otherwise.
In particular, if $\del\nabla = \{\0\}$, then all of $\fqo$ acts
by~$1$~on~$\kk[\del\nabla]$.
\end{remark}

The case of most interest here is $\nabla = \nda$, the shape of~$D$
at~$\aa$ (Proposition~\ref{p:shape}).

\begin{example}\label{e:soc-Rn-downset}
For a downset $D$ in a real polyhedral group, $\fqsoc\delta\kk[D]$ has
$\kk[\del\nda]$ in each degree~$\aa$, because $\delta\kk[D]$ itself
has $\kk[\nda]$ in each degree~$\aa$ by
Definition~\ref{d:upper-closure} and
Proposition~\ref{p:downset-upper-closure}.  Note that $\delta\kk[D]$
and $\fqsoc\delta\kk[D]$ are direct sums over faces~$\sigma$, since
they are $\fqo$-modules.  What $\rnsoc$ then does, for each~$\sigma$,
is find the degrees~$\aa \in \cQ$ maximal among those where $\sigma
\in \del\nda$, by
Proposition~\ref{p:downset-upper-closure}.\ref{i:nabla} and
Example~\ref{e:socc}.

Taking socles in the other order, first $\rnsoc\kk[\dsd]$ asks
whether $\sigma \in \nda$ but $\sigma \not\in \nd[\bb]$ for any $\bb
\succ \aa$ in~$\cQ$.  That can happen even if $\sigma$ contains a
smaller face where it still happens.  What $\fqsoc$ then does is
return the smallest faces at~$\aa$ where it happens.
\end{example}

\begin{cor}\label{c:soc-Rn-downset}
The socle of the indicator quotient $\kk[D]$ for any downset $D$ in a
real polyhedral group~$\cQ$ is nonzero only in degrees lying in the
topological boundary~$\del D$.
\end{cor}
\begin{proof}
By Proposition~\ref{p:downset-upper-closure}.\ref{i:containment},
$\delta\kk[D]$ is a direct sum of indicator quotients.
Example~\ref{e:socc} and Proposition~\ref{p:shape} show that the socle
of an indicator quotient over a real polyhedral group lies along the
boundary of the downset in question.
\end{proof}

\subsection{Closed socles along faces of positive dimension}\label{b:socc+}\mbox{}

\medskip
\noindent
Every element should divide a socle element.  But some elements
persist indefinitely in some directions.  In some modules---coprimary
ones (Definition~\ref{d:coprimary}), for example---every element
persists indefinitely along a face of the positive cone.  This
behavior is not rare; it occurs in all modules of positive Krull
dimension.  Such cases require socles along faces, in contrast to the
socles in Sections~\ref{b:socc+} and~\ref{b:soc}, which occur at
isolated points.  Locating closed socles along faces
(Definition~\ref{d:socct}) still uses $\Hom$ functors, in analogy with
Definition~\ref{d:socc}, but since they occur along faces of~$\cQ$,
the $\Hom$ functors try to insert indicator modules~$\kk[\tau]$ of
positive-dimensional faces (Definition~\ref{d:hhom}) instead of
individual points.
Localization and quotient-restriction the reduce back to the
zero-dimensional setting.  Much of this section concerns the order in
which these various $\Hom$, localization, and quotient-restriction
operations are applied.

\begin{defn}\label{d:hhom}
Fix a face~$\tau$ of a partially ordered abelian group~$\cQ$.  The
\emph{skyscraper} $\cQ$-module at $\aa \in \cQ$ along~$\tau$ is
$\kk[\aa+\tau]$, the subquotient $\kk[\aa + \cQ_+]/\kk[\aa +
\mm_\tau]$ of~$\kk[\cQ]$, where $\mm_\tau = \cQ_+\! \minus \tau$.  Set
$$
  \hhom_\cQ\bigl(\kk[\tau],-\bigr)
  =
  \bigoplus_{\aa \in \cQ} \Hom_\cQ\bigl(\kk[\aa+\tau],-\bigr).
$$
\end{defn}

\begin{defn}\label{d:socct}
Fix a partially ordered abelian group~$\cQ$, a face~$\tau$, and a
$\cQ$-module~$\cM$.
\begin{enumerate}
\item\label{i:global-socc-tau}%
The \emph{global closed cogenerator functor along~$\tau$} takes $\cM$
to its \emph{global closed socle along~$\tau$}: the
$\kk[\qztp]$-module
$$
  \socct\cM = \hhom_\cQ\bigl(\kk[\tau],\cM\bigr)/\tau,
$$
where the quotient modulo~$\tau$ is quotient-restriction as in
Definition~\ref{d:quotient-restriction}.

\item\label{i:local-socc-tau}%
If $\qzt$ is partially ordered, then the \emph{local closed
cogenerator functor along~$\tau$} takes $\cM$ to its \emph{local
closed socle along~$\tau$}: the $\qzt$-module
$$
  \socc(\cmt) = \hhom_{\qzt}(\kk,\cM/\tau).
$$
Elements of $\socc(\cmt)$ are identified with elements of~$\cmt$ via
$\phi \mapsto \phi(1)$.

\item\label{i:global-closed-cogen}%
Regard the $\cQ$-module $\hhom_\cQ\bigl(\kk[\tau],\cM\bigr)$ naturally
as contained in~$\cM$ via \mbox{$\phi \mapsto \phi(1)$}.  A
homogeneous element in this $\cQ$-submodule that maps to a nonzero
element of $\socct\cM$ is a \emph{global closed cogenerator} of~$\cM$
along~$\tau$.  For an inteval $I \subseteq \cQ$, a \emph{global closed
cogenerator} of~$I$ is the degree in~$\cQ$ of a global closed
\mbox{cogenerator}~of~$\kk[I]$.

\item\label{i:local-closed-cogen}%
Regard $\socc(\cmt)$ as naturally contained in~$\cmt$ via $\phi
\mapsto \phi(1)$.  A nonzero homogeneous element in $\socc(\cmt)$ is a
\emph{local closed cogenerator} of~$\cM$ along~$\tau$.
\end{enumerate}
Assume the default modifier ``global'' when neither ``local'' nor
``global'' is written.
\end{defn}

\begin{example}\label{e:global-socc-tau}
As noted in Example~\ref{e:soc-min-primary}, the global closed socle
of the module
$\kk[\raisebox{-.6ex}[0pt][0pt]{\includegraphics[height=2.8ex]{decomp}}]$
at the left-hand end of Example~\ref{e:min-primary} along~$\tau = $ the
positive $x$-axis (Definition~\ref{d:socct}.\ref{i:global-socc-tau})
has an element~$\wt z$ represented by the horizontal ray in the upper
boundary.  In detail, any point~$\aa$ along the ray yields an injection
$\kk[x\text{-axis}_+] \into
\kk[\raisebox{-.6ex}[0pt][0pt]{\includegraphics[height=2.8ex]{decomp}}]$
generated in degree~$\aa$ because the ray extends infinitely far to the
right but is annihilated upon moving upward in any direction.  This
injection constitutes a global closed cogenerator
(Definition~\ref{d:socct}.\ref{i:global-closed-cogen})~$z$
that maps to~$\wt z$ in the quotient-restriction modulo~$\tau$
(Definition~\ref{d:quotient-restriction}), and the degree~$\aa$ is a
global closed cogenerator of the
interval~$\raisebox{-.6ex}[0pt][0pt]{\includegraphics[height=2.8ex]{decomp}}$.
In contrast, the global closed socle
of~$\kk[\raisebox{-.6ex}[0pt][0pt]{\includegraphics[height=2.8ex]{decomp}}]$
along~$\tau = $ the positive $y$-axis vanishes because the module has
no vertical rays that are annihilated upon moving to the right.
\end{example}

\begin{example}\label{e:local-socc-tau}
Localizing the module
$\kk[\raisebox{-.6ex}[0pt][0pt]{\includegraphics[height=2.8ex]{decomp}}]$
from Example~\ref{e:global-socc-tau} (originally from
Example~\ref{e:min-primary}) along the $x$-axis yields the right-hand
module from~Example~\ref{e:min-primary}, which is the interval module
for a closed lower half-plane.  The quotient-restriction along the
$x$-axis is therefore the interval module for a closed
downward-pointing ray, which has closed socle~$\kk$ residing at the
closed endpoint.  Remarkably, localizing
$\kk[\raisebox{-.6ex}[0pt][0pt]{\includegraphics[height=2.8ex]{decomp}}]$
along the $y$-axis yields the interval module for a \emph{closed} left
half-plane, so the quotient-restriction
$\kk[\raisebox{-.6ex}[0pt][0pt]{\includegraphics[height=2.8ex]{decomp}}]/y\text{-axis}_+$
along the $y$-axis is the interval module for a closed left-pointing
ray.  Therefore
$\kk[\raisebox{-.6ex}[0pt][0pt]{\includegraphics[height=2.8ex]{decomp}}]$
has local closed socle
(Definition~\ref{d:socct}.\ref{i:local-socc-tau}) isomorphic to~$\kk$
residing at the closed right endpoint of the ray.  Compare the
conclusion of Example~\ref{e:global-socc-tau}.
\end{example}

\begin{remark}\label{r:notation-soc-tau}
Notationally, a subscript on ``$\soc$'' serves as a visual cue that
the functor is over a partially ordered abelian group, as faces of
posets are not defined in more generality.  This visual cue persists
throughout the more general notions of~socle.
\end{remark}

\begin{remark}\label{r:socc0}
The closed cogenerator functor over a partially ordered abelian group
is the global closed cogenerator functor along the trivial face:
$\socc = \socc[\{\0\}]$ and it equals the local cogenerator functor
along~$\{\0\}$.
\end{remark}

\begin{remark}\label{r:witness}
In looser language, a closed cogenerator of~$\cM$ along~$\tau$ is an
element
\begin{itemize}
\item%
annihilated by moving up in any direction outside of~$\tau$ but that
\item%
remains nonzero when pushed up arbitrarily along~$\tau$.
\end{itemize}
Equivalently, a closed cogenerator along~$\tau$ is an element whose
annihilator under the action of~$\cQ_+$ on~$\cM$ equals the prime
ideal $\mm_\tau = \cQ_+ \minus \tau$ of the positive cone~$\cQ_+$.
Elements like this are sometimes known as ``witnesses'' in commutative
algebra.
\end{remark}

\begin{remark}\label{r:automatic-partial-order-qzt}
The next Example, like all of the results between here and
Section~\ref{b:soc-along}, refers to a face~$\tau$ such that $\qzt$
is partially ordered.  The fact that $\qzt$ is not automatically
partially ordered is discussed in Remark~\ref{r:qrt}, which notes that
in many cases of interest it is, with the positive cone for~$Q_+$
being the image of~$Q_+$ in~$\qzt$.
\end{remark}

\begin{example}\label{e:socct}
The closed socle along a face~$\tau$ of the indicator quotient
$\kk[I]$ for any interval~$I$ in a partially ordered abelian
group~$\cQ$ with partially ordered quotient~$\qzt$~is
$$
  \socct\, \kk[I] = \kk[\mti],
$$
where $\mti$ is the image in~$\qzt$ of the set of closed cogenerators
of~$I$ along~$\tau$:
$$
  \mti =
  \big\{\aa \in I \mid (\aa+\cQ_+) \cap I = \aa + \tau\big\}/\,\ZZ\tau.
$$
The set of closed cogenerators of a downset~$D$ along~$\tau$ can also
be characterized as the elements of~$D$ that become maximal in the
localization~$D_\tau$ of~$D$ (Definition~\ref{d:localize-downset}).
\end{example}

Every global closed cogenerator yields a local one.

\begin{prop}\label{p:local-vs-global-closed}
Fix a partially ordered abelian group~$\cQ$.  There is a natural
injection
$$
  \socct\cM \into \socc(\cmt)
$$
for any $\cQ$-module~$\cM$ if $\tau$ is a face with partially ordered
quotient~$\qzt$.
\end{prop}
\begin{proof}
Localizing any homomorphism $\kk[\aa+\tau] \to \cM$ along~$\tau$
yields a homomorphism $\kk[\aa+\ZZ\tau] \to \cM_\tau$, so
$\hhom_\cQ\bigl(\kk[\tau],\cM\bigr){}_\tau$ is naturally a submodule
of $\hhom_\cQ\bigl(\kk[\ZZ\tau],\cM_\tau\bigr)$.  The claim now
follows from Lemma~\ref{l:exact-qr} and the next result.
\end{proof}

\begin{lemma}\label{l:socct}
If $\cQ$ and $\qzt$ are partially ordered, there is a canonical
\mbox{isomorphism}
$$
  \hhom_\cQ\bigl(\kk[\ZZ\tau],\cM_\tau\bigr)/\tau
  \cong
  \hhom_{\qzt}(\kk,\cM/\tau).
$$
\end{lemma}
\begin{proof}
Follows from the definitions, using that $\kk[\ZZ\tau]/\tau = \kk$ in
$(\qzt)$-degree~$\0$.
\end{proof}

The following crucial remark highlights the difference between
real-graded algebra and integer-graded algebra.  It is the source of
much of the subtlety in the theory developed in this paper,
particularly Sections~\ref{s:socle}--\ref{s:hulls}.

\begin{remark}\label{r:soc-vs-supp}
In contrast with taking support on a face
\cite[Proposition~4.6]{prim-decomp-pogroup} and also with socles in
commutative algebra over noetherian local or graded rings,
localization need not commute with taking closed socles along faces of
positive dimension in real polyhedral groups.  In other words, the
injection in Proposition~\ref{p:local-vs-global-closed} need not be
surjective: there can be local closed cogenerators that do not lift to
global ones.  The problem comes down to the homogeneous prime ideals
of the monoid algebra $\kk[\cQ_+]$ not being finitely generated, so
the quotient $\kk[\tau]$ fails to be finitely presented; it means that
$\Hom_{\kk[\cQ_+]}\bigl(\kk[\tau],-\bigr)$ need not commute with $A
\otimes_{\kk[\cQ_+]}-$, even when $A$ is a flat $\kk[\cQ_+]$-algebra
such as a localization of~$\kk[\cQ_+]$.  The context of
$\RR^n$-modules complicates the relation between support on~$\tau$ and
closed cogenerators along~$\tau$ because the ``thickness'' of the
support can approach~$0$ without ever quantum jumping all the way
there and, importantly, remaining there along an entire translate
of~$\tau$, as it would be forced to for a discrete group like~$\ZZ^n$.
See Example~\ref{e:global-support}, for instance, where the support on
the $y$-axis does not contribute any closed socle along the $y$-axis
to the ambient module.  Or see another manifestation in
Examples~\ref{e:global-socc-tau} and~\ref{e:local-socc-tau}, where the
global and local socles differ.  This issue is independent of the
density phenomenon explored in Section~\ref{s:minimality}; indeed, the
interval in Example~\ref{e:global-support} is closed, so its socle
equals~its~closed~socle~and~is~closed.
\end{remark}

\begin{prop}\label{p:left-exact-tau-closed}
The global closed cogenerator
$\socct\!$ along any face~$\tau$ of a partially ordered abelian group
is left-exact, as is the local version if $\qzt$ is partially ordered.
\end{prop}
\begin{proof}
For the global case, $\hhom_Q(\kk[\tau],-)$ is exact because it occurs
in the category of graded modules over the monoid
algebra~$\kk[\cQ_+]$, and quotient-restriction is exact by
Lemma~\ref{l:exact-qr}.  For the local case, use exactness of $\cM
\mapsto \cmt$ again (Lemma~\ref{l:exact-qr}) and left-exactness of
closed socles (Lemma~\ref{l:left-exact-socc}), the latter applied
over~$\qzt$.
\end{proof}

\begin{remark}\label{r:along}
Closed socles, without reference to faces, work over arbitrary posets
and are actually used that way in this work (over $\fqo$, for
instance, in Section~\ref{b:soc}).  That explains why
Section~\ref{b:socc} is needed in the generality presented there.
\end{remark}

\subsection{Socles along faces of positive dimension}\label{b:soc-along}\mbox{}

\medskip
\noindent
Sections~\ref{b:soc} and~\ref{b:socc+} generalize Section~\ref{b:socc}
in two independent ways: with non-closed cogenerators and with closed
cogenerators of positive dimemsion.  What remains is to take the join
of these by defining non-closed cogenerators along faces of positive
dimension (Definition~\ref{d:soct}).  This requires upper closure
functors along faces of positive dimension, and it is followed by
analyses of changing the order of operations, enlarging the relevant
face, computing socles of indicator modules of faces, and exactness.

\begin{lemma}\label{l:nabt}
If $\tau$ is a face of a real polyhedral group~$\cQ$ then the face
poset of the quotient real polyhedral group $\qrt$ is isomorphic to
the open star $\nabt$ from Example~\ref{e:nabla} by the map $\nabt \to
(\qrt)_+$ sending $\sigma \in \nabt$ to its image $\sigma/\tau$
in~$\qrt$.
\end{lemma}
\begin{proof}
The image $\qrtp$ of~$\cQ_+\!$~in~$\qrt$ is a cone that
generates~$\qrt$, and the group of units (lineality space) of the cone
$\cQ_+ \!+ \RR\tau$ is just~$\RR\tau$ itself because $\cQ_+$ is
pointed.  (See Remark~\ref{r:qrt} for further discussion.)
\end{proof}

\begin{defn}\label{d:upper-closure-tau}
In the situation of Lemma~\ref{l:nabt},
view $\nabt$ as a monoid and poset as in
Remark~\ref{r:semilattice=monoid}, so $\sigma \preceq \sigma'$
in~$\nabt$ if $\sigma \supseteq \sigma'$.  The \emph{upper closure
functor along~$\tau$} takes a $Q$-module~$\cM$ to the $\cQ \times
\nabt$-module $\dt\cM = \delta\cM/\bigoplus_{\sigma\not\in\nabt}\dsm$,
so $\dt\cM = \bigoplus_{\sigma\in\nabt}\ds[\tau]\cM$.
\end{defn}

The notation is such that $\ds[\tau] \neq 0 \iff \sigma \supseteq \tau$.

\begin{example}\label{e:upper-closure-tau}
Starting from Example~\ref{e:closure-module}, upper closures along the
positive $x$-axis and along $\RR^2_+$ are obtained by erasing certain
points.  The upper boundary along the $x$-axis, drawn at left,
$$
  \psfrag{0}{\footnotesize$\0$}
  \psfrag{x}{\footnotesize$x$}
  \psfrag{y}{\footnotesize$y$}
  \psfrag{1}{\footnotesize$\RR^2_+$}
  \psfrag{F}{$\fro$}
  \includegraphics[height=30ex]{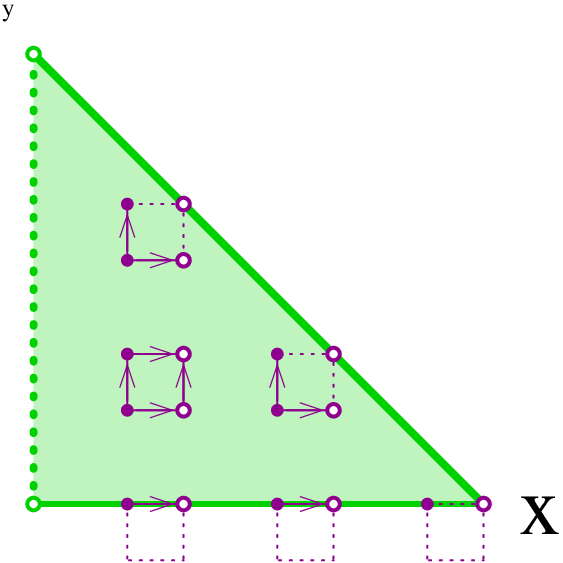}
\qquad\qquad
  \includegraphics[height=30ex]{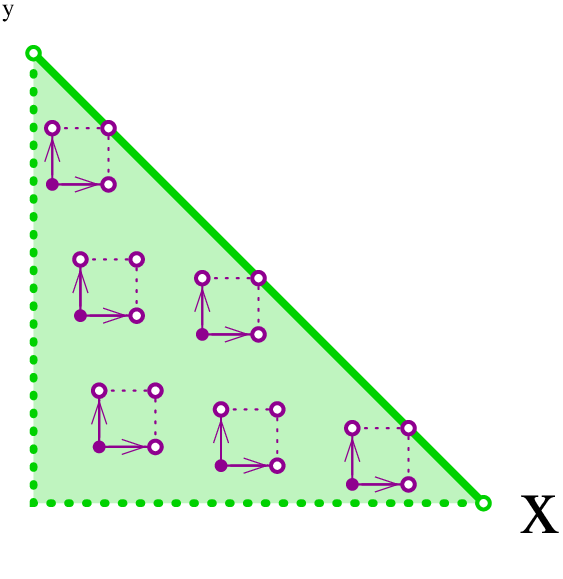}
\vspace{-1ex}%
$$
is obtained from Example~\ref{e:closure-module} by erasing the
lower-right corner of every little square that has one (because the
$y$-axis does not contain the face $\tau = x$-axis).  This results
automatically in erasure of the vertical left edge of the triangle.
The upper boundary along the interior~$\RR^2_+$, drawn at right, is
obtained by further erasing the upper-left corner of every little
square, which results automatically in erasure of the horizontal
bottom edge of the triangle.  The hypotenuse of the triangle is drawn
as a solid line here to emphasize that the $\fro$ factor has nonzero
components in these $\RR^2$-graded degrees, even though the module
itself has no nonzero elements in these degrees.
\end{example}

\begin{defn}\label{d:kats}
Fix a partially ordered abelian group~$\cQ$, a face~$\tau$, and an
arbitrary commutative monoid~$\cP$.  The \emph{skyscraper} $(\cQ
\times \cP)$-module at $(\aa,\sigma) \in \cQ \times \cP$ along~$\tau$
is
$$
  \kats = \kk[\aa+\tau] \otimes_\kk \kk_\sigma,
$$
the right-hand side being a module over the ring $\kk[\cQ_+]
\otimes_\kk \kk[\cP] = \kk[\cQ_+ \times \cP]$ with tensor factors as
in Definitions~\ref{d:socc} and~\ref{d:hhom}.
Set
$$
  \hhom_{\cQ\times\cP}\bigl(\kk[\tau],-\bigr) =
  \bigoplus_{(\aa,\sigma) \in
  \cQ\times\cP} \Hom_{\cQ\times\cP}\bigl(\kats,-\bigr).
$$
\end{defn}

\begin{remark}\label{r:kats}
When $\cP$ is trivial, this notation agrees with
Definition~\ref{d:hhom}, because $\cQ\times\{\0\} \cong \cQ$
canonically, so $\hhom_{\cQ\times\{\0\}}\bigl(\kk[\tau],-\bigr) =
\hhom_\cQ\bigl(\kk[\tau],-\bigr)$.
\end{remark}

\begin{defn}\label{d:soct}
Fix a real polyhedral group~$\cQ$, a face~$\tau$, and a
$\cQ$-module~$\cM$.
\begin{enumerate}
\item\label{i:global-soc-tau}%
The \emph{global cogenerator functor along~$\tau$} takes $\cM$ to its
\emph{global socle along~$\tau$}:
$$
  \soct\cM = \hhom_{\cQ\times\nabt}\bigl(\kk[\tau],\dt\cM\bigr)/\tau.
$$
The $\nabt$-graded components of $\soct\cM$ are denoted by
$\soct[\sigma]\cM$ for $\sigma \in \nabt$.

\item\label{i:local-soc-tau}%
The \emph{local cogenerator functor along~$\tau$} takes $\cM$ to its
\emph{local socle along~$\tau$}:
$$
  \soco(\cmt)
  =
  \socc\delta(\cmt)
  =
  \hhom_{\qrt\times\nabt}\bigl(\kk,\delta(\cM/\tau)\bigr),
$$
where the upper closure is over~$\qrt$ and the closed socle is over
$\qrt \times \nabt$.  Elements of $\soco(\cmt)$ are identified with
elements of~$\delta(\cmt)$ via $\phi \mapsto \phi(1)$.

\item\label{i:global-cogen}%
Regard $\hhom_{\cQ\times\nabt}\bigl(\kk[\tau],\dt\cM\bigr)$ as a
$(\cQ\times\nabt)$-submodule of~$\dt\cM$ via $\phi \mapsto \phi(1)$.
A homogeneous element $s$ in this submodule that maps to a nonzero
element of $\soct\cM$ is a \emph{global cogenerator} of~$\cM$
along~$\tau$, and if $s \in \ds[\tau]\cM$ then it has
\emph{nadir}~$\sigma$.  If $I \subseteq \cQ$ is an interval, then a
\emph{cogenerator} of~$I$ along~$\tau$ with nadir~$\sigma$ is the
degree in~$\cQ$ of a cogenerator of~$\kk[I]$ with nadir~$\sigma$
along~$\tau$.

\item\label{i:local-cogen}%
Regard $\soco(\cmt)$ as contained in~$\delta(\cmt)$ via $\phi \mapsto
\phi(1)$.  A nonzero homogeneous element in $\soco(\cmt)$ is a
\emph{local cogenerator} of~$\cM$ along~$\tau$.
\end{enumerate}
Assume the default modifier ``global'' when neither ``local'' nor
``global'' is written.
\end{defn}

\begin{example}\label{e:intro-soc}
See Example~\ref{e:nadir-x} for a socle along a face of positive
dimension (Definition~\ref{d:soct}.\ref{i:global-soc-tau}).  For
details, the upper closure of the interval module $\cM = \kk[I]$ there
along the face $\tau = x\text{-axis}_+$
(Definition~\ref{d:upper-closure-tau}) is $\dt\cM = \delta^\tau\cM
\oplus \delta^{\RR^2_+}\cM$.  Along the omitted (i.e., dotted)
horizontal ray, $\delta^\tau\cM$ vanishes because $\cM_\aa = 0$
for~$\aa$ in the omitted ray.  On the other hand,
$\delta^{\RR^2_+}\cM$ contains the skyscraper module
\noheight{$\kk_{\RR^2_+}[\aa + \tau]$} at $(\aa,\RR^2_+) \in \RR^2
\times \nabt$ along~$\tau$ (Definition~\ref{d:kats}) for any
degree~$\aa$ in the omitted ray and vanishes elsewhere.  The module
$\hhom_{\cQ\times\nabt}\bigl(\kk[\tau],\dt\cM\bigr)$ is therefore the
interval module whose degree set is the omitted ray; any element
therein is a global cogenerator of~$\cM$ along the horizontal
ray~$\tau$ with nadir~$\sigma = \RR^2_+$
(Definition~\ref{d:soct}.\ref{i:global-cogen}).  Localizing this
$\hhom$ along~$\tau$ yields the interval module for the entire line
containing the omitted ray, following the depiction in
Example~\ref{e:nadir-x}.  Taking quotient-restriction of this $\hhom$
module therefore yields a vector space of dimension~$1$ for the global
socle along the positive $x$-axis~$\tau$
(Definition~\ref{d:soct}.\ref{i:global-soc-tau}).  This vector space
resides in degree $\wt\aa = \aa + \ZZ\tau$, the coset through~$\aa$ of
the group $\ZZ\tau$ generated by~$\tau$, which is the single open dot
in the rightmost image in Example~\ref{e:nadir-x}.
\end{example}

\begin{example}\label{e:local-intro-soc}
In the setting of Example~\ref{e:intro-soc}, the corresponding local
socle computation first takes the quotient-restriction modulo~$\tau$,
as in Example~\ref{e:nadir-x}, and then takes the closed socle to get
$\soco(\cmt)$.  The local socle along~$\tau$ is naturally isomorphic
to the global socle along~$\tau$ in this case.
(This passage from global to local is general, by
Proposition~\ref{p:local-vs-global}: global cogenerators yield local
cogenerators.)
For a situation where the local and global socles do not agree, take
$\tau$ to be the positive $y$-axis in Example~\ref{e:global-support},
so the global socle along~$\tau$ vanishes while the local socle
along~$\tau$ is nonzero, as noted already in
Remark~\ref{r:soc-vs-supp}.
\end{example}

\begin{example}\label{e:nadir}
The closure $\delta\cM$ in Example~\ref{e:closure-module} explains the
nadirs in Example~\ref{e:trouble-soc}.  The other $\fro$-modules in
Example~\ref{e:closure-module}, which have solid dots in their
upper-right corners, do not yield socle elements of~$\delta\cM$
because the $\RR^2_+$-component can be moved up without annihilation.
That is not a universal statement, though: it only holds in this
example because all of the closed socles of~$\cM$
(Definition~\ref{d:socct}) vanish.  If~one of the points along the
antidiagonal upper boundary of the triangle were added to the
triangular interval, then the $\fro$-module at that upper boundary
point would have a solid dot in Example~\ref{e:closure-module}, and
that dot would represent a closed socle element of~$\cM$.
\end{example}

\begin{remark}\label{r:soct/tau}
The reason to quotient by~$\tau$ in
Definition~\ref{d:soct}.\ref{i:global-soc-tau} is to lump together all
cogenerators with nadir~$\sigma$ along the same translate
of~$\RR\tau$.  This lumping makes it possible for a socle basis to
produce a downset hull that is (i)~as minimal as possible and
(ii)~finite.  The lumping also creates a difference between the notion
of socle element and that of cogenerator: a socle element is a class
of cogenerators, these classes being indexed by elements in the
quotient-restriction; in Example~\ref{e:intro-soc}, for instance,
cogenerators lie along the omitted ray before the quotient
modulo~$\tau$, whereas the degree of the socle element is a (single)
coset of the $x$-axis.  In contrast, a local cogenerator is a
cogenerator of the quotient-restriction itself, so a local cogenerator
is already an element in the socle of the quotient-restriction; see
Example~\ref{e:local-intro-soc}.
This difference between socle element and cogenerator already arises
for closed socles along faces (Definition~\ref{d:socct}) but
disappears in the context of socles not along faces (see
Remark~\ref{r:socc0}), be they over real polyhedral groups
(Definition~\ref{d:soc}) or closed over posets
(Definition~\ref{d:socc}).
\end{remark}

\begin{remark}\label{r:soct-nabt}
If localization commuted with cogenerator functors, then the
restriction from $\fqo$ to~$\nabt$ in
Definition~\ref{d:soct}.\ref{i:global-soc-tau} would happen
automatically, because localizing $\cM$ along~$\tau$ would yield a
module over $\cQ_+\! + \RR\tau$, whose face poset is
naturally~$\nabt$.  But in this real polyhedral setting, the
restriction from $\fqo$ to~$\nabt$ must be imposed manually because
the $\Hom$ must be taken before localizing
(Remark~\ref{r:soc-vs-supp}), when the default face poset is
still~$\fqo$.
\end{remark}

\begin{remark}\label{r:nabt}
If $\aa$ is a cogenerator of a downset~$D$ along~$\tau$, then the
topology of~$D$ at~$\aa$ is induced by downsets of the form $\aa' -
\sigma^\circ$ for faces $\sigma \in \nabt$ and elements $\aa' \in \aa
+ \tau^\circ$.  This subtle issue regarding shapes of cogenerators
along~$\tau$ is a vital reason for using $\nabt$ instead of~$\fqo$.
It is tempting to expect that if a face~$\sigma$ is minimal in the
shape~$\nda$, then any expression of~$D$ as an intersection of
downsets must induce the topology of~$D$ at~$\aa$ by explicitly taking
$\aa - \sigma^\circ$ into account in one of the intersectands.  One
way to accomplish that would be for an intersectand to be a union of
downsets of the form $\bb - \cQ_{\nda}$ (see
Definition~\ref{d:cocomplex}) in which one of the elements~$\bb$
is~$\aa$.  But if $\sigma \in \nd[\aa']$ for all $\aa' \in \aa +
\tau$, or even merely for a single element $\aa' \in \aa +
\tau^\circ$, then
$$
  \aa - \sigma^\circ
  =
  \aa' - (\aa'-\aa - \sigma^\circ)
  \in
  \aa' - (\tau^\circ + \sigma^\circ)
  \subseteq
  \aa' - (\tau \vee \sigma)^\circ.
$$
As the purpose of cogenerators is to construct downset
decompositions as minimally as possible, it is counterproductive to
think of~$\sigma$ as being a valid $\fqo$-socle degree unless $\sigma
\in \nabt$, because otherwise it fails to give rise to an essential
cogenerator.  See Theorem~\ref{t:downset=union} for the most general
possible view of considerations in this Remark.
\end{remark}

\begin{remark}\label{r:open-upper-closure}
In terms of persistent homology, cogenerators are deaths of classes.
In that context, the need for upper closure functors and socle theory
beyond closed socles is particularly crucial, because the modules most
pertinent to applied topology are precisely those whose closed socles
vanish \cite{kashiwara-schapira2018, berkouk-petit2019,
strat-conical}.  That is, the upper boundaries of these modules are as
far from closed as possible.
\end{remark}

\begin{remark}\label{r:soc-as-k-vect}
Although $\soct\cM$ is a module over $\qrt \times \nabt$ by
construction, the actions of $\qrt$ and $\nabt$ on it are trivial, in
the sense that attempting to move a nonzero homogeneous element up in
one of the posets either takes the element to~$0$ or leaves it
unchanged.  (The latter only happens if the degree is unchanged, which
occurs only when acting by the identity~$\0 \in \qrt$ or when acting
by $\sigma \in \nabt$ on an element of $\nabt$-degree $\sigma'
\subseteq \sigma$.)  That is what it means to be a direct sum of
skyscraper modules.  It immediately implies the following, which is
used in the proof of Theorem~\ref{t:dense-hull-M}.
\end{remark}

\begin{lemma}\label{l:soc-as-k-vect}
Any direct sum decomposition of $\soct\cM$ as a vector space graded by
$\qrt \times \nabt$ is also a decomposition of $\soct\cM$ as a
$\qrt$-module or as a $\nabt$-module.
\end{lemma}
\begin{proof}
The argument appears in Remark~\ref{r:soc-as-k-vect}.
\end{proof}

\begin{lemma}\label{l:intersection-of-kernels}
If~$\tau$ is a face of a real polyhedral group~$\cQ$ and $\cN =
\bigoplus_{\sigma\in\nabt} \cN_\sigma$ is a module over $\cQ \times
\nabt$, then $\Hom_{\nabt}(\kk_\sigma,\cN)/\tau \cong
\Hom_{\nabt}(\kk_\sigma,\cnt)$, and hence
$$
  (\ntsoc\cN)/\tau \cong \ntsoc(\cnt).
$$
\end{lemma}
\begin{proof}
$\Hom_{\nabt}(\kk_\sigma,\cN)$ is the intersection of the kernels of
the $\cQ$-module homomorphisms $\cN_\sigma \to \cN_{\sigma'}$ for
faces $\sigma \supset \sigma'$, so the isomorphism of $\Hom$ modules
follows from Lemma~\ref{l:exact-qr}.  The socle isomorphism follows by
taking the direct sum over $\sigma \in \nabt$.
\end{proof}

\begin{prop}\label{p:either-order}
The functors $\socct\!$ and $\ntsoc$ commute.  In particular,
$$
  \ntsoc(\socct\dt\cM)
  \cong
  \soct\cM
  \cong
  \socct(\ntsoc\dt\cM).
$$
\end{prop}
\begin{proof}
By taking direct sums over $\aa$ and~$\sigma$, this is mostly the
natural isomorphisms
\begin{align*}
\Hom_{\nabt}\bigl(\kk_\sigma,\Hom_\cQ\bigl(\kk[\aa+\tau],-\bigr)\bigr)
  &\cong\Hom_{\cQ\times\nabt}\bigl(\kats,-\bigr)
\\
  &\cong\Hom_\cQ\bigl(\kk[\aa+\tau],\Hom_{\nabt}(\kk_\sigma,-)\bigr)
\end{align*}
that result from the adjunction between Hom and~$\otimes$.  Taking the
quotient-restriction along~$\tau$
(Definition~\ref{d:quotient-restriction}) almost yields the desired
result; the only issue is that the left-hand side requires
Lemma~\ref{l:intersection-of-kernels}.
\end{proof}

\begin{example}\label{e:soct}
If $\aa$ is a cogenerator of a downset $D \subseteq \cQ$ along~$\tau$
with nadir~$\sigma$, then reasoning as in
Example~\ref{e:soc-Rn-downset} and using Definition~\ref{d:del-nabla},
computing $\ntsoc$ first in Proposition~\ref{p:either-order} shows
that $\sigma \in \del(\nda \cap \nabt)$.  What $\socct$ then does is
verify that the image
of~$\aa$ in~$\qrt$ is maximal with this property, by
Example~\ref{e:socct}.
\end{example}

\begin{prop}\label{p:local-vs-global}
There is a natural injection
$$
  \soct\cM \into \soco(\cmt)
$$
for any module~$\cM$ over a real polyhedral group~$\cQ$ and any face
$\tau$ of~$\cQ$.
\end{prop}
\begin{proof}
By Proposition~\ref{p:local-vs-global-closed} $\socct\cN \into
\socc(\cnt)$ for $\cN = \ntsoc\dt\cM$ viewed as a $\qrt$-module.
Proposition~\ref{p:either-order} yields $\socct\cN = \soct\cM$.  It
remains to show that $\qtsoc(\cnt) = \soco(\cmt)$.  To that end, first
note that
$$
  (\ntsoc\dt\cM)/\tau
  \cong
  \ntsoc\bigl((\dt\cM)/\tau\bigr)
  \cong
  \ntsoc\delta(\cM/\tau),
$$
the first isomorphism by Lemma~\ref{l:intersection-of-kernels} and the
second by Lemma~\ref{l:ds-vs-qr}, which shows that the modules acted
on by $\ntsoc$ are isomorphic.  Now apply the last isomorphism in
Lemma~\ref{l:either-order}, with $\cQ$ replaced by~$\qrt$ so that
automatically $\fqo$ must be replaced by~$\nabt$ via
Lemma~\ref{l:nabt}.
\end{proof}

\begin{lemma}\label{l:ds-vs-qr}
If $\sigma \supseteq \tau$ then $(\dsm)/\tau \cong \dst(\cmt)$.
\end{lemma}
\begin{proof}
\hspace{-1.205pt}Explicit calculations from the definitions show that in
degree~$\aa/\hspace{-1pt}\tau$ both sides~equal
$$
  \dirlim_{\substack{\aa'\in\aa-\sigma^\circ\\\vv\in\tau}}\cM_{\aa'+\vv},
$$
although they take the colimits in different orders: $\vv$ first or
$\aa'$ first.  The hypothesis that $\sigma \supseteq \tau$ enters to
show that any direct limit over $\{\aa' \in \cQ \mid
\aa'/\tau\in\aa/\tau-(\sigma/\tau)^\circ\}$ can equivalently be
expressed as a direct limit over $\aa'\in\aa-\sigma^\circ$.
\end{proof}

\begin{cor}\label{c:at-most-one}
An indicator quotient for a downset in a real polyhedral group has at
most one linearly independent socle element along each face with given
nadir and~degree.  In fact, the degrees of independent socle elements
along~$\tau$ with fixed nadir are incomparable in~$\qrt$, and nadirs
of socle elements with fixed degree are
\mbox{incomparable}~in~$\nabt$.
\end{cor}
\begin{proof}
A socle element of an indicator quotient~$E$ along a face~$\tau$
of~$\cQ$ is a local socle element of~$E$ along~$\tau$ by
Proposition~\ref{p:local-vs-global}.  Local socle elements
along~$\tau$ are socle elements (along the minimal face~$\{\0\}$) of
the quotient-restriction along~$\tau$ by
Definition~\ref{d:soct}.\ref{i:local-soc-tau}.  But $E/\tau$ is an
indicator quotient of~$\kk[\qrt]$, so its socle degrees with fixed
nadir~$\sigma$ are incomparable, as are its nadirs with fixed socle
degrees, by~Example~\ref{e:soc-Rn-downset}.
\end{proof}

\begin{example}\label{e:soct-k[tau]}
Propositions~\ref{p:either-order} and~\ref{p:local-vs-global} ease
some socle computations.  To see how, consider the indicator
$\cQ$-module~$\kk[\rho]$ for a face~$\rho$ of~$\cQ$.
Proposition~\ref{p:local-vs-global} immediately implies that
$\soct\kk[\rho] = 0$ unless $\rho \supseteq \tau$, because localizing
along~$\tau$ yields $\kk[\rho]_\tau = 0$ unless $\rho \supseteq \tau$.

Next compute $\ds\kk[\rho]$.  When either $\aa \not\in \rho$ or
$\sigma \not\subseteq \rho$, the direct limit in
Definition~\ref{d:atop-sigma} is over a set $\aa - \sigma^\circ$ of
degrees in which $\kk[\rho] = 0$ in a neighborhood of~$\aa$.  Hence
the only faces that can appear in $\dt\kk[\rho]$ lie in the interval
between~$\tau$ and~$\rho$, so assume $\tau \subseteq \sigma \subseteq
\rho$.  If $\bigl(\ds\kk[\rho]\bigr){}_\aa \neq 0$ then it
equals~$\kk$ because $\kk[\rho]$ is an indicator module for a subset
of~$\cQ$.  Moreover, if $(\ds\kk[\rho])_\aa = \kk$ then the same is
true in any degree $\bb \in \aa + \rho$ because $(\bb -\nolinebreak
\aa) +\nolinebreak (\aa - \sigma^\circ) \cap \rho \subseteq (\bb -
\sigma^\circ) \cap \rho$.  Thus $\ds\kk[\rho]$ is torsion-free as a
$\kk[\rho]$-module.

The $\socct\!$ on the left side of Proposition~\ref{p:either-order},
which by Definition~\ref{d:socct}.\ref{i:global-socc-tau} is a
quotient-restriction of a module
$\hhom_\cQ\bigl(\kk[\tau],\dt\kk[\rho]\bigr)^{\!}$, can only be
nonzero if $\tau = \rho$, as any nonzero image of~$\kk[\tau]$ is a
torsion $\kk[\rho]$-module.  Hence the socle of~$\kk[\rho]$
along~$\tau$ equals the closed socle along $\tau = \rho$, which is
computed directly from Definition~\ref{d:soct}.\ref{i:global-soc-tau}
and Definition~\ref{d:quotient-restriction} to be
$\hhom_\cQ\bigl(\kk[\tau],\kk[\tau]\bigr)/\tau = \kk[\tau]/\tau$.  In
summary,
$$
  \soct\kk[\rho] =
  \begin{cases}
  \kk_\0\text{ for } \0 \in \qrt & \text{if } \tau = \rho
\\               0               & \text{otherwise.}
\end{cases}
$$
\end{example}

\begin{prop}\label{p:left-exact-tau}
The global cogenerator functor $\soct$ along any face~$\tau$ of a
real polyhedral group is left-exact, as is the local cogenerator
functor along~$\tau$.
\end{prop}
\begin{proof}
Proposition~\ref{p:left-exact-tau-closed} and
Lemma~\ref{l:exact-delta}.
\end{proof}

\section{Tame, semialgebraic, and PL socles}\label{s:semialg}

The tame, semialgebraic, and PL conditions are preserved under taking
socles.  That is the goal of this section, Theorem~\ref{t:soct-tame},
which states this preservation for the most general form of socle,
namely the cogenerator functor along a face
(Definition~\ref{d:soct}.\ref{i:global-soc-tau}) over any real
polyhedral group.  But because the various forms of socles in
Section~\ref{s:socle} occur in contexts more general than real
polyhedral groups, it is necessary to record the statements separately
for each form of socle.  The order in which they are covered here is
the same as in Section~\ref{s:socle}: closed socles $\socc$ over an
arbitrary poset~$\cQ$ (Proposition~\ref{p:socc-tame}); socles $\soco$
over real polyhedral groups (Corollary~\ref{c:soco-tame}); closed
socles $\socct$ along faces of polyhedral groups
(Proposition~\ref{p:socct-tame}); and finally socles $\soct$ along
faces of real polyhedral groups (Theorem~\ref{t:soct-tame}).  The
proofs are based on the observation that everything reduces to the
effects of cogenerator functors on downset modules.
Verifying the hypotheses for the criterion in
Theorem~\ref{t:endofunctor} for cogenerator functors in the tame case
is relatively straightforward.  The semialgebraic and PL cases require
more power\linebreak (Lemma~\ref{l:semialg-closure} and onward).
First, here is a handy concept
\cite[Definition~3.14]{hom-alg-poset-mods}.

\begin{defn}\label{d:connected-homomorphism}
Let each of $S$ and $S'$ be a nonempty interval in~$Q$
(Definition~\ref{d:indicator}.\ref{i:interval}).  A homomorphism $\phi:
\kk[S] \to \kk[S']$ of interval modules is \emph{connected} if there
is a scalar $\lambda \in \kk$ such that $\phi$ acts as multiplication
by~$\lambda$ on the copy of~$\kk$ in degree~$q$~for~all~$q \!\in\! S
\cap S'$.
\end{defn}

\begin{thm}\label{t:endofunctor}
Fix posets $\cQ$ and~$\cQ'$.  Suppose a left-exact functor~$\cS$ from
the category of $\cQ$-modules to the category of $\cQ'$-modules takes
each
\begin{itemize}
\item%
downset module $\kk[D]$ to a subquotient $\cS(\kk[D]) = \kk[\cS D]$
of\/~$\kk[\cQ']$, and
\item%
connected morphism $\kk[D] \to \kk[D']$ of downset modules to a
connected morphism $\kk[\cS D] \to \kk[\cS D']$ of interval modules.
\end{itemize}
Then
\begin{enumerate}
\item\label{i:finitely-encoded}%
the restriction of\/~$\cS$ to any tame morphism of tame $\cQ$-modules
(see Proposition~\ref{p:abelian-category}) yields a tame morphism of
of tame $\cQ'$-modules; and

\item\label{i:semialg}%
if $\cQ$ and~$\cQ'$ are partially ordered real vector spaces, and $\cS
D$ is semialgebraic in~$\cQ'$ for all semialgebraic downsets $D
\subseteq \cQ$, then $\cS$ takes every semialgebraic morphism of
$\cQ$-modules a semialgebraic morphism of $\cQ'$-modules.
\end{enumerate}
The previous claim remains true with ``PL'' in place of
``semialgebraic''.
\end{thm}
\begin{proof}
Assume $\cM$ is a tame $\cQ$-module.  Then $\cM = \ker(E^0
\to E^1)$ is the kernel of a tame morphism of finite
direct sums of
downset modules by the syzygy theorem for poset modules
\cite[Theorem~6.12.5]{hom-alg-poset-mods}.
Left-exactness implies that~\mbox{$\cS\cM\! = \ker(\cS E^0
\hspace{-.51pt}\to\hspace{-.51pt} \cS E^1)$}.  The first goal is to
show that $\cS\cM$ is a tame module.  For that it suffices by
Proposition~\ref{p:abelian-category} to show that $\cS E^0 \to \cS
E^1$ is a tame morphism.  But that also follows from
Proposition~\ref{p:abelian-category} because each component of $\cS
E^0 \to \cS E^1$ has the form $\kk[\cS D] \to \kk[\cS D']$ and is
hence a tame morphism by hypothesis.  The argument works mutatis
mutandis in the semialgebraic and PL cases by
\cite[Theorem~6.12]{hom-alg-poset-mods}.

Now suppose that a morphism $\cM \to \cM'$ is given.  The syzygy
theorem again \cite[Theorem~6.12]{hom-alg-poset-mods} yields
a~copresentation $\cM' = \ker(E^0{}' \!\to\! E^1{}')$ such that the
map $\cM \to\nolinebreak \cM'$ is induced by a morphism of
copresentations.  The map $\cS\cM \to \cS E^0{}'$
obtained by restricting $\cS E^0 \to \cS E^0{}'$ to $\cS\cM$ is tame
because, in general, any tame morphism restricts to a tame morphism on
any tamely included submodule; the same goes for semialgebraic or PL
in place of tame.  The morphism $\cS\cM \to \cS\cM'$ is hence tame,
semialgebraic, or PL by any common refinement of two encodings
(Definition~\ref{d:encoding}) of~$\cS E^0{}'$ subordinate to the
morphisms (Definition~\ref{d:tame-morphism}) from $\cS\cM$
and~$\cS\cM'$.
\end{proof}

\begin{example}\label{e:upper-closure-tame}
The upper closure functor $\cS = \ds$ atop~$\sigma$ in
Definition~\ref{d:upper-closure} satisfies the hypotheses of
Theorem~\ref{t:endofunctor}.\ref{i:finitely-encoded} with $\cQ' = \cQ$
by Lemma~\ref{l:exact-delta} and
Proposition~\ref{p:downset-upper-closure}.
Therefore $\dsm$ is a tame $\cQ$-module if $\cM$ is tame, and $\dsm
\to \dsm'$ is a tame morphism if $\cM \to \cM'$ is.  Hence the same is
true with~$\delta$ in place of~$\ds$, since \noheight{$\delta =
\bigoplus_{\sigma \in \cfq} \ds$} as endofunctors on the category of
$\cQ$-modules.  And the same is true of $\dt$ for any face $\tau \in
\cfq$ (Definition~\ref{d:upper-closure-tau}), since it is a subdirect
sum of~$\delta$.  The corresponding semialgebraic conclusions are true
as well, this time using Theorem~\ref{t:endofunctor}.\ref{i:semialg},
but checking that uses Proposition~\ref{p:semialgebraic-D}, which
requires a bit more power.
\end{example}


\begin{lemma}\label{l:semialg-closure}
If $X \subseteq \RR^n$ and $X \to Y$ is a morphism of semialgebraic
varieties, then the family $\ol X_Y$ obtained by taking the closure
in~$\RR^n$ of every fiber of~$X$ is semialgebraic.
\end{lemma}
\begin{proof}
This is a consequence of Hardt's theorem \cite[Theorem~4]{hardt80}
(see also \cite[Remark~II.3.13]{shiota97}), which says that over a
subset of~$Y$ whose complement in~$Y$ has dimension less than $\dim
Y$, the family $X \to Y$ is trivial.
\end{proof}

\begin{prop}\label{p:semialgebraic-D}
If $D$ is a semialgebraic or PL downset in a real polyhedral
group~$\cQ$ and $\sigma \in \cfq$ is a face then $\dsd$ is similarly
semialgebraic or~PL.
\end{prop}
\begin{proof}
Semialgebraic case: Lemma~\ref{l:semialg-closure} with $Y = \qrs$ and
$X = D$ by Lemma~\ref{l:downset-upper-closure}.

PL case: in Lemma~\ref{l:downset-upper-closure} the union can be
broken over finitely many relatively open polyhedral cells
comprising~$D$.  So assume $D$ is a single relatively open polyhedral
cell.  The union in Lemma~\ref{l:downset-upper-closure} is plainly a
subset of~$\oD$.  Indeed, if $\pi: \oD \to Q/\RR\sigma$, then the
union is $\pi^{-1}\bigl(\pi(D)\bigr)$.  That is, the union is the
complement in~$\oD$ of the (closed) faces of~$\oD$ whose projections
mod $\RR\sigma$ are contained in the boundary of~$\pi(\oD)$.
\end{proof}

Proposition~\ref{p:socc-tame} covers the case of closed socles over an
arbitrary poset, with the next lemma needed for the semialgebraic and
PL cases.

\begin{lemma}\label{l:semialgebraic}
If $D$ is a semialgebraic or PL downset in a real polyhedral group
then its set $\max D$ of maximal elements is similarly semialgebraic
or PL, as well.
\end{lemma}
\begin{proof}
The semialgebraic proof relies on standard operations on subsets that
preserve the semialgebraic property; see \cite[Chapter~II]{shiota97},
for instance.  As it happens, the proof works verbatim for the PL case
because the relevant (in)equalities are linear.

Inside of $\RR^n \times \RR^n$, consider the subset~$X$ whose fiber
over each point $\aa \in D$ is $\aa + \mm$, where $\mm = \cQ \minus
\{\0\}$ is the maximal monoid ideal of~$\cQ_+$.  Note that $\mm$ is
semialgebraic because it is defined by linear inequalities and a
single linear inequation.  The subset~$X \subseteq \RR^n \times \RR^n$
is semialgebraic because it is the image of the algebraic morphism $D
\times \mm \to D \times \RR^n$ sending $(\aa,q) \mapsto (\aa,\aa+q)$.
The intersection of~$X$ with the semialgebraic subset $D \times D$
remains semialgebraic, as does the projection of this intersection
to~$D$.  The image of the projection is $D \minus \max D$ because
$(\aa + \mm) \cap D = \nothing$ precisely when $\aa \in \max D$.
Therefore $\max D = D \minus (D \minus \max D)$ is semialgebraic.
\end{proof}

\begin{prop}\label{p:socc-tame}
If a module $\cM$ over any poset is tame then so is its closed socle
$\socc\cM$.  If $\cM$ is semialgebraic or PL over a real polyhedral
group then so is~$\socc\cM$.  If $\cM \to \cM'$ is a tame,
semialgebraic, or PL morphism, then so is $\socc\cM \to \socc\cM'$.
\end{prop}
\begin{proof}
Apply Theorem~\ref{t:endofunctor}: left-exactness is
Lemma~\ref{l:left-exact-socc}, the criteria on downset modules and
connected morphisms between them both follow from
Example~\ref{e:socc}, and the semialgebraic or PL criterion is
Lemma~\ref{l:semialgebraic}.
\end{proof}

The next three results cover the case of socles over real polyhedral
groups.

\begin{lemma}\label{l:frontier-tame}
The homomorphisms $\dsm \to \dsp\cM$ for faces $\sigma \supseteq
\sigma'$ afforded by Proposition~\ref{l:natural} are tame,
semialgebraic, or PL if $\cM$~is.
\end{lemma}
\begin{proof}
Use \cite[Theorem~6.12.5]{hom-alg-poset-mods}
to express $\cM = \ker(E^0 \to E^1)$ as the kernel of a downset
copresentation that is tame, semialgebraic, or PL as the case may be.
For a single downset~$D$, observe that $\dsp D \subseteq \dsd$
whenever $\sigma \supseteq \sigma'$ by
Proposition~\ref{p:downset-upper-closure}.\ref{i:nabla}.
Therefore, by Proposition~\ref{p:downset-upper-closure}, the natural
map $\ds\kk[D] \to \dsp\kk[D]$ is a quotient of downset modules, which
is a connected homomorphism
(Definition~\ref{d:connected-homomorphism}) and hence tame,
semialgebraic, or PL, as the case may be.  The homomorphism $\dsm
\to\nolinebreak \dsp\cM$ is induced by the morphism $\ds E^\spot \to
\dsp E^\spot$ of copresentations.~~The argument in the final two
sentences of the proof of Theorem~\ref{t:endofunctor} therefore works
here.
\end{proof}

The next result is stated in the generality of $\ntsoc\dt$ (see
Section~\ref{b:soc-along}, including~$\dt$ in
Definition~\ref{d:upper-closure-tau}) for the eventual purpose of
Theorem~\ref{t:soct-tame}, even though for the time being all that is
needed is the case $\tau = \{\0\}$, where $\ntsoc\dt =
\fqsoc\delta$~(see~\mbox{Section}~\ref{b:socc}).

\begin{prop}\label{p:ntsoc}
Fix a real polyhedral group~$\cQ$ and a face~$\tau \in \cfq$.  The
endofunctor on the category of $\cQ$-modules that takes $\cM$ to
$\ntsoc\dt$ restricts to endofunctors on the categories of tame
$\cQ$-modules, semialgebraic $\cQ$-modules, and PL $\cQ$-modules.
\end{prop}
\begin{proof}
The same proof works, mutatis mutandis, for the semialgebraic and PL
cases.

The $\nabt$-graded component of~$\ntsoc\dt\cM$ in
$\nabt$-degree~$\sigma$ is the intersection of the kernels of the
$\cQ$-module homomorphisms $\dsm \to \dsp\cM$ for $\sigma \supseteq
\sigma'$.  These are tame morphisms, if $\cM$ is tame, by
Lemma~\ref{l:frontier-tame}.  The intersection of their kernels is
tame by Proposition~\ref{p:abelian-category} because the intersection
of kernels of homomorphisms from a single module to finitely many
modules is the kernel of the diagonal homomorphism to the direct sum.
So $\ntsoc\dt\cM \to \dt\cM$ is tame.

Any given tame morphism $\cM \to \cN$ induces a tame morphism
$\delta\cM \to \delta\cN$ by Example~\ref{e:upper-closure-tame}.
Hence the composite $\ntsoc\dt\cM \to \dt\cM \to \dt\cN$ is a tame
morphism that happens to have its image in $\ntsoc\dt\cN$.  On the
other hand, $\ntsoc\dt\cN \to \dt\cN$ is also a tame morphism.  The
morphism $\ntsoc\dt\cM \to \ntsoc\dt\cN$ is tame by any common
refinement of two poset encodings of~$\dt\cN$ subordinate to the
morphisms from $\ntsoc\dt\cM$ and~$\ntsoc\dt\cN$.
\end{proof}

\begin{cor}\label{c:soco-tame}
If a module~$\cM$ over a real polyhedral group is tame, semialgebraic,
or PL then so is its socle $\soco\cM$.  If $\cM \to \cM'$ is a tame,
semialgebraic, or PL morphism, then the natural map $\soco\cM \to
\soco\cM'$ is, as well.
\end{cor}
\begin{proof}
By Lemma~\ref{l:either-order}, $\soco\cM$ is the composite of the
functors $\rnsoc$ and $\fqsoc\delta$, which preserve the tame,
semialgebraic, and PL categories by Propositions~\ref{p:socc-tame}
and~\ref{p:ntsoc}, the latter of which has $\ntsoc\dt = \fqsoc\delta$
when $\tau = \{\0\}$.
\end{proof}

The next two results cover closed socles along faces of arbitrary
polyhedral groups.

\begin{lemma}\label{l:semialgebraic'}
If $D$ is a semialgebraic or PL downset in a real polyhedral
group~$\cQ$ then $\mtd$ is similarly semialgebraic or PL in~$\qrt$ for
any face~$\tau$ of~$\cQ_+$.
\end{lemma}
\begin{proof}
The projection of a semialgebraic set is semialgebraic, so by
Example~\ref{e:socct} it suffices to prove that the set of degrees of
closed cogenerators of~$\kk[D]$ along~$\tau$ is semialgebraic.  The
argument comes in two halves, both following the framework of the
proof of Lemma~\ref{l:semialgebraic}.  For the first half, simply
replace $\mm$ by $\mm_\tau = \cQ_+ \minus \tau$ to find that
$\big\{\aa \in D \mid (\aa + \cQ_+) \cap D \subseteq \aa + \tau\big\}$
is semialgebraic.  The second half uses $\tau$ instead of~$\mm$, and
this time it intersects the subset~$X$ with $D \times (\cQ \minus D)$
to find that $\big\{\aa \in D \mid (\aa + \cQ_+) \cap D \supseteq \aa
+ \tau\big\}$ is semialgebraic.  The desired set of degrees is the
intersection of these two semialgebraic sets.  Replacing
``semialgebraic'' with ``PL'' works because, again, the relevant
(in)equalities are linear.
\end{proof}

\begin{prop}\label{p:socct-tame}
If a module $\cM$ over a partially ordered abelian group is tame then
so is its closed socle $\socct\cM$ along any face~$\tau$.  If $\cM$ is
semialgebraic or PL over a real polyhedral group then so
is~$\socct\cM$.  If $\cM \to \cM'$ is a tame, semialgebraic, or PL
morphism, then the natural map $\socct\cM \to \socct\cM'$ is, as well.
\end{prop}
\begin{proof}\enlargethispage{.25ex}%
Apply Theorem~\ref{t:endofunctor}: left-exactness is
Proposition~\ref{p:left-exact-tau-closed}; the criteria on downset
modules and connected morphisms between them follow from
Example~\ref{e:socct}, noting that the set $\mtd$ is an upset in the
downset $\dzt$ because it is contained in the set of maximal elements
of~$\dzt$; and the semialgebraic or PL criterion is
Lemma~\ref{l:semialgebraic'}.
\end{proof}

Finally, here is the version covering total socles over real
polyhedral groups.

\begin{thm}\label{t:soct-tame}
Over a real polyhedral group~$\cQ$, the cogenerator functor $\soct$
along any face~$\tau$ restricts to endofunctors on the categories of
tame, semialgebraic, or PL modules over~$\cQ$.  For any face~$\sigma
\supseteq \tau$, this statement remains true for the cogenerator
functor $\soct[\sigma]$ along~$\tau$ with nadir~$\sigma$.
\end{thm}
\begin{proof}
The cogenerator functor $\soct\cM$ is the composite of $\soct$
and~$\ntsoc\dt$ by Proposition~\ref{p:either-order}.  These functors
preserve the tame, semialgebraic, and PL categories by
Propositions~\ref{p:socct-tame} and~\ref{p:ntsoc}.  The
$\soct[\sigma]$ claim follows by taking $\nabt$-graded pieces.
\end{proof}

\section{Essential property of socles}\label{s:essential}

In this section, $\cQ$ is a real polyhedral group unless otherwise
stated.

The culmination of the foundations developed in Section~\ref{s:socle}
says that socles and cogenerators detect injectivity of homomorphisms
between tame modules over real polyhedral groups
(Theorem~\ref{t:injection}), as they do for noetherian rings in
ordinary commutative algebra.  The theory is complicated by there
being no actual submodule containing a given non-closed socle element;
that is why socles are functors that yield submodules of localizations
of auxiliary modules rather than submodules of localizations of the
given module itself.  Nonetheless, it comes down to the fact that,
when $D \subseteq \cQ$ is a downset, every element can be pushed up to a
cogenerator.  Theorem~\ref{t:divides} contains a precise statement of
this assertion that suffices for the purpose of
Theorem~\ref{t:injection}, although the definitive version of
Theorem~\ref{t:divides} occurs in Section~\ref{s:minimality}, namely
Theorem~\ref{t:downset=union}.

The proof of Theorem~\ref{t:divides} requires a
definition---essentially the notion dual to that of shape
(Proposition~\ref{p:shape}).  Informally, it is the set of faces
$\sigma$ such that a neighborhood of~$\aa$ in $\aa + \sigma^\circ$ is
contained in the downset~$D$.  The formal definition reduces by
negation to the discussion surrounding tangent cones of downsets
(Section~\ref{b:tangent}), noting that the negative of an upset is a
downset.

\begin{defn}\label{d:upshape}
The \emph{upshape} of a downset $D$ in a real polyhedral group~$\cQ$
at~$\aa$ is
$$
  \dda = \cfq \minus \nabla_{\!-U}^{-\aa\,},
$$
where $U = \cQ \minus D$ is the upset complementary to~$D$.
\end{defn}

\begin{lemma}\label{l:upshape}
The upshape $\dda$ is a polyhedral complex (a downset) in~$\cfq$.  As
a function of\/~$\aa$, for fixed~$D$ the upshape $\dda$ is decreasing,
meaning $\aa \preceq \bb \implies \dda \supseteq \ddb$.
\end{lemma}
\begin{proof}
These claims are immediate from the discussion in
Section~\ref{b:tangent}.
\end{proof}

\begin{remark}\label{r:upshape}
The upshape $\dda$ is a rather tight analogue of the Stanley--Reisner
complex of a simplicial complex, or more generally the lower Koszul
simplicial complex \cite[Definition~5.9]{cca} of a monomial ideal in a
degree from~$\ZZ^n$.  To make the analogy even tighter, the complex
$K_\bb(I)$ would need to be indexed by $\bb - \mathrm{supp}(\bb)$, and
the cocomplex $\dda$ would need to be replaced by its opposite
poset~$(\dda)^\op$, which is a polyhedral subcomplex of the cone polar
to~$\cQ_+$.  Along these same lines, the shape $\nda$ of a downset~$D$
at an element $\aa \in \cQ$ (Proposition~\ref{p:shape} via
Definition~\ref{d:tangent-cone}) is analogous to the upper Koszul
simplicial complex of a monomial ideal \cite[Definition~1.33]{cca}.
This makes $\dda$ and~$\nda$ Alexander dual to one another, in a sense
generalizing that of \cite[Proof of Proposition~5.1]{cca}, which
occurs in a simplex.  Self-duality of the simplex means that a the
Alexander dual simplicial complex is a subcomplex of an isomorphic
simplex, whereas here the Alexander dual polyhedral complexes lie in
cones that are polar to one another.  But the analogue of the
topological Alexander duality theorem \cite[Theorem~5.9]{cca} holds
for $\nda$ and~$(\dda)^\op$ for the same combinatorial reason: the
union $\nda \cup \dda$ of posets is the face poset of~$\cQ_+$, which
is contractible.
\end{remark}

The general statement in Theorem~\ref{t:divides} about pushing up to
cogenerators relies on the special case of closed cogenerators for
closed downsets as follows.

\begin{lemma}\label{l:cogenerator}
If $D \subseteq \cQ$ is a downset and the part of~$D$ above $\bb \in
D$ is closed, so $(\bb + \cQ_+) \cap D = (\bb + \cQ_+) \cap \oD$, then
$\bb \preceq \aa$ for some closed cogenerator~$\aa$ of~$D$.
\end{lemma}
\pagebreak[3]
\begin{proof}
It is possible that $\bb + \cQ_+ \subseteq D$, in which case $D = \cQ$
and $\bb$ is by definition a closed cogenerator along $\tau = \cQ_+$.
Barring that case, the intersection $(\bb + \cQ_+) \cap \del D$ of the
principal upset at~$\bb$ with the boundary of~$D$ is nonempty.  Among
the points in this intersection, choose~$\aa$ with minimal upshape.
Observe that $\{\0\} \in \dda$ because $\aa \in D$, so~$\dda$ is
nonempty.  By assumption, the part of~$D$ above~$\bb$ is closed.
Thus~$\aa \in D$.

Let $\tau \in \dda$ be a facet (maximal simplex).  The goal is to
conclude $\dda = \cF_\tau$ has no facet other than~$\tau$, for then
$\dd{\aa'} = \cF_\tau$ for all $\aa' \succeq \aa$ in~$D$ by upshape
minimality and Lemma~\ref{l:upshape},
so~$\aa$ is a cogenerator of~$D$ along~$\tau$ by
Definition~\ref{d:socct} (see also Remark~\ref{r:witness}).

Suppose that $\rho \in \cfq$ is a ray that lies outside of~$\tau$.  If
$\rho \in \dda$ then upshape minimality implies $\rho \in \dd{\aa'}$
for any $\aa' \in (\aa + \tau^\circ) \cap D$, and such an $\aa'$
exists by definition of upshape.  Consequently, some face containing
both~$\rho$ and~$\tau$ lies in~$\dda$: if $\vv$ is any sufficiently
small vector along~$\rho$, then $\aa' + \vv = \aa + (\aa' - \aa) + \vv
\in D$, and the smallest face containing $(\aa' - \aa) + \vv$ contains
both the interior of~$\tau$ (because it contains $\aa' - \aa$)
and~$\rho$ (because it contains~$\vv$).  But this is impossible, so in
fact $\dda = \cF_\tau$.
\end{proof}

\begin{thm}\label{t:divides}
If $D$ is a downset in a real polyhedral group~$\cQ$ and~$\bb \in
\oD$, then there are faces $\tau \subseteq \sigma$ of~$\cQ_+\!$~and a
cogenerator~$\aa$ of~$D$ along~$\tau$ with nadir~$\sigma$ such that
$\bb \preceq \aa$.
\end{thm}
\begin{proof}
It is possible that $\bb + \cQ_+ \subseteq D$, in which case $D = \cQ$
and $\bb$ is by definition a closed cogenerator along $\tau =
\cQ_+\!$, which is the same as a cogenerator along~$\cQ_+\!$~with
nadir~$\cQ_+\!$.  Barring that case, the intersection $(\bb + \cQ_+)
\cap \del D$ of the principal upset at~$\bb$ with the boundary of~$D$
is nonempty.  Among the points in this intersection, there is one with
minimal shape, and it suffices to treat the case where this point
is~$\bb$ itself.

Minimality of~$\nd[\bb]$ implies that the shape does not change upon
going up from~$\bb$ while staying in the closure~$\oD$.
Consequently, given any face $\sigma \in \nd[\bb]$, the shape of~$D$
at every point in $\bb + \cQ_+\!$~that lies in~$\oD$ also
contains~$\sigma$.  Equivalently by
Proposition~\ref{p:downset-upper-closure}.\ref{i:nabla}, $(\bb +
\cQ_+) \cap \dsd = (\bb + \cQ_+) \cap \oD$.  Lemma~\ref{l:cogenerator}
applied to $\dsd$ produces a closed cogenerator~$\aa$ of~$\ds D$,
along some face~$\tau$, satisfying $\bb \preceq \aa$.  Since $\nda$ is
a nonempty cocomplex, its intersection with $\nabt$ is nonempty, so
assume $\sigma \in \nda \cap \nabt$.  The closed cogenerator~$\aa$
of~$\dsd$ need not be a cogenerator of~$D$, but if $\sigma$ is
\mbox{minimal}~\mbox{under}~inclusion~in~$\nda \cap\nolinebreak
\nabt$, then $\aa$ is indeed a cogenerator of~$D$ along~$\tau$ with
nadir~$\sigma$ by Proposition~\ref{p:either-order}---specifically the
first displayed isomorphism---applied to Example~\ref{e:socc}.
\end{proof}

\begin{remark}\label{r:essential}
The arguments in the preceding two proofs are essential to the whole
theory of socles, which hinges upon them.  The structure of the
arguments dictate the forms of all of the notions of socle,
particularly those involving cogenerators along~faces.
\end{remark}

Theorem~\ref{t:injection} is intended for tame modules, but because it
has no cause to deal with generators, in actuality it only requires
half of a fringe presentation (or a little less; see
Definition~\ref{d:downset-hull}).  The statement uses divisibility
(Definition~\ref{d:divides}), which works verbatim for~$\dt\cM$, by
Definition~\ref{d:upper-closure-tau}, because it refers only to upper
boundaries atop a single~face~$\sigma$.

\pagebreak[3]

\begin{thm}[Essentiality of socles]\label{t:injection}
Fix a homomorphism $\phi: \cM \to \cN$ of modules over a real
polyhedral group~$\cQ$.
\begin{enumerate}
\item\label{i:phi=>soct}%
If $\phi$ is injective then $\soct\phi: \soct\cM \to \soct\cN$ is
injective for all faces~$\tau$ of~$\cQ_+$.
\item\label{i:soct=>phi}%
If $\soct\phi: \soct\cM \to \soct\cN$ is injective for all
faces~$\tau$ of~$\cQ_+\!$~and~$\cM$ is downset-finite, then $\phi$
is~injective.
\end{enumerate}
If $\cM$ is downset-finite then each homogeneous element of~$\cM\!$
divides a cogenerator of~$\cM$.
\end{thm}
\begin{proof}
Item~\ref{i:phi=>soct} is a special case of
Proposition~\ref{p:left-exact-tau}.  Item~\ref{i:soct=>phi} follows
from the divisibility claim, for if $z$ divides a cogenerator~$s$
along~$\tau$ then $\phi(z) \neq 0$ whenever $\soct\phi(\wt s) \neq 0$,
where $\wt s$ is the image of~$s$ in~$\soct\cM$.

For the divisibility claim, fix a downset hull $\cM \into
\bigoplus_{j=1}^k E_j$ and a nonzero $z \in \cM_\bb$.  For some~$j$
the projection $z_j \in E_j$ of~$z$ divides a cogenerator of~$E_j$
along some face~$\tau$ with some nadir~$\sigma$ by
Theorem~\ref{t:divides}.  Choose one such cogenerator~$s_j$, and
suppose it has degree $\aa \in \cQ$.  There can be other indices~$i$
such that $(\soct[\sigma]E_i)_\wt\aa \neq 0$, where $\wt\aa$ is the
image of~$\aa$ in~$\qrt$.  For any such index~$i$, as long as $z_i
\neq 0$ it divides a unique cogenerator in $s_i \in \ds[\tau]E_i$ by
Corollary~\ref{c:at-most-one}.  Therefore the image of~$z$ in $E =
\bigoplus_{j=1}^k E_j$ divides the sum of these cogenerators~$s_j$.
But that sum is itself another cogenerator of~$E$ along~$\tau$ with
nadir~$\sigma$ in degree~$\aa$, and the fact that $z$ divides it
places the sum in the image of the injection
(Lemma~\ref{l:exact-delta}) $\ds[\tau]\cM \into \ds[\tau]E$.
\end{proof}

\begin{remark}\label{r:right-endpoints}
In terms of persistent homology, Theorem~\ref{t:injection} says that a
homomorphism of real multipersistence modules is injective if and only
if it takes the ``right endpoints'' of the source injectively to a
subset of the ``right endpoints'' of the target.
\end{remark}

\begin{cor}\label{c:essential-submodule}
Fix a downset-finite module $\cM$ over a real polyhedral~group.
\begin{enumerate}
\item\label{i:0}%
$\cM = 0$ if and only if $\soct\cM = 0$ for all faces~$\tau$.

\item\label{i:cap}%
$\soct\cM' \cap \soct\cM'' = \soct(\cM' \cap \cM'')$ in~$\soct\cM$ for
submodules $\cM'$ and~$\cM''$~of~$\cM$.
\end{enumerate}
\end{cor}
\begin{proof}
That $\cM = 0 \implies \cM = 0$ is trivial.  On the other hand, if
$\soct\cM = 0$ for all~$\tau$ then $\cM$ is a submodule of~$0$ by
Theorem~\ref{t:injection}.\ref{i:soct=>phi}.

The second equality follows from left-exactness
(Proposition~\ref{p:left-exact-tau}):
\begin{align*}
  \soct(\cM' \cap \cM'')
& =
  \soct\ker(\cM' \to \cM/\cM'')
\\*
& =
  \ker\bigl(\soct\cM' \to \soct(\cM/\cM'')\bigr)
\\*
& =
  \ker(\soct\cM' \to \soct\cM/\soct\cM'')
\\*
& =
  \soct\cM' \cap \soct\cM'',
\end{align*}
where the penultimate equality is because $\soct\cM''$ is the kernel
of the homomorphism $\soct\cM \to \soct(\cM/\cM'')$, so that
$\soct\cM/\soct\cM'' \into \soct(\cM/\cM'')$.
\end{proof}

There is a much stronger statement connecting socles to essential
submodules (Theorem~\ref{t:essential-submodule}), but it requires
language to speak of density in socles as well as tools to produce
submodules from socle elements, which are the main themes of
Section~\ref{s:minimality}.

\section{Minimality of socle functors}\label{s:minimality}

Socles capture any downset by hanging coprincipal downsets from
maximal elements in closures along faces; that is the main content of
socle essentiality (Theorem~\ref{t:injection}), or more precisely
Theorem~\ref{t:divides}.  But since limits and closures are involved,
it is reasonable to ask if anything smaller than the full socle still
captures the entirety of any given downset.  Algebraically, for
arbitrary modules, this asks for subfunctors of cogenerator functors.
The particular subfunctors here concern the graded degrees of socle
elements, for which notation is needed.

\begin{defn}\label{d:degree}
The \emph{degree set} of any module~$\cN$ over a poset~$\cP$ is
$$
  \deg\cN  = \{\aa \in \cP \mid \cN_\aa \neq 0\}.
$$
Write $\deg_\cP = \deg$ if more than one poset could be intended.
\end{defn}

Determining whether certain cogenerators along given faces can be
omitted from the socle requires notions of limit and closure.  The
relevant topologies are introduced in Section~\ref{b:nbds}
(Definitions~\ref{d:vicinity}, \ref{d:sigma-topology},
and~\ref{d:sigma-closure}), with numerous explicit examples and some
consequences for interval modules, such as
Proposition~\ref{p:sigma-nbd-cogen} and
Corollary~\ref{c:soc(coprimary)}.  This topological infrastructure is
applied to characterize which cogenerators are needed for a given
downset in Theorem~\ref{t:downset=union}, one of the main results of
the paper by virtue of generalizing irredundant irreducible
decomposition of monomial ideals to arbitrary downsets in real
polyhedral groups.  (See also Theorem~\ref{t:interval=union} and
Corollary~\ref{c:interval=union}, which are versions for arbitrary
intervals instead of downsets.)  Theorem~\ref{t:downset=union} is the
only result in Section~\ref{b:dense-downsets}, albeit surrounded by
examples, remarks, and a couple of useful corollaries.  These lead to
the goal of Section~\ref{s:minimality}---another main result of the
paper---the topological characterization in
Theorem~\ref{t:dense-subfunctor} of socle subfunctors that detect
injectivity.  (The same topological characterization detects essential
submodules; see Theorem~\ref{t:essential-submodule}.)

\subsection{Neighborhoods of group elements}\label{b:nbds}\mbox{}

\medskip
\noindent
The topological condition characterizing when enough cogenerators are
present is a sort of density in the set of all cogenerators.
Lemma~\ref{l:downset-upper-closure} has a related closure notion.

Recall the statement and context of Lemma~\ref{l:<<}, which says that
$\qns = \sigma^\circ + \cQ_+$.

\begin{defn}\label{d:vicinity}
Fix faces~$\sigma \supseteq \tau$ of a real polyhedral group~$\cQ$.  A
\emph{$\sigma$-vicinity} of a point $\wt\aa \in \qrt$ is a subset of
$\qrt$ of the form $(\uu + \qns)/\hspace{.2ex}\RR\tau$ such that $\aa
- \uu \in \sigma^\circ$ and $\aa$~is a representative for the coset
$\wt\aa = \aa + \RR\tau \in \qrt$.
\end{defn}

\begin{example}\label{e:sigma-vicinity}
Let $\cQ = \RR^2$ and $\tau = \{\0\}$.  Take for $X \subset \RR^2$ the
convex hull of $\0,\ee_1,\ee_2$ but with the first standard basis
vector $\ee_1$ removed.  If $\sigma$ is the $x$-axis of~$\RR^2_+$,
then in
$$
  X =
  \begin{array}{@{}c@{}}\includegraphics[height=20mm]{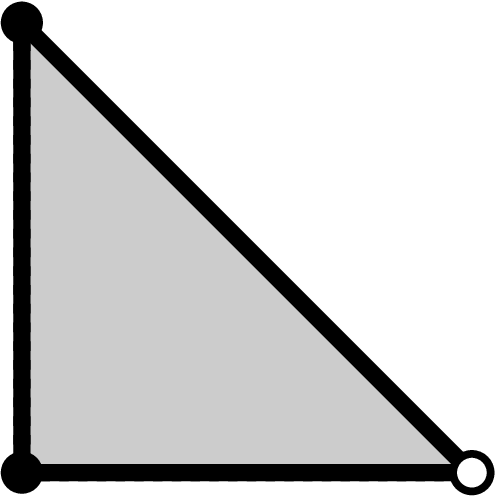}\end{array}
  \text{the blue points of}\quad
  \begin{array}{@{}c@{}}\includegraphics[height=20mm]{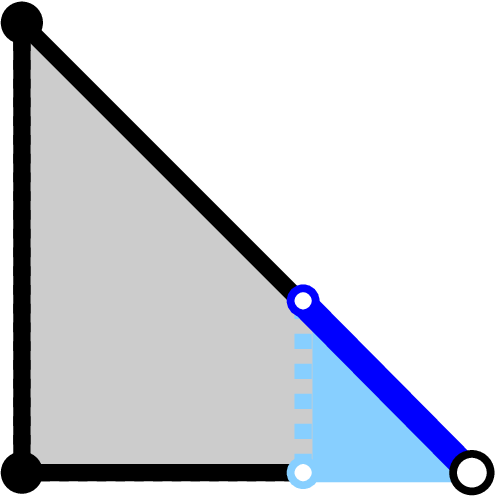}\end{array}
$$
constitute a $\sigma$-vicinity of~$\ee_1$ in~$X$ (i.e., the intersection
with~$X$ of a $\sigma$-vicinity of~$\ee_1$).  In addition, the bold blue
segment in the half-open hypotenuse~$H$ is a $\sigma$-vicinity
of~$\ee_1$~in~$H$.  Further examples expanding on this one can be found
in Example~\ref{e:sigma-nbd-cogen}.
\end{example}

\begin{lemma}\label{l:vicinity}
The $\sigma$-vicinities of points in~$\qrt$ form a base for a
topology~on~$\qrt$.
More strongly, the intersection of any finite set of
$\sigma$-vicinities (for perhaps different points in~$\qrt$) contains
a $\sigma$-vicinity of each point in their intersection.
\end{lemma}
\begin{proof}
The $\sigma$-vicinities of the points of~$\qrt$ cover $\qrt$ because
$\wt\aa$ lies in any of its $\sigma$-vicinities.  So it suffices to
prove the stronger claim, which by induction reduces to: the
intersection of any $\sigma$-vicinity $(\uu +
\qns)/\hspace{.2ex}\RR\tau$ of~$\wt\aa$ and any $\sigma$-vicinity
$(\vv + \qns)/\hspace{.2ex}\RR\tau$ of~$\wt\bb$ contains a
$\sigma$-vicinity of each point $\wt\cc$ in their intersection.  Pick
a coset representative $\cc \in \cQ$ for $\cc + \RR\tau = \wt\cc$.
Translating $\cc$ by a vector far inside of~$\tau^\circ$ (any
sufficiently high multiple of any vector in~$\tau^\circ$), if
necessary, assume that $\cc \in \uu + \qns$.  Perhaps translating
along~$\tau^\circ$ further, assume that $\cc \in \vv + \qns$ also.
Then $\uu + \qns$ contains the intersection with $\cc - \sigma^\circ$
of a neighborhood of~$\cc$ that is open in the usual topology, as does
$\vv + \qns$.  Any vector $\ww \in \cc - \sigma^\circ$ in the
intersection of these neighborhoods yields a $\sigma$-vicinity $(\ww +
\qns)/\hspace{.2ex}\RR\tau$ of~$\wt\cc$ contained in both of the given
$\sigma$-vicinities.
\end{proof}

\begin{defn}\label{d:sigma-topology}
The topology in Lemma~\ref{l:vicinity} is called the
\emph{$\nabs$-topology} on~$\qrt$.
\end{defn}

\begin{remark}\label{r:topology}
In this paper, all topological notions in real vector spaces---limit,
closure, neighborhood, and so on---refer to the usual topology unless
explicitly otherwise stated.  For example, Remark~\ref{r:nabs-closure}
refers to $\sigma$-closure, $\nabs$-closure, and $\nabs$-open
neighborhoods.
\end{remark}

\begin{remark}\label{r:compare-topologies}
The $\nabs$-topologies for various~$\sigma$ are more general than the
$\gamma$-topologies or Alexandrov topologies from
\cite{kashiwara-schapira2018} because the cone $\qns$ is not
necessarily closed.  Its non-closedness reflects directions that ought
to be thought of as inverted, and the image of $\qns$ in the collapse
modulo the inverted directions is closed, but the $\nabs$-topology is
needed on the vector space before this collapse.  The fact that the
image of $\qns$ modulo $\RR\sigma$ is closed implies that when $\sigma
= \tau$ in Definition~\ref{d:vicinity}, the $\nabs$-topology on~$\qrs$
is the Alexandrov topology: the only $\sigma$-vicinity of
the~coset~$\wt\aa =\nolinebreak \aa +\nolinebreak \RR\sigma
\in\nolinebreak \qrs$ is the principal upset $\wt\aa + (\qrs)_+$
itself.  Indeed, $\wt\aa + (\qrs)_+ = (\aa +
Q_+)/\hspace{.2ex}\RR\sigma = (\aa + \qns)/\hspace{.2ex}\RR\sigma =
(\uu + \qns)/\hspace{.2ex}\RR\sigma$ whenever $\aa - \uu \in
\sigma^\circ$ because $\aa - \uu \in \sigma^\circ \implies \wt\aa =
\wt\uu$.
\end{remark}

\begin{remark}\label{r:vicinity}
The strength of Lemma~\ref{l:vicinity} beyond providing a base for a
topology rests on a $\sigma$-vicinity of~$\wt\aa$ not being the same
as a basic $\nabs$-open set containing~$\wt\aa$.  Indeed, a
$\sigma$-vicinity $\uu + \qns$ is required to contain a representative
for~$\wt\aa$ that lies in the face $\uu + \sigma^\circ$, not merely
somewhere arbitrary in $\uu + \qns$.
\end{remark}

\begin{defn}\label{d:sigma-closure}
Fix faces~$\sigma \supseteq \tau$ of a real polyhedral group~$\cQ$.
\begin{enumerate}
\item\label{i:limit-point}%
A \emph{$\sigma$-limit point} of a subset $X \subseteq \qrt$ is a
point $\wt\aa \in \qrt$ that is a limit (in the usual topology) of
points in~$X$ each of which lies in a $\sigma$-vicinity of~$\wt\aa$.

\item\label{i:closure}%
The \emph{$\sigma$-closure} of $X \subseteq \qrt$ is the set of points
$\wt\aa \in \qrt$ such that $X$ has at least one point in every
$\sigma$-vicinity of~$\wt\aa$.
\end{enumerate}
\end{defn}

\begin{example}\label{e:sigma-limit}
In Example~\ref{e:sigma-vicinity}, the point~$\ee_1$ itself is a
$\sigma$-limit point of~$H$, because
$\alpha\ee_1 + \qns$ contains~$\ee_1$ as long as $\alpha < 1$.  Once
$\alpha = 1$, the intersection of $\alpha\ee_1 + \qns$ with~$X$ (and
hence with the half-open hypotenuse~$H$) is empty.  Thus the
$\sigma$-closure of~$H$ is the closed hypotenuse $H \cup \ee_1$.
\end{example}

\begin{remark}\label{r:nabs-closure}
The $\sigma$-closure of~$X$ equals its $\nabs$-closure, by which is
meant the closure of~$X$ in the $\nabs$-topology.  The reason: every
basic $\nabs$-open neighborhood of a point contains a
$\sigma$-vicinity of that point by Lemma~\ref{l:vicinity}.
\end{remark}

\begin{remark}\label{r:usual-closure}
The sets $X$ to which Definition~\ref{d:sigma-closure} is applied are
typically decomposed as finite unions of antichains (but see
Proposition~\ref{p:sigma-nbd-cogen} for an instance where this is not
the case).  Such sets ``cut across'' subsets of the form $(\uu +
\qns)/\hspace{.2ex}\RR\tau$, rather than being swallowed by them, so
$\sigma$-vicinities have a fighting chance of reflecting some concept
of closeness in antichains.  If $\sigma = \cQ_+$ and $\tau = \{\0\}$,
for example, and $X$ is an antichain in~$\cQ$, then a
$\sigma$-vicinity in~$X$ of a point $\aa \in X$ is simply a usual open
neighborhood of~$\aa$ in~$X$, because projection of~$X$ from the
origin~$\0$ to any hyperplane~$H$ transverse to the positive
cone~$\cQ_+$ is a homeomorphism from~$X$ to its image in~$H$.
\end{remark}

The concept of $\sigma$-vicinity provides a means to connect socles
(Definition~\ref{d:soct}) with support (Definition~\ref{d:support})
and primary decomposition (Definition~\ref{d:primDecomp}).

\begin{prop}\label{p:sigma-nbd-cogen}
In a real polyhedral group~$\cQ$, every cogenerator of a downset~$D$
along~$\tau$ with nadir~$\sigma$ has a $\sigma$-vicinity $\OO$ in~$D
\subseteq \cQ$ (so $\sigma \supseteq \{\0\}$ are the pair of faces in
Definition~\ref{d:sigma-closure}) such that $\kk[\OO] \subseteq
\kk[D]$ is $\tau$-coprimary and globally supported~on~$\tau$.
\end{prop}
\begin{proof}
Let $\aa$ be such a cogenerator of~$D$.  Suppose $\{\aa_k\}_{k\in\NN}
\subseteq \aa - \sigma^\circ \subseteq D$ is any sequence converging
to~$\aa$.  If $\aa_k$ is supported on a face~$\tau'$, then $\tau'
\supseteq \tau$ because $\aa \succeq \aa_k$ and $\aa$ remains a
cogenerator of the localization of~$D$ along~$\tau$ by
Proposition~\ref{p:local-vs-global}.  The same argument shows that the
$\sigma$-vicinity $\OO = \aa_k + \qns$ yields a submodule $\kk[\OO]
\subseteq \kk[D]$ such that $\kk[\OO] \into \kk[\OO]_\tau$.  The goal
is therefore to show that some $\aa_k$ is supported on~$\tau$, for
then~$\OO$ is supported on~$\tau$, as support can only
decrease~upon~going up~in~$\cQ$.

If each $\aa_k$ is supported on a face properly containing~$\tau$,
then, restricting to a subsequence if necessary, assume that it is the
same face~$\tau'$ for all~$k$.  (This uses the finiteness of the
number of faces.)  But then $\aa + \tau' = \lim_k(\aa_k + \tau')$ is
contained in $\dsd$ by Definition~\ref{d:atop-sigma}, contradicting
the fact that $\aa$ is supported on~$\tau$ in~$\dsd$.
\end{proof}

\begin{example}\label{e:sigma-nbd-cogen}
All three of the downsets
$$
\begin{array}{@{}*3{c@{\qquad\qquad}}c}
 \begin{array}{@{}c@{}}\includegraphics[height=30mm]{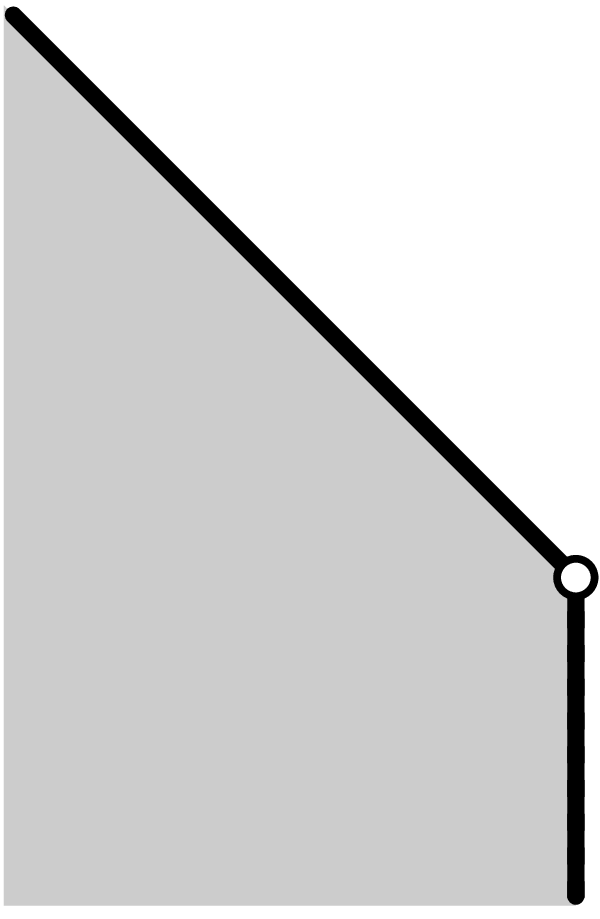}\end{array}
&\begin{array}{@{}c@{}}\includegraphics[height=30mm]{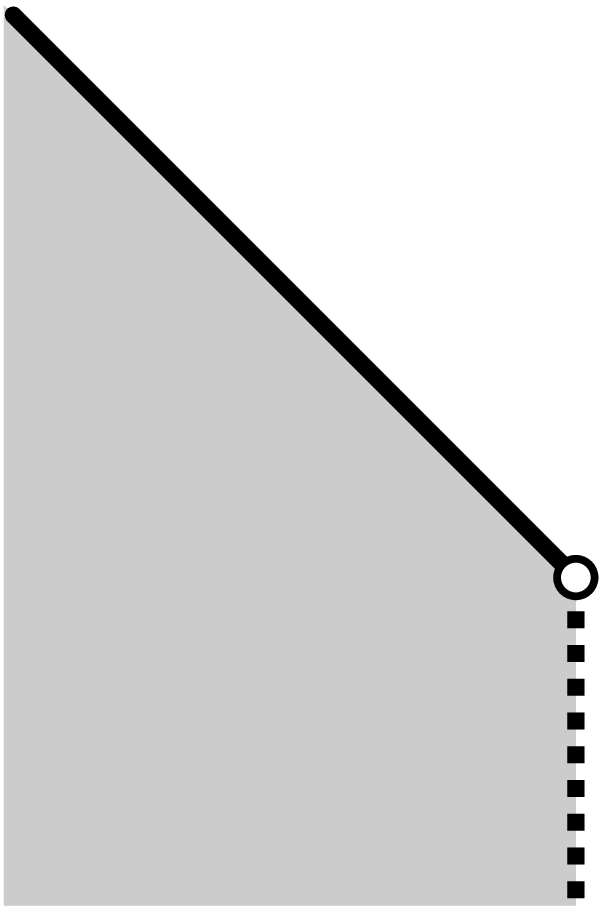}\end{array}
&\begin{array}{@{}c@{}}\includegraphics[height=30mm]{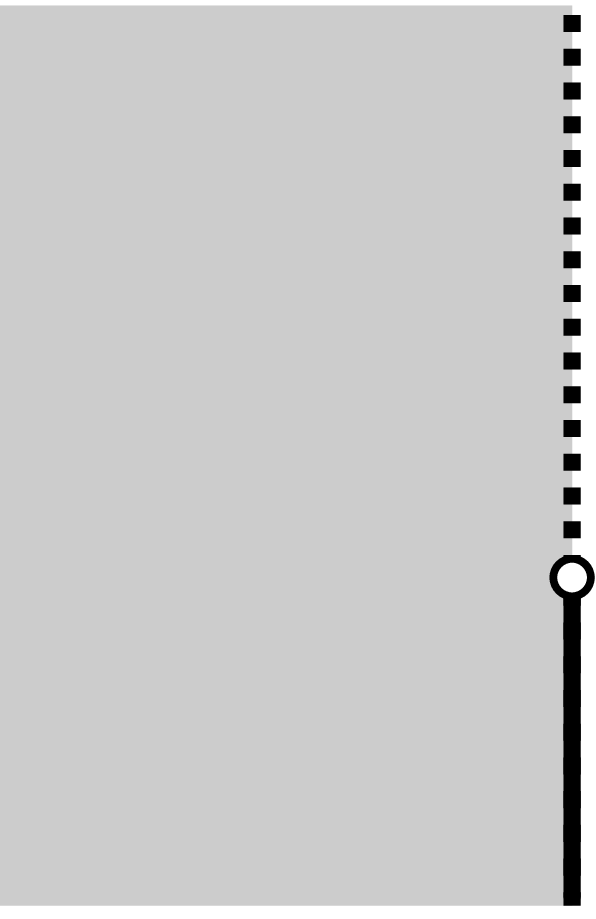}\end{array}
\\
D_1 & D_2 & D_3
\\[-.5ex]
\end{array}
$$
in~$\RR^2$ have a cogenerator at the open corner~$\aa = \ee_1$ (so the
top-left corner of each region is~$\ee_2$) along the face $\tau = \{\0\}$,
but their behaviors near~$\aa$ differ in character.  Write $\sigma_x$
and~$\sigma_y$ for the faces of~$\RR^2_+$ that are its horizontal and
vertical axes, respectively.
\begin{enumerate}
\item\label{i:pinch-0}%
Here $\aa$ has two nadirs: it is a cogenerator along~$\tau = \{\0\}$
for both~$\sigma_x$ and~$\sigma_y$ by
Proposition~\ref{p:downset-upper-closure} and Example~\ref{e:socc}.
The blue set in Example~\ref{e:sigma-vicinity} constitutes a
$\tau$-coprimary $\sigma_x$-vicinity of~$\aa$ globally supported
on~$\tau$, as in Proposition~\ref{p:sigma-nbd-cogen}.

\item\label{i:pinch-negative-y}%
Here $\aa$ has only the nadir~$\sigma_x$, because the downset has no
points in $\aa + \RR\sigma_y$ to take the closure of in
Lemma~\ref{l:downset-upper-closure}.  Again, the blue set in
Example~\ref{e:sigma-vicinity} constitutes the desired
$\sigma_x$-vicinity of~$\aa$.

\item\label{i:half-plane}%
Here $\aa$ has only the nadir~$\sigma_y$.  It is possible to compute
this directly, but it is more apropos to note that
Proposition~\ref{p:sigma-nbd-cogen} rules out $\sigma_x$ as a nadir.
Indeed, every $\sigma_x$-vicinity of~$\aa$ in~$D_3$ is an infinite
vertical strip.  None of these $\sigma_x$-vicinities are supported on
$\tau = \{\0\}$, since elements therein persist forever
along~$\sigma_y$.  In contrast, every choice of $\vv \in
\sigma_y^\circ$ yields a $\sigma_y$-vicinity $(-\vv + \RR^2_{\nabs_y})
\cap D_3$ supported on~$\{\0\}$; that is, the entire negative $y$-axis
is supported on~$\{\0\}$.
\end{enumerate}
\end{example}

Compare the following with Example~\ref{e:soct-k[tau]}; it is the
decisive more or less explicit calculation that justifies the general
theory of socles and provides its foundation.

\begin{cor}\label{c:soc(coprimary)}
Fix a face~$\tau$ of a real polyhedral group~$\cQ$ and a subquotient
$\cM$ of\/~$\kk[\cQ]$ that is $\tau$-coprimary and globally supported
on~$\tau$.  Then $\socp\cM = 0$ unless $\tau' = \tau$.
\end{cor}
\begin{proof}
Proposition~\ref{p:local-vs-global} implies that $\socp\cM = 0$ unless
$\tau' \supseteq \tau$ by definition of global support: localizing
along~$\tau'$ yields $\cM_{\tau'} = 0$ unless $\tau' \subseteq \tau$.
On the other hand, $\cM$ being a subquotient of~$\kk[\cQ]$ means that
$\cM \subseteq \kk[D]$ for some downset~$D$.  By left-exactness of
socles (Proposition~\ref{p:left-exact-tau}), every cogenerator
of~$\cM$ is a cogenerator of~$\kk[D]$.  Applying
Proposition~\ref{p:sigma-nbd-cogen} to any such cogenerator
along~$\tau'$ implies that $\tau = \tau'$, because no $\tau$-coprimary
module has a submodule supported on a face strictly contained
in~$\tau$.
\end{proof}

\begin{remark}\label{r:soc(coprimary)}
Remember that being $\tau$-coprimary does not require the whole module
to be globally supported on~$\tau$; only an essential submodule need
be globally supported on~$\tau$.  This occurs for the global support
on $\tau = \{\0\}$ in Example~\ref{e:global-support}, which is
strictly contained in the corresponding $\tau$-primary component from
Example~\ref{e:min-primary}.  Dually, being globally supported
on~$\tau$ allows for elements with support strictly contained
in~$\tau$.
\end{remark}

\begin{remark}\label{r:sigma-closure}
Definition~\ref{d:sigma-closure}.\ref{i:limit-point} stipulates no
condition the generators of the relevant $\sigma$-vicinities---the
vectors $\uu$ in Definition~\ref{d:vicinity}.  The a~priori difference
between being a $\sigma$-limit point and lying in the $\sigma$-closure
is hence that for $\sigma$-closure, the convergence is stipulated on
the generators of the $\sigma$-vicinities rather than on the points
of~$X$.  That said, the a~priori weaker (that is, more inclusive)
notion of $\sigma$-limit point is equivalent: the generators can be
forced to converge.
\end{remark}

\begin{prop}\label{p:aa_k}
If a sequence $\{\aa'_k\}_{k\in\NN}^{\,}$ in a real polyhedral
group~$\cQ$ has $\aa'_k \to \aa$ and $\aa'_k \in \aa_k + \qns$ for
some $\aa_k \in \aa - \sigma^\circ$, where $\sigma$ is a fixed face,
then it is possible to choose the elements~$\aa_k$ so that $\aa_k \to
\aa$.  Consequently, if~$\sigma \supseteq \tau$ then the
$\sigma$-closure of any set $X \subseteq \qrt$ equals the set of its
$\sigma$-limit points.
\end{prop}
\begin{proof}
Writing $\aa'_k = \aa - \vv_k + \zz_k$ with $\vv_k \in \RR\sigma$ and
$\zz_k \in \sigma^\perp$, the only relevant consequence of the
hypothesis $\aa'_k \in \aa_k + \qns$ is to force~$\zz_k$ to land
in~$\qrsp$ when projected to~$\qrs$.
Consider the set $Z \subseteq \sigma^\perp$ of vectors
in~$\sigma^\perp$ whose images in~$\qrs$ lie in~$\qrsp$ and have
magnitude $\leq 1$.  Let $V \subseteq \RR\sigma$ be the ball of
radius~$1$.  Find $\sss \in \sigma^\circ$ so that $\sss + V + Z
\subseteq \cQ_+$.  To see that such an~$\sss$ exists, first find
$\sss' \in \sigma^\circ$ so that $\sss' + Z \subseteq \cQ_+$.  To
construct~$\sss'$, rescale any element $\sss'' \in \sigma^\circ$; this
works because the projection of $(\sss'' + \sigma^\perp) \cap \cQ_+$
to~$\qrs$ contains a neighborhood of~$\0$ in~$\qrsp$.  Then observe
that the condition $\sss' + Z \subseteq \cQ_+$ remains true after
adding any element of~$\sigma$ to~$\sss'$.  In particular,
construct~$\sss$ by adding the center of any ball in~$\sigma^\circ$ of
radius~$1$, which exists because $\sigma^\circ$ is nonempty, open
in~$\RR\sigma$, and closed under positive scaling.

Having fixed~$\sss$ with $\sss + V + Z \subseteq \cQ_+$, set $\aa_k =
\aa - 2\epsilon_k\sss$, where $\epsilon_k = |\aa'_k - \aa|$.  The
reason for this choice of~$\epsilon_k$ is that $2\epsilon_k \to 0$
(because $\aa_k' \to \aa$) and $\epsilon_k$ bounds the magnitudes
of~$\vv_k$ and~$\zz_k$.  This latter condition implies $\aa'_k = \aa -
\vv_k + \zz_k \in \aa + \epsilon_k V + \epsilon_k Z = (\aa -
2\epsilon_k\sss) + \epsilon_k\sss + (\epsilon_k\sss + \epsilon_k V +
\epsilon_k Z) \subseteq \aa_k + \sigma^\circ + \cQ_+ = \aa_k + \qns$
by Lemma~\ref{l:<<}.

The claim involving~$\tau$ follows, when $\tau = \{\0\}$, from
Proposition~\ref{p:<<}: it implies that each element $\bb \in \aa -
\sigma^\circ$ precedes some~$\aa_k$, and hence the $\sigma$-vicinity
generated by~$\bb$ contains~$\aa'_k$.  The case of arbitrary~$\tau$
reduces to $\tau = \{\0\}$ by working modulo~$\RR\tau$.
\end{proof}

\subsection{Dense cogeneration of downsets}\label{b:dense-downsets}\mbox{}

\medskip
\noindent
The subfunctor version of density in socles for modules requires first
a geometric version for downsets.  For geometric intuition, it is
useful to again recall Lemma~\ref{l:<<}, which says that $\qns =
\sigma^\circ + \cQ_+$.  Thus $\aa - \qns$ is the ``coprincipal''
downset with apex~$\aa$ and shape~$\nabs$.  Adding $\tau$ to get
$\aa+\tau-\qns$ takes the union of these downsets along $\aa + \tau$.
The downset thus constructed is preserved under translation
by~$\RR\tau$, since $-\tau \subseteq -\cQ_+$.

\begin{lemma}\label{l:aa+tau-qns}
Let $\cQ$ be a real polyhedral group with faces $\sigma \supseteq
\tau$.  Write $\wt\aa \in \qrt$ for the coset $\aa + \RR\tau$
containing $\aa \in \cQ$.  Then $\aa+\tau-\qns = \aa'+\tau-\qns$ for
all $\aa' \in \wt\aa$.  Hence there is only one downset $\aa+\tau-\qns
= \wt\aa-\qns$ per coset $\wt\aa = \aa + \RR\tau$.
\end{lemma}
\begin{proof}
Using Lemma~\ref{l:<<} to write $\aa + \tau - \qns = \aa + \tau -
\sigma^\circ - \cQ_+$, this set is translation-invariant
along~$\RR\tau$ because $-\cQ_+$ contains~$-\tau$.
\end{proof}

\begin{example}\label{e:aa+tau-qns}
Let $\cQ = \RR^3_+$ with $\tau = y$-axis and $\sigma = xy$-plane.  View $D = \aa +
\tau - \qns$ from behind the $xz$-plane at positive height.  Then $D$ is a
quadrant in~$\RR^3$ (meaning that it occupies one quarter of~$\RR^3$)
whose horizontal top boundary contains the relatively open negative
half-plane bounded by the line $\aa + \RR\tau$ but does not contain the
line~itself:
$$
  \psfrag{x}{\footnotesize$x$}
  \psfrag{y}{\footnotesize$y$}
  \psfrag{z}{\footnotesize$z$}
  \psfrag{a}{\footnotesize$\aa$}
  \psfrag{s}{\footnotesize$\sigma$}
  \psfrag{t}{\footnotesize$\tau$}
  \includegraphics[height=30ex]{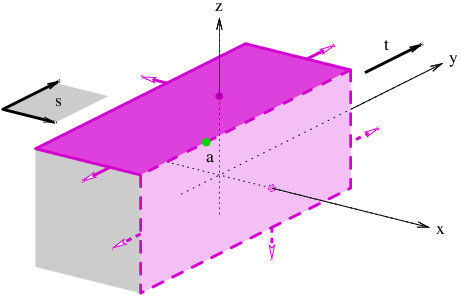}
\vspace{-1ex}%
$$
The difference bewteen $D$ and its closure is the missing vertical
half-plane beneath the line $\aa + \RR\tau$ parallel to the $y$-axis.  As per
Lemma~\ref{l:aa+tau-qns}, nothing about~$D$ changes if $\aa$ is replaced
by another point along the line $\aa + \RR\tau$.
\end{example}

The question is which downsets from Lemma~\ref{l:aa+tau-qns} must
appear in any expression of a given downset as a union of downsets of
that form.  It suffices, for instance, by Theorem~\ref{t:divides}, to
use the set of all cogenerators along~$\tau$ with nadir~$\sigma$, for
all pairs $\sigma \supseteq \tau$.  But using all of these
cogenerators is usually unnecessary, in part because of the redundancy
resulting from Lemma~\ref{l:aa+tau-qns} but in part because of density
considerations.  The main decomposition theorem for downsets isolates
the precise condition on a set $\ats[\,] \subseteq \cQ$ of
cogenerators of a downset~$D$ along a face~$\tau$ to suffice.  In the
statement, the union separates cogenerators according to their
nadirs~$\sigma$ because the component $\aa + \tau - \qns$ is
constructed using~$\sigma$; but in the end the condition on~$\ats[\,]$
relies on the $\sigma$-closure of all of~$\ats[\,]$ for each $\sigma
\supseteq \tau$, not merely the closure of the set of cogenerators
in~$\ats[\,]$ that happen to have nadir~$\sigma$.

\begin{thm}\label{t:downset=union}
Let $\ats \subseteq \cQ$ be a set of cogenerators of a downset~$D$ in
a real polyhedral group~$\cQ$ along a face~$\tau$ with nadir~$\sigma$
for each $\sigma \in \nabt$, and let $\ats[\,]\! = \bigcup_{\sigma
\supseteq \tau} \ats$.  Then
\begin{equation*}
D
  =
\bigcup_{\substack{\hbox{\footnotesize\rm faces $\sigma,\tau$}\\
		   \hbox{\footnotesize\rm with $\sigma\supseteq\tau$}}}
\bigcup_{\raisebox{-.7ex}{\footnotesize$\ \aa\!\in\!\!\ats$}} \aa + \tau - \qns
\end{equation*}
if for each face~$\tau$, the $\sigma$-closure of the image
of~$\ats[\,]\!$~in~$\qrt$ for each $\sigma \supseteq \tau$ contains
the projection modulo~$\RR\tau$ of every cogenerator of~$D$
along~$\tau$ with nadir~$\sigma$.
\end{thm}

\begin{proof}[Proof of Theorem~\ref{t:downset=union}]
Theorem~\ref{t:divides} is equivalent to the desired result in the
case that every $\ats$ is the set of all cogenerators of~$D$
along~$\tau$ with nadir~$\sigma$, by Example~\ref{e:soct} and
Remark~\ref{r:nabt}.  Hence it suffices to show~that
$$
  \bigcup_{\raisebox{-.6ex}{$\scriptstyle\sigma'\supseteq\tau$}}
  \bigcup_{\ \aa'\in\ats[\sigma']} \aa' + \tau - \qnp
  \supseteq
  \aa + \tau - \qns
$$
for any fixed cogenerator~$\aa$ of~$D$ along~$\tau$ with
nadir~$\sigma$.  In fact, by definition of $\sigma$-limit point, it is
enough to show that
$$
  \bigcup_{k=1}^\infty \aa'_k + \tau - \qnk
  \supseteq
  \aa + \tau - \qns,
$$
where $\{\aa'_k\}_{k\in\NN}^{\,}$ is a sequence of elements 
of~$\ats[\,]$ such that
\begin{itemize}
\item%
$\aa_k$ lands in a $\sigma$-vicinity of the image~$\wt\aa$ of~$\aa$
when projected to~$\qrt$, and
\item%
these images $\wt\aa'_k$ converge to~$\wt\aa$ in~$\qrt$
\end{itemize}
and $\sigma_k$ is a nadir of the cogenerator $\aa'_k$ along~$\tau$.

Note that there is something to prove even when $\sigma = \tau$
because $\ats[\tau]$ only needs to have at least one closed
cogenerator in~$\cQ$ for each closed socle degree in~$\qrt$, whereas
the set of all closed cogenerators along~$\tau$ mapping to a given
socle degree might not be a single translate of~$\tau$.  On the other
hand, $\tau - \qnt = \tau - \tau^\circ - \cQ_+\!$~by Lemma~\ref{l:<<},
and this is just $\RR\tau - \cQ_+$.  Hence $\aa + \tau - \qnt$
contains the translate of the negative cone~$-\cQ_+$ at every point
mapping to $\wt\aa$, cogenerator or otherwise, completing the
case~$\sigma = \tau$.

For general $\sigma \supseteq \tau$, Lemma~\ref{l:aa+tau-qns} reduces
the question to the quotient~$\qrt$, where it becomes
$$
  \bigcup_{k=1}^\infty \wt\aa'_k - \qnk/\tau
  \supseteq
  \wt\aa - \qns/\tau.
$$
But as $\qnk/\tau = \sigma_k^\circ/\tau + (\qrt)_+$ by
Lemma~\ref{l:<<}, it does no harm (and helps the notation) to assume
that $\tau = \{\0\}$.  The desired statement is now
$$
  \bigcup_{k=1}^\infty \aa'_k - \qnk
  \supseteq
  \aa - \qns,
$$
the hypotheses being those of Proposition~\ref{p:aa_k}.  The proof is
completed by applying Proposition~\ref{p:<<} to the
sequence~$\{\aa_k\}_{k\in\NN}$ produced by Proposition~\ref{p:aa_k},
noting that $\aa_k - \cQ_+ \subseteq \aa'_k - \qnk$ as soon as
$\aa_k \in \aa'_k - \qnk$, because $\aa'_k - \qnk$ is a
downset.
\end{proof}

\begin{example}\label{e:downset=union}
In Example~\ref{e:sigma-nbd-cogen} $D_1$ and $D_2$ only have
cogenerators along $\tau = \{\0\}$ because their degree sets
in~$\RR^2$ are bounded above.  In contrast, $D_3$ also has
cogenerators along $\tau = \sigma_y$.  Here is what
Theorem~\ref{t:downset=union} has to say in all three cases.
\begin{enumerate}
\item\label{i:pinch-0'}%
All cogenerators~$\bb$ of~$D_1$ that lie on the open diagonal ray are
closed, meaning that their nadir is $\sigma = \{\0\}$.  All of these
cogenerators are forced to contribute a term $\bb - \cQ$ to the display
in Theorem~\ref{t:downset=union}.  The only other cogenerator of~$D_1$
resides at~$\aa = \ee_1$.  The shape of~$D$ at~$\aa$
(Proposition~\ref{p:shape}) is the cocomplex
(Definition~\ref{d:cocomplex}) $\nda$ with minimal faces~$\sigma_x$ and
$\sigma_y$, the two rays of~$\cQ_+$, both of which are nadirs for~$\aa$.  The
term $\aa - \qns[_x]$ that ``wants'' to be present due to
Theorem~\ref{t:divides} can in fact be omitted due to
Theorem~\ref{t:downset=union} because $\aa$ is a $\sigma_x$-limit point
of~$H$ in Example~\ref{e:sigma-limit}: the triangle there shares its
geometry with~$D_1$ on the relevant set near~$\aa$, namely the $x$-axis
and above.  However, Theorem~\ref{t:downset=union} does force a term
$\aa - \qns[_y]$ in its displayed union because every $\sigma_y$-vicinity
of~$\aa$ in~$D_1$ contains exactly one cogenerator, namely~$\aa$~itself.

\item\label{i:pinch-negative-y'}%
$D_2$ has the same cogenerators as~$D_1$ except that the
cogenerator~$\aa$ does not have~$\sigma_y$ as a nadir.  It is instructive to
note that if the term $\aa - \qns[_y]$ is omitted from the union for
$D_1$ described in the previous part of this Example, then what
results is a union that equals~$D_2$.  That is, $D_2$ equals the union
of the closed negative quadrants hanging from the open diagonal ray.
Abstractly, the ability to omit~$\aa$ as a cogenerator is because
(i)~$\aa$ is a $\sigma_x$-limit point of the open diagonal ray, and
(ii)~$D_2$ is missing precisely the negative $y$-axis that caused $\aa$
to be a cogenerator with nadir~$\sigma_y$.

\item\label{i:half-plane'}%
$D_3$ has precisely one cogenerator along~$\tau = 0$, namely $\aa$
with nadir~$\sigma_y$ by Example~\ref{e:sigma-nbd-cogen}.3.  The only
other nonzero socle of~$\kk[D_3]$ occurs along~$\tau = \sigma_y$ with
nadir $\sigma = \RR^2_+$.  The union in Theorem~\ref{t:downset=union}
is $D_3 = (\aa - \qns[_y]) \cup (\aa + \sigma_y - \cQ^\circ)$, where
$\cQ^\circ = \cQ_{\RR^2_+}$ is the interior of~$\cQ$, so $\aa +
\sigma_y - \cQ^\circ$ is an open left half-plane.
\end{enumerate}
\end{example}

\begin{example}\label{e:open-cogenerators}
The curve atop each of the following two downsets in~$\RR^2$ is a
hyperbola.
$$
\psfrag{x}{\footnotesize$x$}
\psfrag{y}{\footnotesize$y$}
\begin{array}{@{}*1{c@{\qquad\qquad\qquad}}c}
\\[-3.8ex]
\begin{array}{@{}c@{}}\includegraphics[height=30mm]{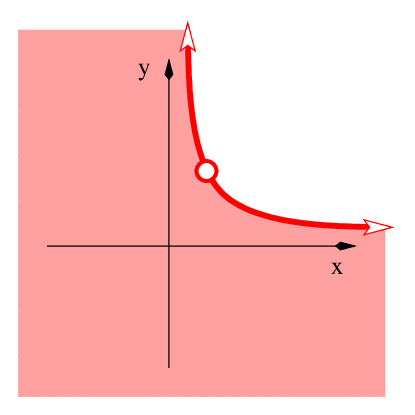}\end{array}
&
\begin{array}{@{}c@{}}\includegraphics[height=30mm]{half-open-hyperbola}\end{array}
\\
D_1 & D_2
\\[-.5ex]
\end{array}
$$
\begin{enumerate}
\item\label{i:plucked-hyperbola}%
Plucking out a single point from the hyperbola has an odd effect.  At
the frontier point Definition~\ref{d:soct} detects two open
cogenerators, with different nadirs as in Example~\ref{e:trouble-soc},
but they are redundant: $D_{\hbox{\tiny\ref{i:plucked-hyperbola}}}$
equals the union of the closed negative orthants cogenerated by the
points along the rest of the hyperbola.

\item\label{i:half-open-hyperbola}%
The ability to omit cogenerators is even more striking upon deleting
an interval from the hyperbola, instead of a single point, to get the
downset $D_{\hbox{\tiny\ref{i:half-open-hyperbola}}}$.  Along the
deleted curve, Definition~\ref{d:soct} detects cogenerators of the
same shape as those for~$D_{\hbox{\tiny\ref{i:plucked-hyperbola}}}$.
Hence $D_{\hbox{\tiny\ref{i:half-open-hyperbola}}}$ is the union of
coprincipal downsets of these shapes along the deleted curve together
with closed coprincipal downsets along the rest of the hyperbola.
However, any finite number of the coprincipal downsets along the
deleted curve can be omitted, as can be checked directly.  In fact,
any subset of them whose complement is dense in the deleted curve can
be omitted by Theorem~\ref{t:downset=union}.  Note that a closed
negative orthant is required at the lower endpoint of the deleted
curve, because the endpoint has not been deleted, whereas the
cogenerator at the upper endpoint of the deleted curve can always be
omitted due to open negative orthants hanging from points along the
hyperbola just below~it.
\end{enumerate}
\end{example}

Here is a restatement of Theorem~\ref{t:divides}, phrased as a special
case of Theorem~\ref{t:downset=union}, in terms of coprincipal
downsets via Lemma~\ref{l:aa+tau-qns}.

\begin{cor}\label{c:downset=union}
Every downset~$D$ in a real polyhedral group~$\cQ$ is the union of the
\emph{coprincipal downsets} $\wt\aa - \qns$ indexed by the degrees
$\wt\aa \in \qrt$ of socle elements of\/~$\kk[D]$ along all
faces~$\tau$ with all nadirs~$\sigma$:
\begin{equation*}
  D
  =
\bigcup_{\substack{\hbox{\footnotesize\rm faces $\sigma,\tau$}\\
		   \hbox{\footnotesize\rm with $\sigma\supseteq\tau$}}}
\bigcup_{\raisebox{-.7ex}{\footnotesize$\ \ \wt\aa\in\deg_{\qrt}\soct[\sigma]\kk[D]$}}
\wt\aa  - \qns.
\end{equation*}
\end{cor}

\begin{remark}\label{r:irred-decomp}
Theorem~\ref{t:downset=union} is an analogue for real polyhedral
groups of the fact that monomial ideals in affine semigroup rings
admit unique irredundant irreducible decompositions
\cite[Corollary~11.5]{cca}.  To see the analogy, note that expressing
a downset as a union is the same as expressing its complementary upset
as an intersection.  In Theorem~\ref{t:downset=union} the union is
neither unique nor irredundant, but only in the sense that a
topological space can have many dense subsets, each of which can
usually be made smaller by omitting some points.  The union in
Corollary~\ref{c:downset=union} is still canonical, though redundant
in a predictable manner, namely that of Theorem~\ref{t:downset=union}.
The analogy in this remark is not the tightest possible; see
Section~\ref{b:intervals} for the true analogy.
\end{remark}

\begin{remark}\label{r:isw}
In the case of $\cQ = \RR^n$ with componentwise partial order,
Ingebretson and Sather-Wagstaff characterized the downsets in $\cQ_+$
that admit decompositions as in Theorem~\ref{t:downset=union} with
finitely many terms
\cite[Theorem~4.12]{ingebretson--sather-wagstaff2013}: they are the
ones whose complementary upsets have finitely many generators.
The new aspects here, for arbitrary downsets, are the related notions
of minimality and density of cogenerating sets.
Theorem~\ref{t:downset=union} implies the characterization of
irreducible downsets
\cite[Theorem~3.9]{ingebretson--sather-wagstaff2013} because
irreducible decompositions are assumed finite~there.
\end{remark}

The final result in this subsection is applied in the proof of
Theorem~\ref{t:dense-subfunctor}.

\begin{cor}\label{c:sigma-vicinity}
Fix a cogenerator~$\aa$ of a downset~$D$ along a face~$\tau$ with
nadir~$\sigma$ in a real polyhedral group.  If\/ $\bb \in D$ and $\bb
\preceq \aa$, then the image $\wt\aa$ of~$\aa$ in~$\qrt$ has a
$\sigma$-vicinity~$\OO$ in $\deg_{\qrt}\soct\kk[D]$ such that $\wt\bb
\preceq \wt\aa'$ for all\/ $\wt\aa' \in \OO$.
\end{cor}
\begin{proof}
Assume $\bb \in D$ and $\bb \preceq \aa$.
Theorem~\ref{t:downset=union} implies that $\bb \in \aa + \tau - \qns
= \aa + \RR\tau - \sigma^\circ - \cQ_+$.  Therefore $\aa + \RR\tau =
\bb + \RR\tau + \sss + \qq$ for some $\sss \in \sigma^\circ$ and $\qq
\in \cQ_+$.  The $\sigma$-vicinity in question is $(\wt\bb + \wt\qq +
\qns) \cap \deg_{\qrt}\soct\kk[D]$.
\end{proof}

\subsection{Dense subfunctors of socles}\label{b:subfunctors}\mbox{}

\medskip
\noindent
In general, a subfunctor $\Phi: \cA \to \cB$ of a covariant functor
$\Psi: \cA \to \cB$ is a natural transformation $\Phi \to \Psi$ such
that $\Phi(A) \subseteq \Psi(A)$ for all objects $A \in \cA$
\cite[Chapter~III]{eilenberg-maclane1945};
denote this by $\Phi \subseteq \Psi$.  (This notation assumes that the
objects of~$\cB$ are sets, which they are here; in general, $\Phi(A)
\to \Psi(A)$ should be monic.)

\begin{defn}\label{d:dense-subfunctor}
A subfunctor $\cst = \bigoplus_{\sigma \in \nabt} \cst[\sigma]
\subseteq \soct$ from modules over~$\cQ$ to modules over $\qrt
\times\nolinebreak \nabt$ is \emph{dense} if the $\sigma$-closure of
$\deg_{\qrt}\cst\kk[D]$ contains $\deg\soct[\sigma]\kk[D]$ for all
faces $\sigma \supseteq \tau$ and downsets $D \subseteq \cQ$.  An
\emph{$\cS$-cogenerator} of a $\cQ$-module~$\cM$ is a cogenerator
of~$\cM$ along some face~$\tau$ whose image in $\soct\cM$ lies
in~$\cst\cM$.
\end{defn}

\begin{thm}\label{t:dense-subfunctor}
Fix subfunctors $\cst \subseteq \soct$ for all faces~$\tau$ of a real
polyhedral group.  Theorem~\ref{t:injection} holds with $\cS$ in place
of $\soc$ if and only if $\cst$ is dense in~$\soct$ for all~$\tau$.
\end{thm}
\begin{proof}
Every subfunctor of any left-exact functor takes injections to
injections; therefore Theorem~\ref{t:injection}.\ref{i:phi=>soct}
holds for any subfunctor of~$\soct$ by
Proposition~\ref{p:left-exact-tau}.  The content is that
Theorem~\ref{t:injection}.\ref{i:soct=>phi} is equivalent to density
of $\cst$ in $\soct$ for all~$\tau$.

First suppose that $\cst$ is dense in $\soct$ for all~$\tau$.  It
suffices to show that each homogeneous element $y \in \cM$ divides
some $\cS$-cogenerator~$s$, for then $\phi(y) \neq 0$ whenever
$\cst\phi(\wt s) \neq 0$, where $\wt s$ is the image of~$s$
in~$\cst\cM \subseteq \soct\cM$.  There is no harm in assuming that
$\cM$ is a submodule of its downset hull: $\cM \subseteq E =
\bigoplus_{j=1}^k E_j$.  Theorem~\ref{t:injection} produces a
cogenerator~$x$ of~$E$ that is divisible by~$y$, and $x$ is
automatically a cogenerator of~$\cM$---say $x \in \ds[\tau]\cM
\subseteq \ds[\tau]E$---because $y$~divides~$x$.  Write $x =
\sum_{j=1}^k x_j \in \ds[\tau]E = \bigoplus_{j=1}^k \ds[\tau]E_j$.
For any index~$j$ such that $x_j \neq 0$,
Corollary~\ref{c:sigma-vicinity} and the density hypothesis yields a
$\sigma$-vicinity of~$\wt\aa$ containing a socle element~$\wt s_j$
mapped to by an $\cS$-cogenerator~$s_j$ that is divisible by~$y_j$.
An $\cS$-cogenerator~$s$ of~$\cM$ divisible by~$y$ is constructed
from~$s_j$ just as an ordinary cogenerator is constructed from~$s_j$
in the
proof of Theorem~\ref{t:injection}.

Now suppose that $\cst[\,]$ is not dense in~$\soct$ for some
face~$\tau$, so some downset $D \subseteq\nolinebreak \cQ$ has a
cogenerator $\aa \in \cQ$ whose image $\wt\aa \in
\deg\soct[\sigma]\kk[D] \subseteq \qrt$ has a $\sigma$-vicinity
$\deg_{\qrt}\soct\kk[D] \cap (\uu + \qns)/\hspace{.2ex}\RR\tau$ devoid
of images of $\cS$-cogenerators along~$\tau$.  Appealing to
Lemma~\ref{l:vicinity}, the intersection of $\uu + \qns$ with a
$\sigma$-vicinity~$\OO$ of~$\aa$ in~$D$ from
Proposition~\ref{p:sigma-nbd-cogen} contains another
$\sigma$-vicinity~$\OO'$ of~$\aa$ that still satisfies the conclusion
of Proposition~\ref{p:sigma-nbd-cogen} because every submodule of any
$\tau$-coprimary module globally supported on~$\tau$ is also
$\tau$-coprimary and globally supported on~$\tau$.
The injection $\kk[\OO'] \into \kk[D]$ yields an injection
$\cst\kk[\OO'] \into \cst\kk[D]$, but by construction $\cst\kk[D]$
vanishes in all degrees from $\deg_{\qrt}\soct\kk[\OO']$, so
$\cst\kk[\OO'] = 0$.  On the other hand, $\socp\kk[\OO'] = 0$ for
$\tau' \neq \tau$ by Corollary~\ref{c:soc(coprimary)}, so the
subfunctor $\csp$ vanishes on $\kk[\OO']$ for all faces~$\tau'$.
Consequently, applying~$\csp$ to the homomorphism $\phi: \kk[\OO'] \to
0$ yields an injection $0 \into 0$ for all faces~$\tau'$ even though
$\phi$ is not injective.
\end{proof}

\section{Essential submodules via density in socles}\label{s:density}

Given a closed cogenerator of~$\cM$, there is an obvious submodule
of~$\cM$ containing the socle element, namely the submodule generated
by the cogenerator itself.  In contrast, open cogenerators are not
elements of~$\cM$ itself.  How, then, do cogenerators detect
injectivity in Theorem~\ref{t:injection}?  Each cogenerator must still
yield a submodule to witness the injectivity, because injectivity
means there is no actual submodule of~$\cM$ that goes to~$0$.  The
cogenerator merely indicates the presence of such a submodule, rather
than being an element of it.  This section reconstructs an honest
submodule around each cogenerator.  It requires much of the theory in
earlier sections.  As a consequence, a submodule $\cM' \subseteq \cM$
is an essential submodule precisely when the socle of~$\cM'$ is dense
in that of~$\cM$ (Theorem~\ref{t:essential-submodule}), the goal of
this section and a main result of the paper.

More precisely, the $\sigma$-vicinities in
Proposition~\ref{p:sigma-nbd-cogen} transfer cogenerators back into
honest submodules; they are, in that sense, the reverse of
Definition~\ref{d:atop-sigma}.  In fact this transference of
cogenerators into submodules works not merely for indicator quotients
but for arbitrary modules with finite downset hulls, as in
Theorem~\ref{t:essential-submodule}.  The key is the generalization of
$\sigma$-vicinities to arbitrary downset-finite modules.

\begin{defn}\label{d:nearby}
Fix a module~$\cM$ over a real polyhedral group~$\cQ$ and a face
$\tau$.
\begin{enumerate}
\item\label{i:nearby}%
A $\sigma$-divisor (Definition~\ref{d:divides}) $z \in \cM$ of a
cogenerator of~$\cM$ along~$\tau$ with nadir~$\sigma$
(Definition~\ref{d:soct}) is \emph{nearby} if $z$ is globally
supported on~$\tau$ (Definition~\ref{d:support}).

\item\label{i:neighborhood-in-M}%
A \emph{$\sigma$-vicinity in~$\cM$} of a cogenerator $s \in
\ds[\tau]\cM$ is a submodule of~$\cM$ generated by a nearby
$\sigma$-divisor of~$s$.

\item\label{i:neighborhood-in-soc}%
A \emph{neighborhood in~$\soct\cM$} of a homogeneous socle element
$\wt s \in \soct[\sigma]\cM \subseteq (\ds[\tau]\cM)/\tau$ is $\soct\cN$ for a $\sigma$-vicinity
$\cN$ in~$\cM$ of a cogenerator in $\ds[\tau]\cM$ that maps to~$\wt s$.

\item\label{i:dense}%
An inclusion $S_\tau \subseteq \soct\cM$ of
$(\qrt\times\nabt)$-modules is \emph{dense} if for all $\sigma
\supseteq \tau$, every neighborhood of every homogeneous element
of~$\soct[\sigma]\cM$ intersects $S_\tau$ nontrivially.
\end{enumerate}
\end{defn}

\begin{example}\label{e:nearby}
The convex hull of $\0,\ee_1,\ee_2$ in $\RR^2$ but with the first standard
basis vector $\ee_1$ removed (see Example~\ref{e:sigma-vicinity})
defines a subquotient~$\cM$ of~$\kk[\RR^2]$.  It has submodule~$\cM'$ that
is the indicator module for the same triangle but with the entire
$x$-axis removed:
$$
\begin{array}{@{}*6{c@{}}c}
\\[-2.8ex]
 \begin{array}{@{}c@{}}\includegraphics[height=20mm]{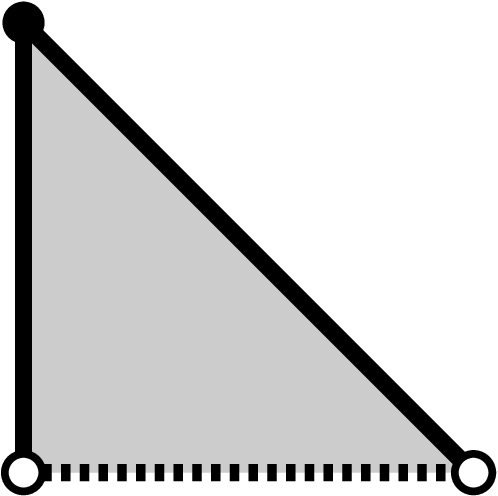}\end{array}
&\quad \subseteq \quad\ \
&\begin{array}{@{}c@{}}\includegraphics[height=20mm]{ambient-essential}\end{array}
&\quad \text{with $\sigma$-vicinities} \quad\ \
&\begin{array}{@{}c@{}}\includegraphics[height=20mm]{sigma-vicinity}\end{array}
&\quad\ \ \text{of~$\ee_1$.}
\end{array}
$$
All of the cogenerators of both modules occur along the face $\tau =
\{\0\}$ because both modules are globally supported on~$\{\0\}$.
However, the ambient module---but not the submodule---has a cogenerator $z
\in \ds[\tau]\cM$ with nadir $\sigma = x$-axis$_+$ of degree~$\ee_1$.  A typical
$\sigma$-vicinity of~$z$ in~$\cM$ is shaded in light blue (in fact, $\cM$
itself is also a $\sigma$-vicinity of~$z$).

As the face~$\tau$ here is the origin, the quotient-restriction
modulo~$\tau$ in Definition~\ref{d:nearby}.\ref{i:neighborhood-in-soc}
has no effect, so the relatively open bold blue segment along the
hypotenuse degree set of the corresponding vicinity of~$\wt z \in
\soct[\sigma]\cM$.  Every such vicinity contains socle elements
in~$\soct[\sigma]\cM'$.  Since the socle of~$M'$ agrees with the socle
of~$M$ away from $\ee_1$, it follows that $\soct\cM' \subseteq
\soct\cM$ is dense.
\end{example}

\begin{lemma}\label{l:nearby}
Every neighborhood in~$\cM$ of every homogeneous element
in~$\soct[\sigma]\cM$ is a $\tau$-coprimary submodule of~$\cM\!$
globally supported on~$\tau$.
\end{lemma}
\begin{proof}
Let $z$ be a nearby $\sigma$-divisor of a cogenerator $s \in
\ds[\tau]\cM$.  Let $x$ be a homogeneous multiple of~$z$.  That~$x$ is
supported on~$\tau$ is automatic from the hypothesis that $z$ is
supported on~$\tau$.  To say that the submodule~$\<z\>$ generated
by~$z$ is $\tau$-coprimary means, given that it is supported
on~$\tau$, that $\<z\>$ is a submodule of its localization
along~$\tau$.  But $s$ remains a cogenerator after localizing
along~$\tau$ by Proposition~\ref{p:local-vs-global}, so~$x$ must
remain nonzero because it still divides~$s$ after localizing.
\end{proof}

\begin{prop}\label{p:nearby}
Fix a downset-finite module $\cM$ over a real polyhedral group with
faces $\sigma \supseteq \tau$.  Every cogenerator in $\ds[\tau]\cM$
has a $\sigma$-vicinity in~$\cM$.
\end{prop}
\begin{proof}
Let $s \in \ds[\tau]\cM$ be the cogenerator, and let its degree be
$\deg_\cQ(s) = \aa \in \cQ$.  Choose a downset hull $\cM \into E =
\bigoplus_{j=1}^k E_j$, so $E_j = \kk[D_j]$ for a downset~$D_j$.
Express $s = s_1 + \dots + s_k \in \ds[\tau]E = \bigoplus_{j=1}^k
\ds[\tau]E_j$.  \mbox{Proposition}~\ref{p:sigma-nbd-cogen} produces~a
$\sigma$-vicinity $\OO_j$ of~$\aa$ in~$\cQ$, for each index~$j$, such
that $\kk[\OO_j \cap D_j]$ is a $\sigma$-vicinity in~$E_j$ of the
image $\wt s_j \in \soct[\sigma]E_j$.  Lemma~\ref{l:vicinity} then
yields a single $\sigma$-vicinity $\OO = \aa - \vv + \cQ_+$ of~$\aa$
in~$\cQ$ that lies in the intersection $\bigcap_{j=1}^k \OO_j$.  The
cogenerator $s \in \ds[\tau]$ is a direct limit over $\aa -
\sigma^\circ$; since $\OO$ contains a neighborhood (in the usual
topology) of~$\aa$ in~$\sigma^\circ$, some element $z \in \cM$ with
degree in~$\OO$ is a $\sigma$-divisor of~$s$.  This element~$z$ is
nearby~$s$~by~construction.
\end{proof}

The following generalization of Corollary~\ref{c:soc(coprimary)} to
modules with finite downset hulls is again the decisive computation.

\begin{cor}\label{c:soc(coprimary)'}
Fix a downset-finite $\tau$-coprimary $\cQ$-module~$\cM$ globally
supported on a face~$\tau$ of a real polyhedral group~$\cQ$.  Then
$\socp\cM = 0$ unless~$\tau' = \tau$.
\end{cor}
\begin{proof}
Proposition~\ref{p:local-vs-global} implies that $\socp\cM = 0$ unless
$\tau' \supseteq \tau$ by definition of global support: localizing
along~$\tau'$ yields $\cM_{\tau'} = 0$ unless $\tau' \subseteq \tau$.
On the other hand, applying Proposition~\ref{p:nearby} to any
cogenerator of~$\cM$ along a face~$\tau'$ implies that $\tau = \tau'$,
because no $\tau$-coprimary module has a submodule supported on a face
strictly contained in~$\tau$.
\end{proof}

\begin{thm}\label{t:essential-submodule}
In a downset-finite module $\cM$ over a real polyhedral~group, $\cM'$
is an essential submodule if and only if $\soct\cM' \subseteq
\soct\cM$ is dense for all faces~$\tau$.
\end{thm}
\begin{proof}
First assume that $\cM'$ is not an essential submodule, so $\cN \cap
\cM' = 0$ for some nonzero submodule $\cN \subseteq \cM$.  Let $s \in
\ds[\tau]\cN$ be a cogenerator.  Any $\sigma$-vicinity of~$s$
in~$\cN$, afforded by Proposition~\ref{p:nearby}, has a socle
along~$\tau$ that is a neighborhood of~$\wt s$ in $\soct\cM$ whose
intersection with $\soct\cM'$ is~$0$.  Therefore $\soct\cM' \subseteq
\soct\cM$ is not dense.

Now assume that $\soct\cM' \subseteq \soct\cM$ is not dense for
some~$\tau$.  That means $\soct[\sigma]\cM$ for some nadir~$\sigma$
has an element $\wt s$ with a neighborhood $\soct\cN$ that intersects
$\soct\cM'$ in~$0$.  But $\soct\cN \cap \soct\cM' = \soct(\cN \cap
\cM')$ by Corollary~\ref{c:essential-submodule}.\ref{i:cap}.  The
vanishing of this socle along~$\tau$ means that $\socp(\cN \cap \cM')
= 0$ for all faces~$\tau'$ by Corollary~\ref{c:soc(coprimary)'}, and
thus $\cN \cap \cM' = 0$ by
Corollary~\ref{c:essential-submodule}.\ref{i:0}.  Therefore $\cM'$ is
not an essential submodule of~$\cM$.
\end{proof}

\begin{example}\label{e:essential-submodule}
The conclusion of Example~\ref{e:nearby} implies that $\cM' \subseteq \cM$ is
an essential submodule by Theorem~\ref{t:essential-submodule}.  Trying
to mimic this example in a finitely generated context is instructive:
pixelated rastering of the horizontal lines either isolates the socle
element at the right-hand endpoint of the bottom edge or prevents it
from existing in the first place by aligning with the right-hand end
of the line above it.
\end{example}

\section{Primary decomposition over real polyhedral groups}\label{s:hulls}

This section takes the join of \cite{prim-decomp-pogroup}, which
develops primary decomposition as far as possible over arbitrary
partially ordered groups with finitely many faces, and
Section~\ref{s:socle}, which develops socles over real polyhedral
groups.  That is, it investigates how socles interact with primary
decomposition in real polyhedral groups.

Having a functorial notion of socle allows the introduction in
Section~\ref{b:ass} of \mbox{coprimary} modules in a way that mirrors
the usual method from ordinary commutative algebra
(Definition~\ref{d:associated}), including characterization via unique
associated face (Theorem~\ref{t:coprimary}).  Section
\ref{b:intervals} then leverages associated faces and the socle
criterion for essentiality (Theorem~\ref{t:essential-submodule}) to
construct minimal irreducible (Theorem~\ref{t:interval=union} and
Corollary~\ref{c:interval=union}) as well as primary
(Theorem~\ref{t:hull-I}) decompositions of interval modules.  The
transition from indicator modules to arbitrary modules by downset and
interval hulls (Theorems~\ref{t:dense-hull-M}
and~\ref{t:interval-hull}) occupies Section~\ref{b:hulls}.  That makes
way in Section~\ref{b:minimal-primary} for formulations of minimal
primary decomposition via socle isomorphism
(Definition~\ref{d:minimal-primary} and
Theorem~\ref{t:minimal-primary}).

\subsection{Associated faces and coprimary modules}\label{b:ass}\mbox{}

\medskip
\noindent
What makes the theory for real polyhedral groups stronger than for
arbitrary partially ordered groups is the following notion that is
familiar from commutative algebra, except that (as noted in
Section~\ref{s:density}) socle elements do not lie in the original
module.

\begin{defn}\label{d:associated}
A face $\tau$ of a real polyhedral group~$\cQ$ is \emph{associated} to
a downset-finite $\cQ$-module~$\cM$ if $\soct\cM \neq 0$.  If $\tau$
is associated to $\cM = \kk[D]$ for a downset~$D$ then $\tau$ is
\emph{associated}~to~$D$.  The set of associated faces of~$\cM$ or~$D$
is denoted by $\ass\cM$ or~$\ass D$.
\end{defn}

\begin{example}\label{e:associated}
Instances of associated faces have occurred numerous times.  Any
Example that produces cogenerators or socles along a face~$\tau$ can be
viewed identifying~$\tau$ as an associated face.  A~selection of specific
cases follows.
\begin{enumerate}
\item%
The downset $D_3$ in
Example~\ref{e:sigma-nbd-cogen}.\ref{i:half-plane} has $\ass D_3 =
\ass\kk[D_3] = \{\sigma_\0, \sigma_y\}$, where $\sigma_\0 = \{\0\}$ is the face of
dimension~$0$ containing only the origin and $\sigma_y$ is the positive
$y$-axis.  This assertion is justified by
Example~\ref{e:downset=union}.\ref{i:half-plane'}.

\item\label{i:min-primary}%
The left-hand module
$\kk[\raisebox{-.6ex}[0pt][0pt]{\includegraphics[height=2.8ex]{decomp}}]$
in Example~\ref{e:min-primary}, has
$\ass\kk[\raisebox{-.6ex}[0pt][0pt]{\includegraphics[height=2.8ex]{decomp}}]
= \{\sigma_\0, \sigma_x\}$.  Indeed, the two summands on the right-hand side
there are coprimary, with associated faces $\sigma_\0$ and~$\sigma_x$,
respectively.  Theorem~\ref{t:injection} guarantees that
$\ass\kk[\raisebox{-.6ex}[0pt][0pt]{\includegraphics[height=2.8ex]{decomp}}]
\subseteq \{\sigma_\0, \sigma_x\}$.  Example~\ref{e:global-socc-tau} explicitly
demonstrates that $\sigma_x$ is associated.  Direct inspection shows that
$\sigma_\0$ is associated, since $\Hom_Q(\kk[\aa],
\kk[\raisebox{-.6ex}[0pt][0pt]{\includegraphics[height=2.8ex]{decomp}}])
\cong \kk$ for any point~$\aa$ along the curved portion of the upper boundary.

\item%
The ray $\sigma_x$ is associated to the interval module appearing on the
left side of Example~\ref{e:nadir-x} by Example~\ref{e:intro-soc}.
\end{enumerate}
\end{example}

\begin{thm}\label{t:coprimary}
A downset-finite module~$\cM$ over a real polyhedral group is
$\tau$-coprimary (Definition~\ref{d:coprimary}) if and only if
$\socp\cM \!=\! 0$ whenever $\tau' \!\neq\! \tau$\hspace{-1.63pt} or
equivalently \mbox{$\ass(\cM) \!=\! \{\tau\}$}.
\end{thm}
\begin{proof}
If $\cM$ is not $\tau$-coprimary then either $\cM \to \cM_\tau$ has
nonzero kernel~$\cN$, or $\cM \to\nolinebreak \cM_\tau$ is injective
while $\cM_\tau$ has a submodule~$\cN_\tau$ supported on a face
strictly containing~$\tau$.  In the latter case, moving up by an
element of~$\tau$ shows that $\cN = \cN_\tau \cap \cM$ is nonzero.  In
either case, any cogenerator of~$\cN$ lies along a face $\tau' \neq
\tau$, so $0 \neq\nolinebreak \socp\cN \subseteq\nolinebreak \socp
\cM$.

On the other hand, if $\cM$ is $\tau$-coprimary then
$\gt\cM$ is an essential submodule of~$\cM$ because every
nonzero submodule of $\cM \subseteq \cM_\tau$ has nonzero intersection
with $\gt\cM_\tau$, and hence with $M \cap
\gt\cM_\tau = \gt\cM$, inside of the ambient
module~$\cM_\tau$ by Definition~\ref{d:coprimary}.
Theorem~\ref{t:essential-submodule} says that $\socp\gt\cM
\subseteq \socp\cM$ is dense for all~$\tau'$.  But
$\socp\gt\cM = 0$ for $\tau' \neq \tau$ by
Corollary~\ref{c:soc(coprimary)'}, so density implies $\socp\cM = 0$
for $\tau' \neq \tau$.
\end{proof}

\begin{lemma}\label{l:antichain}
A downset $D$ in a real polyhedral group is $\tau$-coprimary if and
only if
$$
  D
  =
  \bigcup_{\substack{\text{\rm faces }\sigma\\\text{\rm with }\sigma\supseteq\tau}}
  \bigcup_{\raisebox{-.7ex}{$\scriptstyle\ \aa\in\ats$}} \aa + \tau - \qns
$$
for sets $\ats \subseteq \cQ$ such that for each $\sigma \supseteq
\tau$, the image in $\qrt \times \nabt$ of\/
$\bigcup_{\sigma\supseteq\tau} \ats \times \{\sigma\}$ is an
antichain.  In that case, $\ats$ projects to a subset of
$\deg_{\qrt}\soct[\sigma]\kk[D] \subseteq \qrt$.
\end{lemma}
\begin{proof}
If $D$ is $\tau$-coprimary, then it is such a union by
Theorem~\ref{t:coprimary} and Theorem~\ref{t:downset=union}, keeping
in mind the antichain consequences of Example~\ref{e:soc-Rn-downset}.

On the other hand, if $D$ is such a union, then first of all it is
stable under translation by~$\RR\tau$ by Lemma~\ref{l:aa+tau-qns}.
Working in~$\qrt$, therefore, assume that $\tau = \{\0\}$.
Example~\ref{e:soc-Rn-downset} implies that every element of~$\ats$ is
a cogenerator of~$D$ with nadir~$\sigma$.
Proposition~\ref{p:sigma-nbd-cogen} produces a
$\sigma$-vicinity~$\oas$ of~$\aa$ in~$D$ that is globally supported
on~$\{\0\}$ (and hence $\{\0\}$-coprimary).  But every element $\bb
\in D$ that precedes~$\aa$ also precedes some element in~$\oas$; that
is, $\bb \preceq \aa \implies (\bb + \cQ_+) \cap \oas \neq \nothing$.
The union of the $\sigma$-vicinities $\oas$ over all faces~$\sigma$
and elements $\aa \in \ats$ therefore cogenerates~$D$, so $D$ is
coprimary by Proposition~\ref{p:elementary-coprimary}.
\end{proof}

\begin{remark}\label{r:antichain}
The antichain condition in Lemma~\ref{l:antichain} is necessary: $\cQ$
itself is the union of all translates of~$-\cQ_+$, but $\cQ$ is
$\cQ_+$-coprimary, whereas $-\cQ_+$ is $\{\0\}$-coprimary.  Moreover,
the $\nabt$ component of the antichain condition is important; that
is, the nadirs also come into play.  For a specific example, take $D
\subseteq \RR^2$ to be the union of the
\end{remark}\vskip -1.6ex
\begin{wrapfigure}{R}{0.15\textwidth}
  \vspace{-1.5ex}
  \includegraphics[height=11ex]{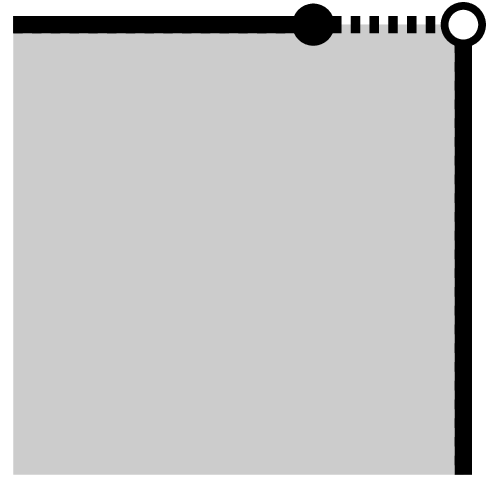}
  \vspace{-2ex}
\end{wrapfigure}
\noindent
open negative quadrant cogenerated by~$\0$ and the closed negative
quadrant cogenerated by any point on the strictly negative $x$-axis.
The $\cQ$-components of the two cogenerators are comparable in~$\cQ$,
but the nadirs are comparable the other way (it is crucial to remember
that the ordering on the nadirs is by~$\fqo$, not~$\cfq$, so smaller
faces are higher in the poset).  Of course, no claim can be made that
$\deg_{\qrt}\soct[\sigma]\kk[D]$ equals the image in~$\qrt$ of~$\ats$;
only the density claim in Theorem~\ref{t:downset=union} can be made.

\begin{lemma}\label{l:sub-coprimary}
Over any real polyhedral group, every submodule of a coprimary module
is coprimary.
\end{lemma}
\begin{proof}
Theorem~\ref{t:coprimary}, Theorem~\ref{t:injection}, and
Definition~\ref{d:associated}.
\end{proof}

\subsection{Canonical decompositions of intervals}\label{b:intervals}\mbox{}

\medskip
\noindent
The true real polyhedral generalizations of unique monomial minimal
irreducible and primary decomposition for monomial ideals in ordinary
polynomial rings are Theorem~\ref{t:interval=union} and
Theorem~\ref{t:hull-I}, respectively.  These concern intervals
rather than downsets because while the set of (exponents of) monomials
outside of an ideal in~$\kk[\NN^n]$ is a downset in~$\NN^n$, as a
subset of~$\ZZ^n$ this set of exponents is merely an interval.
These real polyhedral decompositions make full use of the notions of
socle, cogenerator, and density introduced in earlier sections: the
topological graded algebra is essential.

Theorem~\ref{t:divides}, on the existence of enough downset
cogenerators, holds verbatim for intervals instead of downsets.  The
proof is a telling application of the dense-socle characterization of
essential submodules in Theorem~\ref{t:essential-submodule}, combined
with its precursor, Theorem~\ref{t:downset=union}, that dense subsets
of cogenerators suffice to express a downset as a union.  With no
additional effort, enough cogenerators of an interval can be located
in any dense subset of the full set of cogenerators, generalizing
Theorem~\ref{t:downset=union} directly.  To state this version it is
useful to make a final definition concerning density.

\begin{defn}\label{d:dense-cogenerator}
Fix a real polyhedral group~$\cQ$.  Let $S$ be a family of
sets~$\sts$, indexed by pairs of faces $\sigma$ and~$\tau$ of~$\cQ_+$
with $\sigma \supseteq \tau$, such that $\sts \subseteq \qrt$.  This
family $S$ is \emph{dense} in another such family~$\hat S$ if the
$\sigma$-closure (Definition~\ref{d:sigma-closure}.\ref{i:closure})
of~$\sts[\,]\! = \bigcup_{\sigma \supseteq \tau} \sts$ in~$\qrt$
contains~\noheight{$\hat\sts$} for all $\sigma \supseteq \tau$.
\end{defn}

\begin{example}\label{e:dense-cogenerator}
Let $R$ be the open antidiagonal ray in the upper boundary of~$D_1$
from Examples~\ref{e:sigma-nbd-cogen}.\ref{i:pinch-0}
and~\ref{e:downset=union}.\ref{i:pinch-0'}.  Shortening $\sigma_\0$, $\sigma_x$,
and~$\sigma_y$ to $\0$, $x$, and~$y$, set
\begin{itemize}
\item%
$S^\0_\0 = R = \hat S^\0_\0$ and
\item%
$S^y_\0 = \{\ee_1\} = \hat S^y_\0$ but
\item%
$S^x_\0 = \nothing$ while $\hat S^x_\0 = \{\ee_1\}$.
\end{itemize}
The content of Example~\ref{e:downset=union}.\ref{i:pinch-0'} is that
$S$ is dense in~$\hat S$.  Consequently,
Example~\ref{e:downset=union}.\ref{i:pinch-0'} also serves as an
example of Theorem~\ref{t:interval=union}.
\end{example}

\begin{thm}\label{t:interval=union}
Fix an interval $I$ in a real polyhedral group~$\cQ$ and~$\bb \in I$.
There are faces $\tau \subseteq \sigma$ of~$\cQ_+\!$~and a
cogenerator~$\aa$ of~$I$ along~$\tau$ with nadir~$\sigma$ such that
$\bb \preceq \aa$.  In fact, the cogenerator\/~$\aa$ can be selected
from any given family $A \hspace{-1pt}=\hspace{-1pt}
\{\hspace{-1pt}\ats\}_{\sigma\supseteq\tau}$,~where~\mbox{$\ats
\hspace{-1pt}\subseteq\hspace{-1pt} \cQ$} is a set of cogenerators
of~$I$ along~$\tau$ with nadir~$\sigma$ such that the family
$\{\ats/\RR\tau\}_{\sigma\supseteq\tau}$ of sets of cosets in $\qrt$
is dense in $\deg\soc\kk[I] =
\{\deg_{\qrt}\soct[\sigma]\kk[I]\}_{\sigma\supseteq\tau}$.  In this
case,
\begin{equation*}
  I\ \
  =
\bigcup_{\substack{\hbox{\footnotesize\rm faces $\sigma,\tau$}\\
		   \hbox{\footnotesize\rm with $\sigma\supseteq\tau$}}}
\bigcup_{\raisebox{-.7ex}{\footnotesize$\ \aa\!\in\!\!\ats$}}(\aa+\tau-\qns)\cap I.
\end{equation*}
\end{thm}
\begin{proof}
Let $D$ be the downset cogenerated by~$I$.  The interval
module~$\kk[I]$ is an essential submodule of~$\kk[D]$ by construction.
Theorem~\ref{t:essential-submodule} implies that $\soc \kk[I]$ is
dense in $\soc \kk[D]$, and hence any family that is dense in $\soc
\kk[I]$ is dense in $\soc \kk[D]$.  Theorem~\ref{t:downset=union}
therefore expresses $D$ as the union of downsets cogenerated by the
given cogenerators of~$I$.  This means that every element of~$D$
precedes one of the given cogenerators of~$I$, so certainly every
element of~$I$ precedes such a cogenerator of~$I$.
\end{proof}

\begin{cor}\label{c:interval=union}
Every interval~$I$ in a real polyhedral group has a canonical
\emph{irreducible decomposition} as a union of intervals in~$I$
cogenerated by the family of global cogenerators of~$I$.  This
decomposition is minimal in the sense that the only subfamilies
yielding a union that still equals~$I$ are dense in the canonical
family.
\end{cor}
\begin{proof}
Let $\ats$ in Theorem~\ref{t:downset=union} comprise all cogenerators
of~$I$ along~$\tau$ with nadir~$\sigma$.
\end{proof}

\begin{remark}\label{r:asw}
The case of Theorem~\ref{t:interval=union} corresponding to finitely
generated monomial ideals in real-exponent polynomial rings is treated
in \cite{andersen--sather-wagstaff2015}.
\end{remark}

\begin{example}\label{e:I->D}
Looking at the proof of Theorem~\ref{t:interval=union}, it is tempting
to posit that if $D$ is the downset cogenerated by an interval $I$ in
a real polyhedral group~$\cQ$, then the inclusion $I \subseteq D$
induces a socle isomorphism $\soc \kk[I] \simto \soc \kk[D]$.  But
alas, it fails: the module~$M'$ in
Examples~\ref{e:nearby} and~\ref{e:essential-submodule}
cogenerates a downset that shares with the ambient module~$M$ its
cogenerator~$z$, which is not a cogenerator of~$M'$ itself.
Consequently, socle density as in Theorem~\ref{t:essential-submodule}
is the best one can hope for.  (See Remark~\ref{r:hull-I}.)
\end{example}

\begin{defn}\label{d:minimal-primDecomp-downset}
A primary decomposition (Example~\ref{e:interval}) $I =
\bigcup_{j=1\!}^k I_j$ of an interval in a real polyhedral group is
\emph{minimal} if
\begin{enumerate}
\item%
each face associated to~$I$ is associated to precisely one of the
downsets~$I_j$, and
\item%
the natural map $\soct\kk[I] \to \soct\bigoplus_{j=1}^k \kk[I_j]$ is
an isomorphism for all~faces~$\tau$.
\end{enumerate}
\end{defn}

\begin{example}\label{e:minimal-primDecomp-downset}
The primary decomposition in Example~\ref{e:min-primary} (see
Example~\ref{e:primDecomp}) is minimal by
Example~\ref{e:associated}.\ref{i:min-primary}.  Indeed, the socle
along~$\sigma_x$ on the right-hand side has dimension~$1$ over~$\kk$, so the
injection in Example~\ref{e:primDecomp} must be an isomorphism because
the socle along~$\tau = \sigma_x$ on the left-hand side is nonzero, owing
to~$\sigma_x$ being associated.  The direct inspection in
Example~\ref{e:associated}.\ref{i:min-primary} works on both the left
and the right to demonstrate the isomorphism of socles along~$\tau =
\{0\}$ induced by the injection.
\end{example}

\begin{thm}\label{t:hull-I}
Every interval $I \subseteq \cQ$ in a real polyhedral group~$\cQ$ has
a canonical primary decomposition (Example~\ref{e:interval}) whose
corresponding primary decomposition of the interval module $\kk[I]$ is
minimal.  Explicitly, this decomposition expresses~$I$ as a~union
$$
  I\
  =
  \bigcup_{\raisebox{-.7ex}{$\scriptstyle\tau\in\ass I$}}
  \bigcup_{\substack{\sigma\supseteq\tau\\\,\aa\in\ats}}
  (\aa + \tau - \qns) \cap I
$$
of coprimary intervals, where $\ats
\hspace{-1pt}\subseteq\hspace{-1pt} \cQ$ is the set of cogenerators
of~$I\hspace{-1pt}$ along~$\tau$ with~nadir~$\sigma$.
\end{thm}
\begin{proof}
That $I$ equals the union is a special case of
Theorem~\ref{t:interval=union}.  The inner union for fixed~$\tau$ is
$\tau$-coprimary by Lemma~\ref{l:sub-coprimary} because it is
coprimary when the intersection with~$I$ is omitted, by
Lemma~\ref{l:antichain}.  To see that the socle maps are isomorphisms,
let $I^\tau$ be the inner union for fixed~$\tau \in \ass I$.  Applying
$\soct$ to the primary decomposition $\kk[I] \into \bigoplus_\tau
\kk[I^\tau]$~ yields $\soct\kk[I] \into \soct\kk[I^\tau]$ by
Theorem~\ref{t:injection}.  The question is whether every cogenerator
of~$I^\tau$ is indeed a cogenerator of~$I$.  Each cogenerator
of~$I^\tau$ has a nadir~$\sigma$ for some face $\sigma \supseteq
\tau$, and Lemma~\ref{l:nearby} produces a corresponding
$\sigma$-vicinity
(Definition~\ref{d:nearby}.\ref{i:neighborhood-in-M})
in~$\kk[I^\tau]$.  But $I^\tau$ is contained in~$I$, explicitly by the
way it is defined as a union, given Theorem~\ref{t:interval=union}, so
this vicinity is contained in~$I$.  Therefore the cogenerator
of~$I^\tau$ in question must also be a cogenerator of~$I$.
\end{proof}

\begin{example}\label{e:redundant-primary}
The global support at the right-hand end of
Example~\ref{e:global-support} yields a redundant primary component,
and hence is not part of the minimal primary decomposition in
Example~\ref{e:min-primary} that comes from Theorem~\ref{t:hull-I},
because the interval~$I =
\raisebox{-.6ex}[0pt][0pt]{\includegraphics[height=2.8ex]{decomp}}$
has no cogenerators along the $y$-axis, in the functorial sense of
Definition~\ref{d:soct}.  The illustrations in
Example~\ref{e:global-support} show that $\kk[I]$ has elements
supported on the face of~$\RR^2$ that is the (positive) $y$-axis, but
the socle of~$\kk[I]$ along the $y$-axis is~$0$, as the top half of
the boundary curve of~$\kk[I]$ is supported on the origin.
\end{example}

\begin{remark}\label{r:hull-I}
The sole reason why Example~\ref{e:I->D} happens is the penultimate
sentence of the proof of Theorem~\ref{t:hull-I}: the downset
cogenerated by~$I^\tau$ need not be contained in~$I$, so a vicinity of
a cogenerator of this downset can have empty intersection with~$I$.
\end{remark}

\begin{example}\label{e:not-Gamma_tau}
The canonical $\tau$-primary component in Theorem~\ref{t:hull-I} can
differ from the $\tau$-primary component $P_\tau(D)$ in
\cite[Definition~3.2.6 and Corollary~3.11]{prim-decomp-pogroup},
namely the downset $\gt(D_\tau) - Q_+$ cogenerated by the local
$\tau$-support of~$D$.  However, it takes dimension at least~$3$ to
force a difference.  For a specific case, let $\tau$ be the $z$-axis
in~$\RR^3$, and let~$D_1$ be the $\{\0\}$-coprimary
(Lemma~\ref{l:antichain}) downset in~$\RR^3$ cogenerated by the
nonnegative points on the surface $z = 1/(x^2 + y^2)$.  Then every
point on the positive $z$-axis is supported on~$\tau$ in~$D_1$.  That
would suffice, for the present purpose, but for the fact that $\tau$
fails to be associated to~$D_1$.  The remedy is to force $\tau$ to be
associated by taking the union of~$D_1$ with any downset $D_2 = \aa +
\tau - \RR^3_+$ with $\aa = (x,y,z)$ satisfying $xy < 0$, the goal
being for $D_2 \not\subseteq D_1$ to be $\tau$-coprimary but not
contain the $z$-axis itself.  The canonical $\tau$-primary component
of $D = D_1 \cup D_2$ is just $D_2$ itself, but by construction $\gt
D$ also contains the positive $z$-axis.  (Note: $D = D_1 \cup D_2$ is
not the canonical primary decomposition of~$D$ because $D_2$ swallows
an open set of cogenerators of~$D_1$, so these cogenerators must be
omitted from the $\{\0\}$-primary component to induce an isomorphism
on socles.)  The reason why three dimensions are needed is that $\tau$
must have positive dimension, because elements supported on~$\tau$
must be cloaked by those supported on a smaller face; but $\tau$ must
have codimension more than~$1$, because there must be enough room
modulo~$\RR\tau$ to have incomparable~elements.
\end{example}

\subsection{Downset and interval hulls of modules}\label{b:hulls}\mbox{}%

\medskip
\noindent
Recall the notion of coprimary module from
Definition~\ref{d:coprimary} and its equivalent characterization in
Theorem~\ref{t:coprimary}.

\begin{defn}\label{d:dense-hull}
A downset hull $\cM \into E = \bigoplus_{j=1}^k E_j$
(Definition~\ref{d:downset-hull}) of a module over a real polyhedral
group is
\begin{enumerate}
\item%
\emph{coprimary} if $E_j = \kk[D_j]$ is coprimary for all~$j$, so
$D_j$ is a coprimary downset,~and
\item%
\emph{dense} if the induced map $\soct\cM \into \soct E$ is dense
(Definition~\ref{d:nearby}.\ref{i:dense}) for all~$\tau$.
\end{enumerate}
\end{defn}

\begin{example}\label{e:dense-hull}
The two modules $\cM'$ and~$\cM$ from Example~\ref{e:nearby}
both include into the indicator module~$\kk[D_2]$ for the downset~$D_2$
from Example~\ref{e:sigma-nbd-cogen}.  In fact, both $\cM'$ and~$\cM$
cogenerate~$\kk[D_2]$, and both inclusions $\cM' \into \kk[D_2]$ and $\cM \into
\kk[D_2]$ are coprimary, since $\kk[D_2]$ is coprimary, as noted in
Example~\ref{e:downset=union}.  However, the inclusion of~$\cM$
induces an isomorphism on socles, whereas the inclusion of~$\cM'$ does
not.  Therefore both inclusions into~$\kk[D_2]$ are dense---this is the
conclusion of Example~\ref{e:nearby}---even though one socle properly
contains the other.
\end{example}

\begin{remark}\label{r:dense=essential}
The density condition in Definition~\ref{d:dense-hull} is equivalent
to $\cM$ being an essential submodule of the downset hull~$E$, by
Theorem~\ref{t:essential-submodule}.
\end{remark}

\begin{thm}\label{t:dense-hull-M}
Every downset-finite module~$\cM$ over a real polyhedral group admits
a dense coprimary downset hull.
\end{thm}
\begin{proof}
Suppose that $\cM \to \bigoplus_{j=1}^k E_j = E$ is any finite downset
hull.  Replacing each~$E_j$ by a primary decomposition of~$E_j$, using
Theorem~\ref{t:hull-I}, assume that this downset hull is coprimary.
Let $E^\tau$ be the direct sum of the $\tau$-coprimary summands
of~$E$.  Then $\soct E = \soct E^\tau$ by Theorem~\ref{t:coprimary}.
Replacing~$\cM$ with its image in~$E^\tau$, it therefore suffices to
treat the case where $\cM$ is $\tau$-coprimary and $E = E^\tau$.

The proof is by induction on the number~$k$ of summands of~$E$.  If $k
= 1$ then $M = \kk[I] \subseteq \kk[D] = \kk[D^\tau]$ is an interval
submodule of a $\tau$-coprimary downset module by hypothesis.  If $D'$
is the downset cogenerated by~$I$, then $\kk[I] \subseteq \kk[D']$ is
an essential submodule by construction and coprimary by
Lemma~\ref{l:sub-coprimary}.  Thus $\cM = \kk[I] \subseteq \kk[D'] =
E$ is a dense coprimary downset hull by
Theorem~\ref{t:essential-submodule}.

When $k > 1$, let $\cM' = \ker(\cM \to E_k)$.  Then $\cM' \into
\bigoplus_{j=1}^{k-1} E_j$, so it has a dense coprimary downset hull
$\cM' \into E'$ by induction.  The $k = 1$ case proves that the
quotient $\cM'' = \cM/\cM'$ has a dense coprimary downset hull $\cM''
\into E''$.  The exact sequence $0 \to \cM' \to \cM \to \cM'' \to 0$
yields an exact sequence
$$
  0 \to \soct\cM' \to \soct\cM \to \soct\cM''
$$
which, if exact, automatically splits by Lemma~\ref{l:soc-as-k-vect}.
Hence it suffices to prove that $\soct\cM \to \soct\cM''$ is
surjective.  For that, note that the image of~$\soct\cM$ in~$\soct E$
surjects onto its projection to~$\soct E_k$, but the image of
$\soct\cM \to \soct E_k$ is the image of the injection $\soct\cM''
\into \soct E_k$ by construction.
\end{proof}

\begin{remark}\label{r:analogue-of-injres}
Theorem~\ref{t:dense-hull-M} is the analogue of existence of minimal
injective hulls for finitely generated modules over noetherian rings
\cite[Section~3.2]{bruns-herzog} (see also
\cite[Proposition~5.7 or Theorem~5.19]{hom-alg-poset-mods} for
finitely determined $\ZZ^n$-modules, which need not be finitely
generated).  The difference here is that a direct sum---as opposed to
direct product---can only be attained by gathering cogenerators into
finitely many~bunches.
\end{remark}

\begin{example}\label{e:downset-indecomposable}
The indicator module for the disjoint union of the strictly negative
axes in the plane injects in an appropriate way into one downset
module (the punctured negative quadrant) or a direct sum of two
(negative quadrants missing one boundary axis each).  Thus the
``required number'' of downsets for a downset hull of a given module
is not necessarily obvious and might not be a functorial invariant.
This may sound bad, but it should not be unexpected: the quotient by
an artinian monomial ideal in an ordinary polynomial ring can have
socle of arbitrary finite dimension, so the number of coprincipal
downsets required is well defined, but if downsets that are not
necessarily coprincipal are desired, then any number between $1$ and
the socle dimension would suffice.  This phenomenon is related to
Remark~\ref{r:soc-as-k-vect}: breaking the socle of a downset into two
reasonable pieces expresses the original downset as a union of the two
downsets cogenerated by the pieces.
\end{example}

\begin{defn}\label{d:interval-hull}
An \emph{interval hull} of a module~$\cM$ over an arbitrary poset is
an injection $\cM \into H = \bigoplus_{j \in J} H_j$ with each $H_j$
being an interval module (Definition~\ref{d:indicator}.\ref{i:interval}).
The hull is \emph{finite} if $J$ is~finite.  Over a real polyhedral
group a finite interval hull is
\begin{enumerate}
\item%
\emph{coprimary} if $H_j = \kk[I_j]$ is coprimary for all~$j$, so
$I_j$ is a coprimary interval,~and
\item%
\emph{minimal} if the induced map $\soct\cM \into \soct H$ is an
isomorphism~ for all faces~$\tau$.
\end{enumerate}
\end{defn}

\begin{example}\label{e:interval-hull}
Any downset hull (Definition~\ref{d:downset-hull}) of a module~$\cM$
is an interval hull of~$\cM$.  If~$\cM$ only has nonzero graded pieces
in (say) the the nonnegative orthant---either open or closed---then
intersecting each of the downsets with that orthant would yield an
interval hull of~$\cM$.
\end{example}

\begin{remark}\label{r:min-vs-dense}
Minimality of interval hulls is stronger than density of downset
hulls: the induced socle inclusion is required to be an isomorphism
rather than merely dense.
\end{remark}

\begin{thm}\label{t:interval-hull}
Every downset-finite module~$\cM$ over a real polyhedral group admits
a minimal coprimary interval hull.
\end{thm}
\begin{proof}
The transition from downsets to intervals alters one key aspect of the
proof of Theorem~\ref{t:dense-hull-M}, namely the base of the
induction: when $k = 1$ the module is already an interval module, so
the identity map is a minimal interval hull, as the socle inclusion is
the identity isomorphism.  The rest of the proof goes through mutatis
mutandis, changing ``downset'' to ``interval'', ``dense'' to
``minimal'', and all
instances~of~``$\hspace{-.25ex}E\hspace{.25ex}$''~to~%
``$\hspace{-.15ex}H\hspace{.15ex}$''.
\end{proof}

\begin{remark}\label{r:filtration}
The proof of Theorem~\ref{t:interval-hull} shows more than its
statement: any coprimary interval hull $\cM \into H = H_1 \oplus \dots
\oplus H_k$ of a coprimary module~$\cM$ induces a filtration $0 =
\cM_0 \subset \cM_1 \subset \dots \subset \cM_k = \cM$ such that
$\soct\cM = \bigoplus_{j=1}^k \soct(\cM_j/\cM_{j-1})$, and furthermore
$\cM \into H$ can be ``minimalized'', in the sense that a minimal
hull~$H'$ can be constructed inside of~$H$ so that $\soct\cM \cong
\soct H'$ decomposes as direct sum of factors $\soct(\cM_j/\cM_{j-1})
\cong \soct H'_j$.  Reordering the summands~$H_j$ yields another
filtration of~$\cM$ with the same property.  That $\soct\cM$ breaks up
as a direct sum in so many ways should not be shocking, in view of
Remark~\ref{r:soc-as-k-vect}.  The main content is that all of the
socle elements of $\cM/\cM_{k-1}$ are inherited from~$\cM$,
essentially because $\cM_{k-1}$ is the kernel of a homomorphism to a
direct sum of downset modules, so $\cM_{k-1}$ has no cogenerators that
are not inherited from~$\cM$.
\end{remark}

\subsection{Minimal primary decomposition of modules}\label{b:minimal-primary}

\begin{defn}\label{d:minimal-primary}
A primary decomposition $\cM \into \bigoplus_{i=1}^r \cM/\cM_i$
(Definition~\ref{d:primDecomp}) of a module over a real polyhedral
group is \emph{minimal} if $\soct\cM \to \soct\bigoplus_{i=1}^r
\cM/\cM_i$ is an isomorphism for all faces~$\tau$.
\end{defn}

\begin{defn}\label{d:primary-component}
Given a coprimary interval hull $\cM \into H$ of an arbitrary
downset-finite module~$\cM$ over a real polyhedral group, write
$H^\tau$ for the direct sum of all summands of~$H$ that are
$\tau$-coprimary.  The kernel $\cM^\tau$ of the composite homomorphism
$\cM \to H \to H^\tau$ is the \emph{$\tau$-primary component
of\hspace{1ex}$0$} for this particular interval hull of~$\cM$.
\end{defn}

\begin{example}\label{e:primary-component}
To get an idea what Definition~\ref{d:primary-component} does, it
suffices to consider a simple case like Example~\ref{e:min-primary}.
Take $\tau = \sigma_x$, for instance, the positive $x$-axis.  The kernel of
the homomorphism from
$\kk[\raisebox{-.6ex}[0pt][0pt]{\includegraphics[height=2.8ex]{decomp}}]$
to the rightmost module in Example~\ref{e:min-primary} is the set~$K$
of elements in
$\kk[\raisebox{-.6ex}[0pt][0pt]{\includegraphics[height=2.8ex]{decomp}}]$
that are nonzero only in degrees lying strictly above the horizontal
line at the upper boundary of the half-space.  This submodule~$K$ is
a~priori nondescript, but the quotient
$\kk[\raisebox{-.6ex}[0pt][0pt]{\includegraphics[height=2.8ex]{decomp}}]/K$
is $\tau$-coprimary: it is the part of
$\kk[\raisebox{-.6ex}[0pt][0pt]{\includegraphics[height=2.8ex]{decomp}}]$
that lands in the closed lower half-space, and it is $\tau$-coprimary
because every submodule of a coprimary module is coprimary.  This is
the (standard commutative algebra) explanation for the ``co'' in
``$\tau$-coprimary'': a module is coprimary when it is the quotient of an
ambient module modulo a primary submodule.
\end{example}

\begin{thm}\label{t:minimal-primary}
Every downset-finite module~$\cM$ over a real polyhedral group admits
a minimal primary decomposition.  In fact, if $\cM \into H$ is a
coprimary interval hull then $\cM \into \bigoplus_\tau \cM/\cM^\tau$
is a primary decomposition that is minimal if $\cM \into H$ is
minimal.
\end{thm}
\begin{proof}
Fix a coprimary interval hull $\cM \into H$.  The quotient
$\cM/\cM^\tau$ is $\tau$-coprimary by Lemma~\ref{l:sub-coprimary}
since it is a submodule of the coprimary module~$H^\tau$, and $\cM \to
\bigoplus_\tau \cM/\cM^\tau$ is injective because the injection $\cM
\into \bigoplus_\tau H^\tau = H$ factors through $\bigoplus_\tau
\cM/\cM^\tau \subseteq H$.

Theorem~\ref{t:coprimary} implies that $\socp(\cM/\cM^\tau) = 0$
unless $\tau = \tau'$, regardless of whether $\cM \into H$ is minimal.
And if the hull is minimal, then $\soct\cM \to \soct H^\tau$ is an
isomorphism (by hypothesis) that factors through the injection
$\soct(\cM/\cM^\tau) \into \soct H^\tau$ (by construction), forcing
$\soct\cM \cong \soct(\cM/\cM^\tau)$ to be an isomorphism for
all~$\tau$.
\end{proof}

\begin{remark}\label{r:prim-decomp}
Theorem~\ref{t:minimal-primary} enables full access to interpretations
of primary decomposition in persistent homology, now with a notion of
minimality for multiple real parameters instead of versions without
minimality for partially ordered groups \cite{prim-decomp-pogroup} or
with minimality in the discrete case
\cite{harrington-otter-schenck-tillmann2019}.  Primary decomposition
has important statistical implications for applications of
multipersistence \cite{primary-distance}.
\end{remark}

\section{Socles and essentiality over discrete polyhedral groups}\label{s:discrete}

The theory developed for real polyhedral groups in
Sections~\ref{s:socle}--\ref{s:hulls} applies as well to discrete
polyhedral groups (Example~\ref{e:polyhedral-discrete}).  The theory
is easier in the discrete case, in the sense that only closed
cogenerator functors are needed, and none of the density
considerations in Sections~\ref{s:minimality}--\ref{s:density} are
relevant.  The results in this section are known or easily deduced
from well known theory when the module is assumed to be finitely
generated.  The deduction of the discrete case in the generality of
downset-finite modules is elementary and not a major extension, but it
is worthwhile to record the results, both because they are useful and
for comparison with the real polyhedral case.  Highlights include
detection of injective homomorphisms and essential submodules via
socles (Theorems~\ref{t:discrete-injection}
and~\ref{t:discrete-essential-submodule}), minimal primary
decomposition of intervals via socle isomorphism
(Definition~\ref{d:discrete-minimal-primDecomp-interval},
Theorem~\ref{t:discrete-hull-I}, and
Corollary~\ref{c:discrete-irred-decomp}), and minimal primary
decomposition of modules via socle isomorphism
(Definition~\ref{d:discrete-minimal-primary} and
Theorem~\ref{t:discrete-minimal-primary}).

For the analogue of Theorem~\ref{t:injection}, the notion of
divisibility in Definition~\ref{d:divides} makes sense, when $\sigma =
\{\0\}$, verbatim in the discrete polyhedral setting: an element $y
\in \cM_\bb$ \emph{divides} $x \in \cM_\aa$ if $\bb \in \aa - \cQ_+$
and $y \mapsto x$ under the natural map $\cM_\bb \to \cM_\aa$.

\begin{thm}[Discrete essentiality of socles]\label{t:discrete-injection}
Fix a homomorphism $\phi: \cM \to \cN$ of modules over a discrete
polyhedral group~$\cQ$.
\begin{enumerate}
\item\label{i:discrete-phi=>soct}%
If $\phi$ is injective then $\socct\phi: \socct\cM \to \socct\cN$ is
injective for all faces~$\tau$ of~$\cQ_+$.
\item\label{i:discrete-soct=>phi}%
If $\socct\phi: \socct\cM \to \socct\cN$ is injective for all
faces~$\tau$ of~$\cQ_+\!$~and~$\cM$ is downset-finite, then $\phi$
is~injective.
\end{enumerate}
Each homogeneous element of~$\cM\!$ divides some closed cogenerator
of~$\cM$.
\end{thm}
\begin{proof}
Item~\ref{i:discrete-phi=>soct} is a special case of
Proposition~\ref{p:left-exact-tau-closed}.  Item~\ref{i:soct=>phi}
follows from the divisibility claim, for if $z$ divides a closed
cogenerator~$s$ along~$\tau$ then $\phi(z) \neq 0$ whenever
$\socct\phi(\wt s) \neq\nolinebreak 0$, where $\wt s$ is the image
of~$s$ in~$\socct\cM$.  The divisibility claim follows from the case
where $\cM$ is generated by $z \in \cM_\bb$.  But $\<z\>$ is a
noetherian $\kk[\cQ_+]$-module and hence has an associated prime.
This prime equals the annihilator of some homogeneous element
of~$\<z\>$, and the quotient of~$\kk[\cQ_+]$ modulo this prime
is~$\kk[\tau]$ for some face~$\tau$ \cite[Section~7.2]{cca}.  That
means, by definition, that the homogeneous element is a closed
cogenerator along~$\tau$ divisible by~$z$.
\end{proof}

The discrete analogue of Theorem~\ref{t:dense-subfunctor} is simpler
in both statement and proof.

\begin{thm}\label{t:subfunctor-discrete}
Fix subfunctors $\oS_\tau \subseteq \socct$ for all faces~$\tau$ of a
discrete polyhedral group.  Theorem~\ref{t:discrete-injection} holds
with $\oS$ in place of\/ $\socc$ if and only if\/ $\oS_\tau = \socct$
for all~$\tau$.
\end{thm}
\begin{proof}
Every subfunctor of any left-exact functor takes injections to
injections; therefore
Theorem~\ref{t:discrete-injection}.\ref{i:phi=>soct} holds for any
subfunctor of~$\socct$ by Proposition~\ref{p:left-exact-tau-closed}.
The content is that
Theorem~\ref{t:discrete-injection}.\ref{i:soct=>phi} fails as soon as
$\oS_\tau\cM \subsetneq \socct\cM$ for some module~$\cM$ and some
face~$\tau$.  To prove that failure, suppose $\wt s \in \socct\cM
\minus \oS_\tau\cM$ for some closed cogenerator~$s$ of~$\cM$
along~$\tau$.  Then $\<s\> \subseteq \cM$ induces an injection
$\oS_\tau\<s\> \into \oS_\tau\cM$, but by construction the image of
this homomorphism is~$0$, so~$\oS_\tau\<s\> = 0$.  But
$\socct[\tau']\<s\> = 0$ for all $\tau' \neq \tau$ because $\<s\>$ is
abstractly isomorphic to~$\kk[\tau]$, which has no associated primes
other than the kernel of~$\kk[\cQ_+] \onto \kk[\tau]$.  Consequently,
applying~$\oS_{\tau'}$ to the homomorphism $\phi: \<s\> \to 0$ yields
an injection $0 \into 0$ for all faces~$\tau'$ even though $\phi$ is
not injective.
\end{proof}

The analogue of Theorem~\ref{t:essential-submodule} is similarly
simpler.

\begin{thm}\label{t:discrete-essential-submodule}
In any module $\cM$ over a discrete polyhedral group, $\cM'$ is an
essential submodule if and only if\/ $\socct\cM' = \socct\cM$ for all
faces~$\tau$.
\end{thm}
\begin{proof}
First assume that $\cM'$ is not an essential submodule, so $\cN \cap
\cM' = 0$ for some nonzero submodule $\cN \subseteq \cM$.  Any closed
cogenerator~$s$ of~$\cN$ along any face~$\tau$ maps to a nonzero
element of $\socct\cM$ that lies outside of~$\socct\cM'$.  Conversely,
if $\socct\cM' \neq \socct\cM$, then any closed cogenerator of~$\cM$
that maps to an element $\socct\cM \minus \socct\cM'$ generates a
nonzero submodule of~$\cM$ whose intersection with~$\cM'$ is~$0$.
\end{proof}

The analogue of Theorem~\ref{t:hull-I} uses slightly modified
definitions but its proof~is~\mbox{easier}.

\begin{defn}\label{d:discrete-minimal-primDecomp-interval}
A primary decomposition (Definition~\ref{e:interval}) $I =
\bigcup_{j=1\!}^k I_j$ of an interval in a discrete polyhedral group
is \emph{minimal} if
\begin{enumerate}
\item%
the intervals~$I_j$ are coprimary for distinct associated faces of~$I$,
and
\item%
the natural map $\socct\kk[I] \to \socct\bigoplus_{j=1}^k \kk[I_j]$ is
an isomorphism for all~faces~$\tau$,
\end{enumerate}
where $\tau$ is \emph{associated} if some element generates an upset
in~$I$ that is a translate of~$\tau$.
\end{defn}

\begin{thm}\label{t:discrete-hull-I}
Every interval $I$ in a discrete polyhedral group has a canonical
minimal primary decomposition $I = \bigcup_\tau I^\tau$ as a union of
coprimary intervals
$$
  I^\tau\
  =
  \bigcup_{\aa_\tau \in \deg\socct\kk[I]} (\aa_\tau - \cQ_+) \cap I,
$$
where $\aa_\tau$ is viewed as an element in~$\qzt$ to write $\aa_\tau
\in \deg\socct\kk[I]$ but $\aa_\tau \subseteq \cQ$ is viewed as a
coset of\/~$\ZZ\tau$ to write $\aa_\tau - \cQ_+$.
\end{thm}
\begin{proof}
The interval~$I$ is contained in the union by the final line of
Theorem~\ref{t:discrete-injection}, but the union is contained in~$I$
because every closed cogenerator of~$I$ is an element of~$I$.  It
remains to show that $I^\tau$ is coprimary and that the socle maps are
isomorphisms.

Each nonzero homogeneous element $z \in \kk[I^\tau]$ divides an
element~$s_z$ whose degree lies in some coset~$\aa_\tau \in
\deg\socct\kk[I]$ by construction.  As $I^\tau \subseteq I$, each such
element~$s_z$ is a closed cogenerator of~$I^\tau$ along~$\tau$.
Therefore $\kk[I^\tau]$ is coprimary, inasmuch as no prime other than
the one associated to~$\kk[\tau]$ can be associated to~$I^\tau$.  The
same argument shows that these elements~$s_z$ generate an essential
submodule of~$I^\tau$, and then
Theorem~\ref{t:discrete-essential-submodule} yields the isomorphism on
socles.
\end{proof}

\begin{cor}\label{c:discrete-irred-decomp}
Every interval~$I$ in a discrete polyhedral group~$\cQ$ has a unique
irredundant \emph{irreducible decomposition} as a union of its
\emph{irreducible components}, namely the coprincipal intervals
$(\aa_\tau - \cQ_+) \cap I$ in Theorem~\ref{t:discrete-hull-I}.
\end{cor}
\begin{proof}
The irredundant condition is a consequence of the socle isomorphisms.
\end{proof}

\begin{defn}\label{d:discrete-minimal-hull}
A downset hull $\cM \into E = \bigoplus_{j=1}^k E_j$
(Definition~\ref{d:downset-hull}) of a module over a discrete
polyhedral group is
\begin{enumerate}
\item%
\emph{coprimary} if $E_j = \kk[D_j]$ is coprimary for all~$j$, so
$D_j$ is a coprimary downset,~and
\item%
\emph{minimal} if the induced map $\socct\cM \to \socct E$ is an
isomorphism for all faces~$\tau$.
\end{enumerate}
\end{defn}

The discrete analogue of Theorem~\ref{t:dense-hull-M} appears to be new.

\begin{thm}\label{t:discrete-hull-M}
Every downset-finite module~$\cM$ over a discrete polyhedral group
admits a minimal coprimary downset hull.
\end{thm}
\begin{proof}
The argument follows that of Theorem~\ref{t:dense-hull-M}, using
Theorem~\ref{t:discrete-hull-I} instead of Theorem~\ref{t:hull-I} and
Theorem~\ref{t:discrete-essential-submodule} instead of
Theorem~\ref{t:essential-submodule}.  In the course of the proof, note
that the discrete analogue of Theorem~\ref{t:coprimary} is the
definition of associated prime, making the analogue of
Lemma~\ref{l:sub-coprimary} trivial,
and that the analogue of Remark~\ref{r:soc-as-k-vect} holds (more
easily) in the discrete polyhedral setting.
\end{proof}

\begin{remark}\label{r:discrete-filtration}
Remark~\ref{r:filtration} holds verbatim over discrete polyhedral
groups.
\end{remark}

Finally, here is the discrete version of minimal primary
decomposition.

\begin{defn}\label{d:discrete-minimal-primary}
A primary decomposition $\cM \into \bigoplus_{i=1}^r \cM/\cM_i$
(Definition~\ref{d:primDecomp}) of a module over a discrete polyhedral
group is \emph{minimal} if $\socct\cM \to \socct\bigoplus_{i=1}^r
\cM/\cM_i$ is an isomorphism for all faces~$\tau$.
\end{defn}

\begin{defn}\label{d:discrete-primary-component}
Given a coprimary downset hull $\cM \into E$ of an arbitrary
downset-finite module~$\cM$ over a discrete polyhedral group, write
$E^\tau$ for the direct sum of all summands of~$E$ that are
$\tau$-coprimary.  The kernel $\cM^\tau$ of the composite homomorphism
$\cM \to E \to E^\tau$ is the \emph{$\tau$-primary component of~$0$}
for this particular downset hull of~$\cM$.
\end{defn}

\begin{thm}\label{t:discrete-minimal-primary}
Every downset-finite module~$\cM$ over a discrete polyhedral group
admits a minimal primary decomposition.  If $\cM \into E$ is a
coprimary downset hull then $\cM \into \bigoplus_{\tau\!}
\cM/\cM^\tau\!$ is a primary decomposition that is minimal if $\cM
\hspace{-.45pt}\into\hspace{-.45pt}\nolinebreak E$~is~\mbox{minimal}.
\end{thm}
\begin{proof}
Follow the proof of Theorem~\ref{t:minimal-primary}, using downset
hulls instead of interval hulls because, in contrast with the real
polyhedral case (Theorems~\ref{t:dense-hull-M}
and~\ref{t:interval-hull}), downset hulls are minimal---not merely
dense---in the discrete context (Theorem~\ref{t:discrete-hull-M}).
\end{proof}

\section{Generator functors and tops}\label{s:gen-functors}

The theory of generators is Matlis dual (Section~\ref{b:matlis}) to
the theory of cogenerators.  Every result for socles, downsets, and
cogenerators therefore has a dual.  All of these dual statements can
be formulated so as to be straightforward, but sometimes they are less
natural (see Remarks~\ref{r:global-closed-gen}
and~\ref{r:no-local-top}, for example), sometimes they are weaker (see
Remark~\ref{r:dir-vs-inv}), and sometimes there are natural
formulations that must be proved equivalent to the straightforward
dual (see Definition~\ref{d:topr} and Theorem~\ref{t:top=socvee}, for
example).  This section presents those Matlis dual notions that are
used in later~\mbox{sections}.

\subsection{Lower closure functors}\label{b:lower-closure}\mbox{}

\medskip
\noindent
In commutative algebra, the top of a module over a local ring is the
quotient of the module modulo the maximal ideal times the module.  As
such, the top is a vector space over~$\kk$ that can alternately be
characterized by tensoring the module with the residue field~$\kk$.
This vector space is interpreted as the space of minimal generators of
the module.  This description of tops is Matlis dual
(Section~\ref{b:matlis}) to the notion of socle (Section~\ref{b:socc};
compare the initial paragraph there).  Hence the development of tops
starting in Section~\ref{b:closed-gen} rests on the duals of upper
closure functors (Section~\ref{b:upper-closure}), which are the lower
closure functors introduced in this subsection
(Definitions~\ref{d:beneath-xi} and~\ref{d:lower-closure}).  Deducing
results about lower closure from those concerning upper closure
encounter a small wrinkle
(Definition~\ref{d:infinitesimally-Q-finite}) because the inverse
limits needed for lower closure have weaker exactness properties than
the direct limits needed for upper closure.

The following are Matlis dual to Definition~\ref{d:atop-sigma}.
Lemma~\ref{l:natural}, and Definition~\ref{d:upper-closure}.

\begin{defn}\label{d:beneath-xi}
For a module~$\cM$ over a real polyhedral group~$\cQ$, a face~$\xi$
of~$\cQ_+$, and a degree $\bb \in \cQ$, the \emph{lower closure
beneath~$\xi$} at~$\bb$ in~$\cM$ is the vector space
$$
  (\lx\cM)_\bb
  =
  \cM_{\bb+\xi}
  =
  \invlim_{\bb' \in \bb + \xi^\circ} \cM_{\bb'}.
$$
\end{defn}

\begin{lemma}\label{l:natural-dual}
The structure homomorphisms of~$\cM$ as a $\cQ$-module induce natural
homomorphisms $\cM_{\bb+\xi} \to \cM_{\cc+\eta}$ for $\bb \preceq
\cc$ in~$\cQ$ and faces $\xi \subseteq \eta$
of~$\cQ_+$.\hfill$\square$
\end{lemma}

\begin{remark}\label{r:semilattice=monoid'}
In contrast with Remark~\ref{r:semilattice=monoid}, the relevant
monoid structure here on the face poset~$\cfq$ of the positive
cone~$\cQ_+$ is opposite to the monoid denoted~$\fqo$.  In this case
the monoid axioms use that $\cfq$ is a bounded join semilattice, the
monoid unit being $\{\0\}$.  The induced partial order on~$\cfq$ is
the usual one, with $\xi \preceq \eta$ if~$\xi \subseteq \eta$.
\end{remark}

\begin{defn}\label{d:lower-closure}
Fix a module~$\cM$ over a real polyhedral group~$\cQ$ and a degree
$\bb \in \cQ$.  The \emph{lower closure functor} takes~$\cM$ to the
$\cQ \times \cfq$-module $\del\cM$ whose fiber over $\bb \in \cQ$ is
the $\cfq$-module
$$
  (\del\cM)_\bb
  =
  \bigoplus_{\xi \in \cfq} \cM_{\bb+\xi}
  =
  \bigoplus_{\xi \in \cfq} (\lx\cM)_\bb.
$$
The fiber of $\del\cM$ over $\xi \in \cfq$ is the \emph{lower
closure~$\lx\cM\!$ of~$\cM$ beneath~$\xi$}.
\end{defn}

\begin{remark}\label{r:dir-vs-inv}
Direct and inverse limits play differently with vector space duality.
Consequently, although the notion of lower closure functor is
categorically dual to the notion of upper closure functor, the
duality only coincides unfettered with vector space duality in one
direction, and some results involving tops are necessarily weaker than
the corresponding results for socles; compare
Theorem~\ref{t:injection} with~\ref{t:surjection-RR} and
Example~\ref{e:surjection}, for instance.  To make precise statements
throughout this section on generator functors, starting with
Lemma~\ref{l:lx-vee}, it is necessary to impose a finiteness condition
that is somewhat stronger than $\cQ$-finiteness
(Definition~\ref{d:tame}.\ref{i:Q-finite}).
\end{remark}

\begin{defn}\label{d:infinitesimally-Q-finite}
A module~$\cM$ over a real polyhedral group~$\cQ$ is
\emph{infinitesimally $\cQ$-fin- ite} if its lower closure module
$\del\cM$ is $\cQ$-finite.
\end{defn}

\begin{example}\label{e:not-infinitesimally-Q-finite}
The type of behavior ruled out by
Definition~\ref{d:infinitesimally-Q-finite} has vector spaces whose
dimensions increase without bound as their locations approach a
limiting degree.  For the prototypical concrete example, let $\cQ =
\RR$ and define $\cM$ to have $\cM_\aa = \kk^n$ for $n \in \NN$
whenever $\aa \in [\frac 1n, \frac 1{n-1})$, with the structure map
$\kk^n \to \kk^{n-1}$ being the projection that acts on the basis
vectors $\ee_1,\ldots,\ee_n$ of~$\kk^n$ via $\ee_i \mapsto \ee_i$
for $i < n$ and $\ee_n \mapsto 0$.  The set of degrees $\aa \in \RR$
where $\cM_\aa \neq 0$ is the open positive ray $\RR_+^\circ$.  The
lower closure of~$\cM$ has the inverse limit $\invlim_{n\in\NN}\kk^n$
in the $\xi = \RR_+$ component of~$\cfq$ at $\cQ$-graded degree~$0$.

These kinds of inverse limits do not play well with vector space
duality; more precisely, these kinds of inverse limits tend to violate
exactness.  The infinitesimally $\cQ$-finite hypothesis rescues
exactness; see Proposition~\ref{p:exact-lower-closure}.
\end{example}

For behavior that infinitesimal $\cQ$-finiteness does still allow, see
Example~\ref{e:infinitesimally-Q-finite}, whose discussion rests on
Lemma~\ref{l:lx-vee}.

\begin{lemma}\label{l:lx-vee}
If $\xi$ is a face of a real polyhedral group~$\cQ$, then
\begin{enumerate}
\item\label{i:all}%
$\lx(\cM^\vee) = (\dx\cM)^\vee$ for all $\cQ$-modules~$\cM$, and
\item\label{i:inf}%
$(\lx\cM)^\vee = \dx(\cM^\vee)$ if $\cM$ is infinitesimally
$\cQ$-finite
\end{enumerate}
\end{lemma}
\begin{proof}
Degree by degree $\bb \in \cQ$, the first of these
is because the vector space dual of a direct limit is the inverse
limit of the vector space duals.
Swapping ``direct'' and ``inverse'' only works with additional
hypotheses, and one way to ensure these is to assume infinitesimal
$\cQ$-finiteness of~$\cM$.  Indeed, then $\cM = \del^{\{\0\}}\cM$ is
$\cQ$-finite, so replacing $\cM$ with~$\cM^\vee$ in the first item
yields $\lx\cM = \bigl(\dx(\cM^\vee)\bigr){}^\vee$ by
Lemma~\ref{l:vee-vee}.  Thus $(\lx\cM)^\vee = \dx(\cM^\vee)$, as
$\lx\cM$---and hence $\bigl(\dx(\cM^\vee)\bigr){}^\vee\!$ and
$\dx(\cM^\vee)$---is~also~\mbox{$\cQ$-finite}.
\end{proof}

\begin{example}\label{e:infinitesimally-Q-finite}
Any module~$\cM$ that is a quotient of a finite direct sum of upset
modules (``upset-finite'' in Definition~\ref{d:minimal-cover}) over a
real polyhedral group is infinitesimally $\cQ$-finite.  Indeed, the
Matlis dual of such a quotient is a downset hull demonstrating that
$\cM^\vee$ is downset-finite and hence $\cQ$-finite.
Proposition~\ref{p:downset-upper-closure} and exactness of upper
closure functors (Lemma~\ref{l:exact-delta}) implies that
$\delta(\cM^\vee)$ remains downset-finite and hence $\cQ$-finite.
Applying Lemma~\ref{l:lx-vee}.\ref{i:all} to~$\cM^\vee$ and using that
$(\cM^\vee)^\vee\! = \cM$ (Lemma~\ref{l:vee-vee}) on the left-hand
side yields that $\del\cM$ is $\cQ$-finite.  This example includes all
tame modules by the syzygy theorem
\cite[Theorem~6.12.4]{hom-alg-poset-mods}.
\end{example}

\begin{prop}\label{p:exact-lower-closure}
The category of infinitesimally $\cQ$-finite modules over a real
polyhedral group~$\cQ$ is a full abelian subcategory of the category
of $\cQ$-modules.  Moreover, the lower closure functor is exact on
this subcategory.
\end{prop}
\begin{proof}
Use Matlis duality, in the form of Lemma~\ref{l:lx-vee}, along with
Lemma~\ref{l:exact-delta}.
\end{proof}

\subsection{Closed generator functors}\label{b:closed-gen}\mbox{}

\medskip
\noindent
The development of the theory of tops is carried by applying duality
to the results surrounding socles rather than by dualizing the proofs.
This simplifies the exposition substantially.  The theory of closed
tops here is Matlis dual to the theory of closed socles in
Section~\ref{b:socc}.  To start, here is the Matlis dual of
Definition~\ref{d:socc}.  Recall the skyscraper $\cP$-module $\kk_p$
there.

\begin{defn}\label{d:topc}
Fix an arbitrary poset~$\cP$.  The \emph{closed generator functor}
$\kk \otimes_\cP -$ takes each $\cP$-module $\cN$ to its \emph{closed
top}: the quotient $\cP$-module
$$
  \topc\cN = \kk \otimes_\cP \cN = \bigoplus_{p\in\cP} \kk_p\otimes_\cP\cN.
$$
When it is important to specify the poset, the notation $\cptop$ is
used instead of~$\topc\!$.  A \emph{closed generator} of
\emph{degree}~$p \in P$ is a nonzero element in $(\topc\cN)_p$.
\end{defn}

\begin{example}\label{e:nttop}
Elements of~$\cN_p$ that persist from lower in the poset die in the
tensor product $\kk \otimes_P \cN$.  Consequently, $\RR$-modules like
the maximal monomial ideal $\mm \subseteq \kk[\RR_+]$ have vanishing
closed top, because every monomial with nonzero positive degree is
divisible by a monomial of smaller positive degree (its square root,
for instance).  This is merely the statement that $\mm$ is not
minimally generated.
\end{example}

\begin{remark}\label{r:nttop}
$\cpsoc\cN \into \cN$ is the universal $\cP$-module monomorphism that
is~$0$ when composed with all nonidentity maps induced by going up
in~$\cP$.  The Matlis dual is $\cN \onto \cptop\cN$, the universal
$\cP$-module epimorphism that is~$0$ when composed with all
nonidentity maps induced by going up in~$\cP$.  This is the essence of
Proposition~\ref{p:dual-topc}.
\end{remark}

\begin{remark}\label{r:asymmetry}
Matlis duality has an intrinsic asymmetry regarding the behavior of
tops and socles.  In the presence of sufficient finiteness, the
asymmetry disappears, but in general it requires care to insert some
finiteness appropriately.  The following definition, proposition, and
proof are presented in (perhaps too much) detail to highlight how
finiteness enters.  Local finiteness
(Definition~\ref{d:locally-finite}) can fail for modules over a
partially ordered abelian group, but it is useful for the discrete
(face lattice) half of the poset used to compute tops and socles over
real polyhedral groups; see Proposition~\ref{p:nrtop}.  The other
finiteness restriction, namely $\cP$-finiteness
(Definition~\ref{d:tame}.\ref{i:Q-finite}) has already appeared, with
consequences (see Example~\ref{e:not-exact-Q-infinite}).
\end{remark}

\begin{defn}\label{d:locally-finite}
A module~$\cN$ over a poset~$\cP$ is \emph{locally finite} if, for
each poset element $p \in \cP$, there is a finite subset $\cP'(p,\cN)
\subseteq \cP$ such that, if $\cN_p \to \cN_{p''}$ is nonzero for some
$p'' \in \cP$, there is some $p' \in \cP'(p,\cN)$ with $p' \prec p''$.
\end{defn}

Loosely, $\cP'(p,\cN)$ sits between $\cN_p$ and any of its nonzero
images higher in~$\cP$.

\begin{prop}\label{p:dual-topc}
Fix a poset~$\cP$ with opposite poset~$\cP^\op$.  For a
$\cP$-module~$\cN$, Matlis duality interacts with tops and socles as
follows:
\begin{enumerate}
\item%
$\cptop(\cN)^\vee = \posoc(\cN^\vee)$ for any $\cP$-module~$\cN$, and
\item%
$\cptop(\cN^\vee) = (\posoc\cN)^\vee$ for any $\cP^\op$-finite or
locally finite $\cP^\op$-module~$\cN$.
\end{enumerate}
All of these hold with $\cP$ and~$\cP^\op$ swapped.
\end{prop}
\begin{proof}
View the $\cP$-module~$\cN$ as a diagram of vector spaces indexed
by~$\cP$.  Tensoring with~$\kk$ in Definition~\ref{d:topc} takes each
vector space~$\cN_p$ to the cokernel of the homomorphism
$\bigoplus_{p' \prec p} \cN_{p'} \to \cN_p$ induced by the maps going
up in~$\cP$ from~$p'$ to~$p$.  Let $L_p$ be the image in~$\cN_p$
of~$\bigoplus_{p' \prec p} \cN_{p'}$.  The vector space dual $(-)^*$
of the surjection $\cN_p \onto \cN_p/L_p$ is the kernel of the
homomorphism $\prod_{p' \succ p} \cN_{p'}^* \from \cN_p^*$, where $p'
\succ p$ in the partial order on $\cP^\op$ here.  This kernel is
$\Hom_{\cP^\op}(\kk,\cN^\vee)_p$, proving the first equation by
Definition~\ref{d:socc}.

The second equation is similar: $\Hom_{\cP^\op}(\kk,-)$ takes each
vector space $\cN_p$ in the diagram $\cN$ indexed by~$\cP^\op$ to the
kernel of the homomorphism $\prod_{p \prec p'} \cN_{p'} \from \cN_p$.
If a finite sub-product---over a subset $\cP'(p,\cN)$, say---suffices
to compute the kernel, as the $\cP^\op$-finite or locally finite
conditions guarantee, then the vector space dual of the kernel is the
cokernel of the homomorphism $\bigoplus_{p' \in \cP'(p,\cN)}
(\cN_{p'})^* \to (\cN_p)^*$ induced by the maps going up in~$\cP$
from~$p'$ to~$p$ in $\cN^\vee$.
\end{proof}

\subsection{Closed generator functors along faces}\label{b:closed-gen-along}\mbox{}

\medskip
\noindent
Closed generators along faces of partially ordered abelian groups make
sense just as closed cogenerators along faces do.  This section makes
the relevant definitions and explores the duality between closed
generators and cogenerators along faces
(Theorem~\ref{t:topc=socc^vee}).  The main complication, when it comes
to how to think about these things, is that generators along faces are
detected not by localization but by the Matlis dual operation in
Example~\ref{e:dual-of-localization}, which is likely unfamiliar (and
is surely less familiar than localization).  An element in the
following can be thought of as an inverse limit of elements of~$\cM$
taken along the negative of the face~$\rho$.  This is Matlis dual to the
construction of the localization~$\cM_\rho$ as a direct limit.

\begin{defn}\label{d:cmr}
Fix a face~$\rho$ of a partially ordered
group~$\cQ$ and a $\cQ$-module~$\cM$.~~Set
$$
  \cmr = \hhom_\cQ\bigl(\kk[\cQ_+]_\rho,\cM\bigr).
$$
\end{defn}

The following is Matlis dual to
Definition~\ref{d:socct}.\ref{i:global-socc-tau}; see
Theorem~\ref{t:topc=socc^vee}.  Duals for the rest of
Definition~\ref{d:socct} are omitted for reasons detailed in
Remarks~\ref{r:global-closed-gen} and~\ref{r:no-local-top}.

\begin{defn}\label{d:topct}
Fix a partially ordered abelian group~$\cQ$, a face~$\rho$, and a
$\cQ$-module~$\cM$.  The \emph{closed generator functor along~$\rho$}
takes $\cM$ to its \emph{closed top along~$\rho$}:
$$
  \topcr\cM
  =
  \bigl(\kk[\rho]\otimes_\cQ\cM\bigr){}^{\rho\!}\big/\rho.
$$
\end{defn}

\begin{example}\label{e:topct}
When $\cQ = \RR^2$ and $\rho$ is the face of~$\RR^2_+$ along the
$x$-axis, the module $\cmr =
\Hom_{\RR^2}\bigl({\pur\kk[\RR^2_+]_\rho}, {\blu \cM}\bigr)$ for the
depicted module~$\blu \cM$ is
$$
\Hom\left(
  \kk\!
  \left[
  \begin{array}{@{}c@{}}\includegraphics[height=20mm]{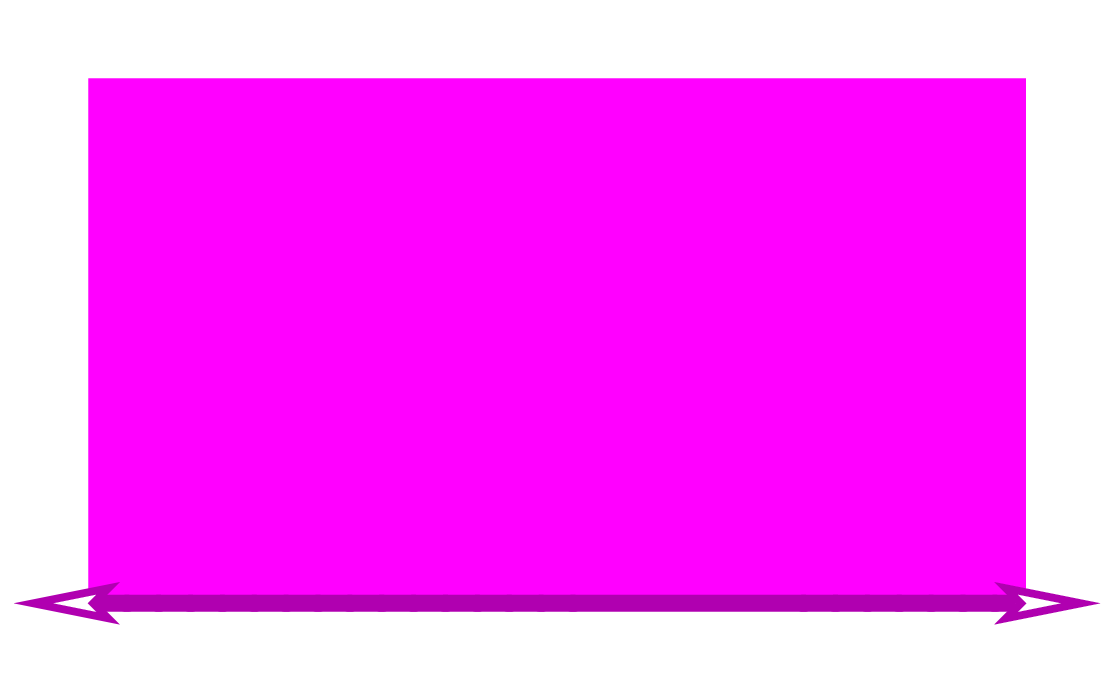}\end{array}
  \right]
\!,\
  \kk\!
  \left[
  \begin{array}{@{}c@{}}\includegraphics[height=20mm]{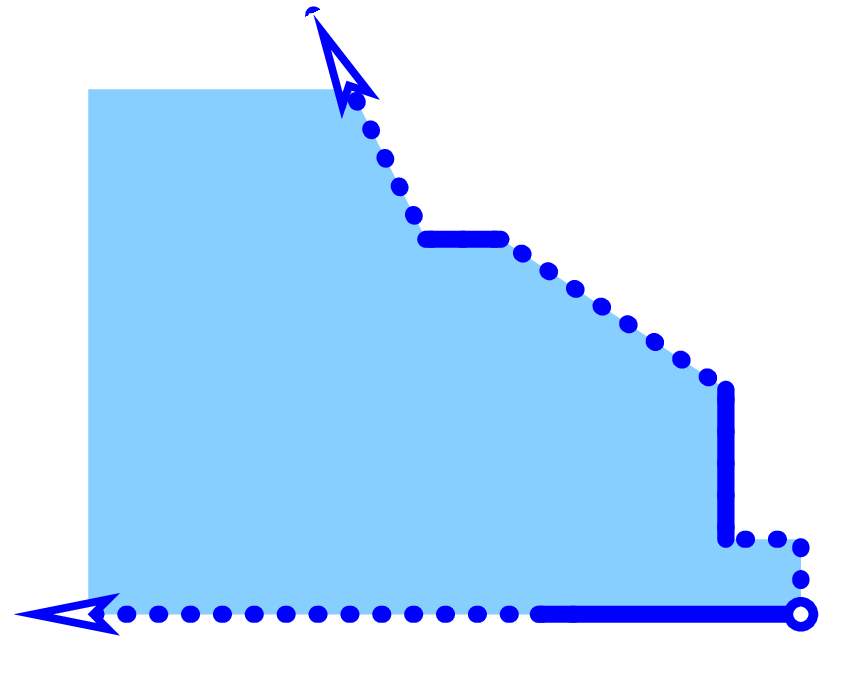}\end{array}
  \right]
\right)
\ \,=\,\
  \kk\!
  \left[
  \begin{array}{@{}c@{}}\includegraphics[height=20mm]{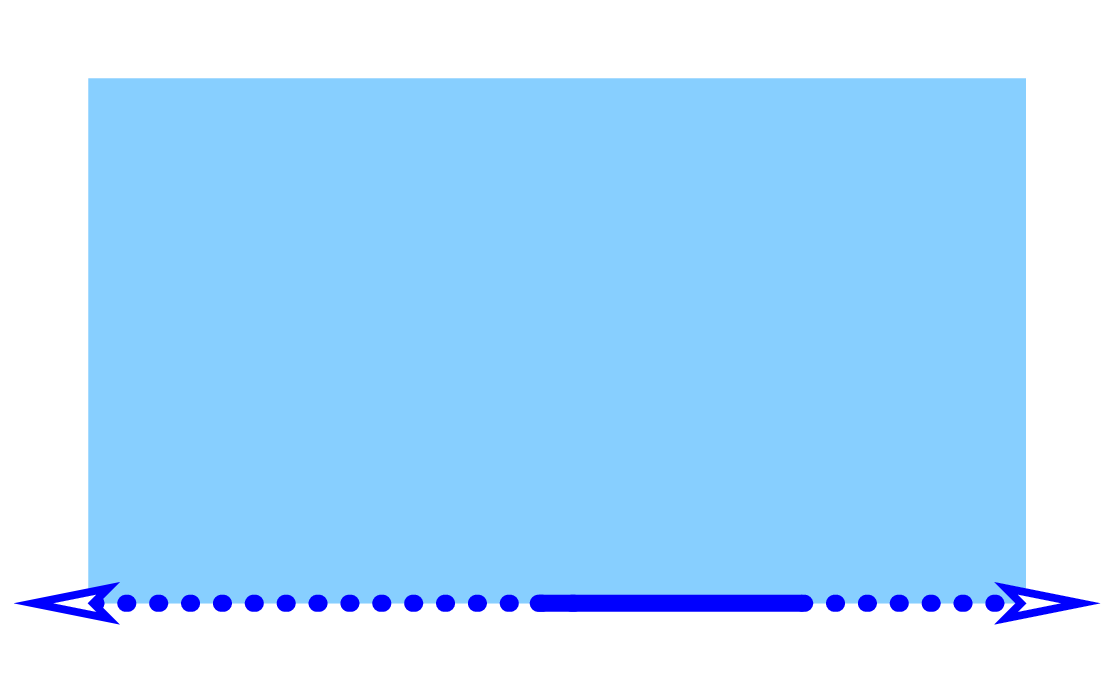}\end{array}
  \right].
$$
The effect is push the jagged right-hand boundary off to~$+\infty$.
It is interesting to note that the bottom edge of~$\blu \cM$ does
become homogenized in~$\cmr$, because there is no way to place the
closed purple upper half plane onto the blue module (either one,
actually, though $\blu \cM$ with the jagged upper boundary is meant
here).  Indeed, the dotted left-pointing horizontal ray has no
location to place the solid left-pointing boundary ray of the half
plane.  In contrast, in any given homomorphism
${\pur\kk[\RR^2_+]_\rho} \to {\blu \cM}$, the right-pointing solid
purple ray goes to~$0$ once it exits the blue region.

The tensor product $\kk[\rho] \otimes_{\RR^2} {\blu\cM}$ is nonzero
only on the solid blue segment along the bottom edge of~$\blu \cM$.
Therefore $\bigl(\kk[\rho] \otimes_{\RR^2} {\blu\cM}\bigr)^\rho = 0$.
Thus, while a nonzero homomorphism ${\pur\kk[\cQ_+]_\rho} \to
{\blu\cM}$ detects any ``infinitely backward-extending'' portions
of~$\blu\cM$, the closed top only detects the ``bottom-justified''
such portions.
\end{example}

\begin{remark}\label{r:global-closed-gen}
The notion of global closed cogenerator has a Matlis dual, but since
the dual of an element---equivalently, a homomorphism
from~$\kk[\cQ_+]$---is not an element, the notion of closed generator
is not Matlis dual to a standard notion related to socles.
A~generator along a face~$\rho$ can be defined as an element of~$\cmr$
that is not a multiple of any generator of lesser degree, but making
this precise requires care regarding what ``degree of generator'' and
``lesser'' mean.
\end{remark}

\begin{remark}\label{r:no-local-top}
The notion of local socle has a Matlis dual, but it is not in any
sense a local top, because localization does not Matlis dualize to
localization (Example~\ref{e:dual-of-localization}).  Instead, Matlis
dualizing the local socle yields a functor $\kk \otimes_{\qzr} \cmr$
that surjects onto~$\topcr\cM$ by the Matlis dual of
Proposition~\ref{p:local-vs-global}.  Local socles found uses in
proofs here and there, such as Corollary~\ref{c:at-most-one},
Proposition~\ref{p:sigma-nbd-cogen}, Corollary~\ref{c:soc(coprimary)},
Lemma~\ref{l:nearby}, and Corollary~\ref{c:soc(coprimary)'}, via
Proposition~\ref{p:local-vs-global}.  But since Matlis duals of
statements hold regardless of their proofs, given appropriate
finiteness conditions (Definition~\ref{d:infinitesimally-Q-finite}),
local socles and their Matlis~duals have no further use in this paper.
\end{remark}

In the next lemma, a prerequisite to duality of closed socles and
tops, a localization of~$\cN$ along~$\rho$ on the left-hand side is
hiding in the quotient-restriction
(Definition~\ref{d:quotient-restriction}).

\begin{lemma}\label{l:N/tau-dual}
For any module~$\cN$ over a partially ordered abelian group~$\cQ$ and
face~$\rho$,
$$
  (\cnr)^\vee = (\cN^\vee)^\rho/\rho.
$$
If $\cN\!$ is $\cQ$-finite, then
$$
  (\cN^\vee)/\rho = (\cN^\rho/\rho)^\vee.
$$
\end{lemma}
\begin{proof}
This is Example~\ref{e:dual-of-localization} plus the observation that
quotient-restriction along~$\rho$ commutes with Matlis duality on
modules that are already localized along~$\rho$ as can be seen
directly from Definitions~\ref{d:matlis}
and~\ref{d:quotient-restriction}.  The detailed calculation goes like
this:
$$
  (\cnr)^\vee
  = (\cN_\rho/\rho)^\vee
  = (\cN_\rho)^\vee\!/\rho
  = (\cN^\vee)^\rho/\rho.
$$
To derive the second displayed equation, use Lemma~\ref{l:vee-vee} to
replace $\cN$ by~$\cN^\vee$ in the first equation, and then use
Lemma~\ref{l:vee-vee} again to take the Matlis duals of both sides.
\end{proof}

\begin{thm}\label{t:topc=socc^vee}
For a module~$\cM$ over a partially ordered abelian group~$\cQ$,
\begin{enumerate}
\item\label{i:top-then-dual}%
$(\topcr\cM)^\vee = \socct[\rho](\cM^\vee)$ if\/
$\kk[\rho]\otimes_\cQ\cM$ is $\cQ$-finite, and
\item\label{i:dual-then-top}%
$\topcr(\cM^\vee) = (\socct[\rho]\cM)^\vee$ if $\cM$ is $\cQ$-finite.
\end{enumerate}
\end{thm}
\begin{proof}
The two are similar, but to indicate why the finitenesses must be
assumed, both are written~out.  The first and last lines of each half
are by
Definitions~\ref{d:topct}~and~\ref{d:socct}.\ref{i:global-socc-tau}:
\begin{align*}
(\topcr\cM)^\vee
  &=
  \bigl((\kk[\rho]\otimes_\cQ\cM)^\rho/\rho\bigr){}^\vee
\\&=
  \bigl(\kk[\rho]\otimes_\cQ\cM\bigr){}^\vee\!\big/\rho
  &&\hspace{-3ex}
    \text{by Lemma~\ref{l:N/tau-dual} and~$\cQ$-finiteness of }
    \kk[\rho]\!\otimes_\cQ\!\cM\hspace{-3ex}
\\&=
  \hhom_\cQ\bigl(\kk[\rho],\cM^\vee\bigr)/\rho
  &&\hspace{-3ex}
    \text{by Example~\ref{e:matlis}}
\\&=
  \socct[\rho](\cM^\vee),
\\[1ex]
\hspace{-3ex}\text{and}\quad
(\socct[\rho]\cM)^\vee
  &=
  \bigl(\hhom_\cQ\bigl(\kk[\rho],\cM\bigr)/\rho\bigr){}^\vee
\\*&=
  \bigl(\hhom_\cQ\bigl(\kk[\rho],\cM\bigr){}^\vee\bigr){}^{\rho\!}\big/\rho\!\!\!
  &&\text{by Lemma~\ref{l:N/tau-dual}}
%
\\*&=
  \bigl(\kk[\rho]\otimes_\cQ\cM^\vee\bigr){}^{\rho\!}\big/\rho
  &&\text{by Example~\ref{e:matlis} and~$\cQ$-finiteness of }\cM
\\*&=
  \topcr(\cM^\vee).
  &&\qedhere
\end{align*}
\end{proof}

\begin{remark}
A blanket hypothesis that $\cM$ be $\cQ$-finite would suffice for both
parts of Theorem~\ref{t:topc=socc^vee}, because tensoring the
surjection $\kk[\cQ] \onto \kk[\rho]$ with~$\cM$ yields a surjection
$\cM \onto \kk[\rho]\otimes_\cQ\nolinebreak \cM$,
but the additional generality may be useful.
\end{remark}

\subsection{Generator functors over real polyhedral groups}\label{b:gen}\mbox{}

\medskip
\noindent
The rest of this section passes from closed generators to arbitrary
generators, carrying out in one subsection the Matlis dual of the
non-closed socle theory in Sections~\ref{b:soc} and~\ref{b:soc-along}.
This brevity is possible because, again, the proofs are by applying
Matlis duality to results about socles rather than by dualizing
proofs, as exemplified by the main result of the section,
Theorem~\ref{t:top=socvee}.

To begin, here is the Matlis dual of
Definition~\ref{d:upper-closure-tau}, using
Definition~\ref{d:lower-closure}.

\begin{defn}\label{d:lower-closure-rho}
For a face~$\rho$ of real polyhedral group, set $\nabro =
(\nabr)^\op$, the open star of~$\rho$ (Example~\ref{e:nabla}) with the
partial order opposite to Definition~\ref{d:upper-closure-tau}, so
$$
  \xi \preceq \eta \text{ in }\nabro\text{ if }\xi \subseteq \eta.
$$
The \emph{lower closure functor
along~$\rho$} takes~$\cM$ to the $\cQ \times \nabro$-module $\lr\cM =
\bigoplus_{\xi\in\nabro}\lx\cM$.
\end{defn}

\begin{prop}\label{p:nrtop}
Fix a face~$\rho$ of a real polyhedral group~$\cQ$.  The Matlis
dual~$\cN^\vee$ over~$\cQ$ of any module~$\cN$ over $\cQ\times\nabro$
is naturally a module over $\cQ\times\nabr$ (without altering degrees
in $\nabro = \nabr$).  Moreover,
\begin{enumerate}
\item\label{i:dtop-then-dual}%
$(\nrtop\cN)^\vee = \nrsoc(\cN^\vee)$ for any module $\cN$ over
$\cQ\times\nabro$, and
\item\label{i:dual-then-dtop}%
$\nrtop(\cN^\vee) = (\nrsoc\cN)^\vee$ for any module $\cN$ over
$\cQ\times\nabr$,
\end{enumerate}
where the Matlis duals are taken over~$\cQ$.  All of these hold with
$\nabr$ and~$\nabro$ swapped.
\end{prop}
\begin{proof}
Matlis duality over~$\cQ$ reverses the arrows in the $\nabro$-module
structure on~$\cN$, making $\cN^\vee$ into a module
over~$\cQ\times\nabr$.  An adjointness calculation then yields
\begin{align*}
(\nrtop\cN)^\vee
  &=
  (\kk\otimes_{\nabro}\cN)^\vee
\\&=
  \hhom_{\nabr}(\kk,\cN^\vee)
\\&=
  \nrsoc(\cN^\vee).
\end{align*}
%
The other adjointness is similar, but it uses finiteness of~$\nabro$
via Proposition~\ref{p:dual-topc}:
\begin{align*}
\nrtop(\cN^\vee)
  &=
  \kk \otimes_{\nabro}(\cN^\vee)
\\&=
  \hhom_{\nabr}(\kk,\cN)^\vee
\\*&=
  (\nrsoc \cN)^\vee.
  \qedhere
\end{align*}
%
%
%
%
\end{proof}

\begin{cor}\label{c:lr-vs-dr}
For a face~$\rho$ over a real polyhedral group~$\cQ$, the lower
closure is Matlis dual over~$\cQ$ to the upper closure: as modules
over $\cQ\times\nabro$,
\begin{enumerate}
\item%
$\lr(\cM^\vee) = (\dr\cM)^\vee$ for all $\cQ$-modules~$\cM$ (see
Definitions~\ref{d:lower-closure-rho}
and~\ref{d:upper-closure-tau}), and
\item%
$(\lr\cM)^\vee = \dr(\cM^\vee)$ if $\cM$ is infinitesimally
$\cQ$-finite.
\end{enumerate}
\end{cor}
\begin{proof}
Lemma~\ref{l:lx-vee} plus the first part of Proposition~\ref{p:nrtop}.
\end{proof}

Next is the Matlis dual of Definition~\ref{d:kats}, using the
skyscraper modules~$\kbrx$ there, followed by the Matlis dual of
Definition~\ref{d:soct}.

\begin{defn}\label{d:kbrx}
Fix a partially ordered abelian group~$\cQ$, a face~$\rho$, and an
arbitrary commutative monoid~$\cP$.  Define a functor
$\kk[\rho]\otimes_{\cQ\times\cP}\ $ on modules~$\cN$ over $\cQ \times
\cP$ by
$$
  \kk[\rho]\otimes_{\cQ\times\cP}\cN =
  \bigoplus_{(\bb,\xi) \in \cQ\times\cP} \kbrx\otimes_{\cQ\times\cP}\cN.
$$
\end{defn}

\begin{defn}\label{d:topr}
Fix a real polyhedral group~$\cQ$, a face~$\rho$, and a
$\cQ$-module~$\cM$.  The \emph{generator functor along~$\rho$} takes
$\cM$ to its \emph{top along~$\rho$}: the $(\qrr\times\nabro)$-module
$$
  \topr\cM
  =
  \bigl(\kk[\rho]\otimes_{\cQ\times\nabro}\lr\cM\bigr){}^{\rho\!}\big/\rho.
$$
The $\nabro$-graded components of $\topr\cM$ are denoted by
$\topr[\xi]\cM$ for $\xi \in \nabro$.
\end{defn}


\begin{prop}\label{p:order-either}
Fix a real polyhedral group $\cQ$ and a module~$\cN$ such that
$\kk[\rho]\otimes_\cQ\cN$ is $\cQ$-finite.  The functors $\topcr\!$
and $\nrtop$ commute on~$\cN$.  In particular, if $\cM$ is a
$\cQ$-module such that $\kk[\rho]\otimes_\cQ\lr\cM$ is $\cQ$-finite
(e.g., if~$\cM$ is
$\cQ$-finite), then
$$
  \nrtop(\topcr\lr\cM)
  \cong
  \topr\cM
  =
  \topcr(\nrtop\lr\cM).
$$
\end{prop}
\begin{proof}
This is Matlis dual to Proposition~\ref{p:either-order}, but to prove
it without an infinitesimal $\cQ$-finiteness restriction requires a
direct argument:
\begin{align*}
\nrtop(\topcr\cN)
  &=
  \kk\otimes_{\nabro}\!\bigl(\kk[\rho]\otimes_\cQ\cN\bigr){}^{\rho\!}\big/\rho
  &&\text{by Definitions~\ref{d:topc} and~\ref{d:topct}}
\\
  &=
  \kk\otimes_{\nabro}\!\bigl(\hhom\bigl(\kk[\cQ_+]_\rho,\kk[\rho]\otimes_\cQ\!\cN\bigr)\big/\rho\bigr)\!\!
  &&\text{by Definition~\ref{d:cmr}}
\\
  &=
  \bigl(\kk\otimes_{\nabro}\hhom\bigl(\kk[\cQ_+]_\rho,\kk[\rho]\otimes_\cQ\!\cN\bigr)\hspace{-.2ex}\bigr)\big/\rho\!\!
  &&\text{by Lemma~\ref{l:exact-qr}}
\\
  &=
  \bigl(\kk[\rho]\otimes_\cQ\cN\otimes_{\nabro}\kk\bigr){}^{\rho\!}\big/\rho
  &&\text{by Lemma~\ref{l:matlis-pair}}
\\
  &=
  \topcr(\nrtop\cN).
\end{align*}
The penultimate line is equal to
$\bigl(\kk[\rho]\otimes_{\cQ\times\nabro}\cN\bigr){}^{\rho\!}\big/\rho$,
and $\cN\hspace{-.2ex} = \lr\cM$ yields~$\topr\cM$.
\end{proof}

\begin{remark}\label{r:order-either}
The condition in Proposition~\ref{p:order-either} that
$\kk[\rho]\otimes_\cQ\lr\cM$ be $\cQ$-finite is weaker than $\cM$
being $\cQ$-finite.  Roughly, in each degree the former counts
generators supported on~$\rho$ while the latter takes into account all
elements, generator or otherwise.  The difference is visible when $\cQ
= \RR$ and $\rho = \0$, in which case $\kk[\rho] = \kk$.  The module
$\cM = \bigoplus_{\alpha \in \RR} \kk[\alpha + \RR_+]$ with one
generator in each real degree~$\alpha$ yields a module $\kk
\otimes_\RR \cM$ that is $\RR$-finite, since it has dimension~$1$ in
every graded degree, but $\cM$ itself has uncountable dimension in
each graded degree.
\end{remark}

\noindent
\parbox{\linewidth}{%
\begin{thm}\label{t:top=socvee}
Over a real polyhedral group, the generator functor along a
face~$\rho$ is Matlis dual to the cogenerator functor along~$\rho$:
if~$\cM$ is infinitesimally $\cQ$-finite, then
\begin{enumerate}
\item%
$\topr(\cM^\vee) = (\socr\cM)^\vee$ and
\item%
$(\topr\cM)^\vee = \socr(\cM^\vee)$.
\end{enumerate}
\end{thm}
\begin{proof}
\mbox{}
\vskip -6.5ex
\begin{align*}
\qquad\qquad\quad
(\topr\cM)^\vee
  &=
  \bigl(\nrtop(\topcr\lr\cM)\bigr){}^\vee
  &&\text{by Proposition~\ref{p:order-either}}\qquad\qquad\quad
\\&=
  \nrsoc\bigl((\topcr\lr\cM)^\vee\bigr)
  &&\text{by Proposition~\ref{p:nrtop}.\ref{i:dtop-then-dual}}
\\&=
  \nrsoc\bigl(\socct[\rho]\bigl((\lr\cM)^\vee\bigr)\bigr)
  &&\text{by Theorem~\ref{t:topc=socc^vee}.\ref{i:top-then-dual}}
\\&=
  \nrsoc\bigl(\socct[\rho]\dr(\cM^\vee)\bigr)
  &&\text{by Corollary~\ref{c:lr-vs-dr}.\ref{i:inf}}
\\&=
  \socr(\cM^\vee)
  &&\text{by Proposition~\ref{p:either-order},}
%
%
\\[1ex]
\qquad\qquad\text{and}\quad
(\socr\cM)^\vee
  &=
  \bigl(\nrsoc(\socct[\rho]\dr\cM)\bigr){}^\vee
  &&\text{by Proposition~\ref{p:either-order}}\qquad\qquad\quad
\\&=
  \nrtop\bigl((\socct[\rho]\dr\cM)^\vee\bigr)
  &&\text{by Proposition~\ref{p:nrtop}.\ref{i:dual-then-dtop}}
\\&=
  \nrtop\bigl(\topcr\bigl((\dr\cM)^\vee\bigr)\bigr)
  &&\text{by Theorem~\ref{t:topc=socc^vee}.\ref{i:dual-then-top}}
\\&=
  \nrtop\bigl(\topcr\lr(\cM^\vee)\bigr)
  &&\text{by Corollary~\ref{c:lr-vs-dr}.\ref{i:all}}
\\&=
  \topr(\cM^\vee)
  &&\text{by Proposition~\ref{p:order-either}.}\qedhere
\end{align*}
\end{proof}
}%

\section{Essential properties of tops}\label{s:tops}

Characterizing injective homomorphisms (Theorem~\ref{t:injection})
via socles admits a dual assertion characterizing surjections, in both
the real (Theorem~\ref{t:surjection-RR}) and discrete
(Theorem~\ref{t:surjection-ZZ}) settings.  Furthermore, notions
surrounding minimal primary decomposition and associated faces dualize
to secondary decomposition and attached faces
(Theorem~\ref{t:cover-M}), with corresponding concepts of density
(Definition~\ref{d:nearby^}) and a corresponding dense relaxation of
the essentiality of tops (Theorem~\ref{t:dense-top}).

To start, here is the dual to Definition~\ref{d:associated}
and~\mbox{Theorem}~\ref{t:coprimary}.%

\begin{defn}\label{d:attached}
Fix a face $\rho$ and a module~$\cM$ over a real or discrete
polyhedral~group.
\begin{enumerate}
\item%
The face~$\rho$ is \emph{attached} to~$\cM$ if $\topr\cM \neq 0$
(Definition~\ref{d:topr}).
\item%
If $\rho$ is attached to $\cM = \kk[I]$ for an interval~$I$ then
$\rho$ is \emph{attached}~to~$I$.
\item%
The set of attached faces of~$\cM$ or~$I$ is denoted by $\att\cM$
or~$\att I$.
\item%
The module~$\cM$ is \emph{$\rho$-secondary} if $\att(\cM) = \{\rho\}$.
\end{enumerate}
\end{defn}

Next comes the Matlis dual of Definitions~\ref{d:downset-hull}
and~\ref{d:interval-hull}.

\begin{defn}\label{d:minimal-cover}
An \emph{interval cover} of a module~$\cM$ over an arbitrary poset is
a surjection $\bigoplus_{j \in J} F^j \onto \cM$ with each $F^j$ an
interval module.  The cover is \emph{finite} if $J$ is~finite.  It is
an \emph{upset cover} if the intervals are all upsets.  The
module~$\cM$ is \emph{upset-finite} if it admits a finite upset cover.
If the poset is a real or discrete polyhedral group, the~cover~is
\begin{enumerate}
\item\label{i:secondary}%
\emph{secondary} if $F^j = \kk[I_j]$ is secondary for all~$j$, so
$I_j$ is a secondary upset,~and
\item%
\emph{minimal} if the induced map $\topr F \to \topr\cM$ is an
isomorphism for all faces~$\rho$.
\end{enumerate}\setcounter{separated}{\value{enumi}}
\end{defn}

The Matlis dual to Theorems~\ref{t:injection} has an infinitesimal
$\cQ$-finiteness hypotheses because duality between tops and socles in
Theorem~\ref{t:top=socvee} requires it (see
also~\mbox{Example}~\ref{e:surjection}), but the dual to
Theorem~\ref{t:discrete-injection} has only the upset-finiteness dual
to downset-finiteness.

\begin{thm}[Essentiality of real tops]\label{t:surjection-RR}
Fix a homomorphism $\phi: \cN \to \cM$ of~modules over a real
polyhedral group~$\cQ$.
\begin{enumerate}
\item\label{i:phi=>topr}%
If $\phi$ is surjective with $\cM$ and~$\cN$ both being
infinitesimally $\cQ$-finite modules, then $\topr\phi: \topr\cN \to
\topr\cM$ is surjective for all faces~$\rho$~of~$\cQ_+$.
\item\label{i:topr=>phi}%
If $\topr\phi: \topr\cN \to \topr\cM$ is surjective for all
faces~$\rho$ of~$\cQ_+\!$~and~$\cM$ is upset-finite, then $\phi$
is~surjective.\hfill$\square$
\end{enumerate}
\end{thm}

\begin{remark}\label{r:nakayama}
One of the versions of Nakayama's lemma says that a homomorphism $M
\to N$ of finitely generated modules over a local ring~$R$ is
surjective if and only if it becomes surjective upon tensoring with
the residue field~$\kk$.  In the language of tops and socles, $M
\otimes_R \kk = \top M$.  Therefore Theorem~\ref{t:surjection-RR} is
the direct generalization of Nakayama's lemma to multigraded modules
over real-exponent polynomial rings.  Some finiteness is still
required, but it is vastly weaker than finitely generated, rather
requiring roughly that the generators can be gathered into finitely
many coherent clumps.  There is, in addition, a quintessentially
real-exponent further weakening that allows the top to be replaced by
a dense image (Theorem~\ref{t:dense-top}).
\end{remark}

\begin{thm}[Essentiality of discrete tops]\label{t:surjection-ZZ}
Fix a homomorphism $\phi: \cN \to \cM$ of modules over a discrete
polyhedral group~$\cQ$.
\begin{enumerate}
\item\label{i:phi=>topcr}%
If $\phi$ is surjective then $\topcr\phi: \topcr\cN \to \topcr\cM$ is
surjective for all faces~$\rho$~of~\hspace{.3ex}$\cQ_+$.
\item\label{i:topcr=>phi}%
If $\topcr\phi: \topcr\cN \to \topcr\cM$ is surjective for all
faces~$\rho$ of~$\cQ_+\!$~and~$\cM$ is upset-finite, then $\phi$
is~surjective.\hfill$\square$
\end{enumerate}
\end{thm}

\begin{remark}\label{r:left-endpoints}
In terms of persistent homology, Theorems~\ref{t:surjection-RR}
and~\ref{t:surjection-ZZ} say that a homomorphism of multipersistence
modules is surjective if and only if it maps the ``left endpoints'' of
the source surjectively onto the ``left endpoints'' of the target.
\end{remark}

\begin{example}\label{e:surjection}
Some hypothesis is needed in
Theorem~\ref{t:surjection-RR}.\ref{i:phi=>topr}, in contrast to
Theorem~\ref{t:injection}.\ref{i:phi=>soct} or indeed
Theorem~\ref{t:surjection-ZZ}.\ref{i:phi=>topr}.  Let $\cM = \kk[U]$
for the open half-plane $U \subset \RR^2$ above the antidiagonal line
$y = -x$.  Then $\cM$ is $\{\0\}$-secondary, with
$(\top_{\{\0\}}^\xi\!\cM)_\bb \neq 0$ precisely when $\bb$ lies on the
antidiagonal and $\xi$ is the $x$-axis or $y$-axis.  The direct sum
$\bigoplus_{\bb \neq \0} \bigl(\kk[\bb + \qny[x]] \oplus \kk[\bb +
\qny[y]]\bigr)$ surjects onto~$\cM$, but the map on tops fails to hit
any element in $\RR^2$-degree~$\0$.  This kind of behavior might lead
one to wonder: why is its Matlis dual not a counterexample to
Theorem~\ref{t:injection}.\ref{i:phi=>soct}?  Because $\cM^\vee$ does
not possess a well defined map to a direct sum indexed by $\aa \neq
\0$ along the antidiagonal line, only to a direct product.  Any
sequence of points $\vv_k \in -U$ converging to~$\0$ yields a sequence
of elements $z_k \in \cM^\vee$.  The image of the sequence
$\{z_k\}_{k=1}^\infty$ in any particular one (or finite direct sum) of
the downset modules of the form $\kk[\aa - \qny[x]]$ with $\aa \neq
\0$ is eventually~$0$, but in the direct product the sequence
$\{z_k\}_{k=1}^\infty$ survives forever.  The direct limit of the
image sequence witnesses the nonzero socle of the direct product at
the missing~point~$\0$.
\end{example}

\begin{thm}\label{t:cover-M}
Every upset-finite module~$\cM$ over a real or discrete polyhedral
group admits a minimal secondary interval cover.  When the polyhedral
group is discrete, it is possible to use upsets for all of the
intervals.
\end{thm}
\begin{proof}
This is the Matlis dual of Theorems~\ref{t:interval-hull}
and~\ref{t:discrete-hull-M}, using
Example~\ref{e:infinitesimally-Q-finite} to allow the results of
Section~\ref{s:gen-functors} to be applied at will, as the strongest
hypothesis there is infinitesimal $\cQ$-finiteness.
\end{proof}

\begin{remark}\label{r:antisymmetry}
Matlis duality in persistent homology might appear to indicate that
generators (births) are in adamantine antisymmetry with cogenerators
(deaths), but when it comes to interactions between the two, the
symmetry is broken by the partial order on~$\cQ$: elements in
$\cQ$-modules move from birth inexorably toward death.
Definition~\ref{d:divides'} treats elements functorially, as
homomorphisms from the monoid algebra of the positive cone.  Doing so
makes it clear that the dual of an element is not an element.  It is
instead a homomorphism to the injective hull of the residue field, as
in Definition~\ref{d:basin}.  This complication in dealing with
generators rather than cogenerators cements the choice to develop the
theory in terms of cogenerators in
Sections~\ref{s:staircases}--\ref{s:discrete}.
\end{remark}

Density considerations are important for use in connection with
resolutions (Section~\ref{s:min}).  They dualize directly, but
phrasing them accurately is touchy because of issues like those in
Remarks~\ref{r:global-closed-gen} and~\ref{r:no-local-top}.  Clearer
duality comes from a functorial recasting of
Definition~\ref{d:divides}; the following is precisely equivalent to
that definition.

\begin{defn}\label{d:divides'}
An \emph{element} $\kk[\bb + \cQ_+] \stackrel{\beta\:}\to \cM$ is said
to \emph{divide} a \emph{closure element} \noheight{$\kk[\aa + \cQ_+]
\stackrel{\alpha\;}\to \dsm$} if $\bb \in \aa - \qns = \aa -
\sigma^\circ - \cQ_+\!$~(Lemma~\ref{l:<<}) and $\alpha$ equals
the~composite
$$
  \kk[\aa + \cQ_+]
  \into
  \kk[\bb + \cQ_+]
  \stackrel{\beta\:}\too
  \cM_\bb
  \to
  \cM_{\aa-\sigma}
$$
of the inclusion of principal upset $\cQ$-modules induced by $\aa +
\cQ_+ \subseteq \bb + \cQ_+$ with~$\beta$ and the natural map from
Lemma~\ref{l:natural}.  The element $\beta$ is said to
\emph{$\sigma$-divide}~$\alpha$ if, more restrictively, $\bb \in \aa -
\sigma^\circ$.
\end{defn}

\begin{defn}\label{d:basin}
Fix a module~$\cM$ over a real polyhedral group~$\cQ$.
\begin{enumerate}
\item%
A \emph{basin} $\cM \stackrel{\beta\:}\to \kk[\bb - \cQ_+]$ is said to
\emph{attract} a \emph{closure basin} $\lx\cM
\stackrel{\alpha\;}\to\nolinebreak \kk[\aa - \cQ_+]$ if $\bb \in \aa +
\qnx = \aa + \xi^\circ + \cQ_+\!$~(Lemma~\ref{l:<<}) and $\alpha$
equals the composite
$$
  \cM_{\aa+\xi}
  \to
  \cM
  \stackrel{\beta\:}\too
  \kk[\bb - \cQ_+]
  \onto
  \kk[\aa - \cQ_+]
$$
of the natural map from Lemma~\ref{l:natural-dual} with~$\beta$ and
the surjection of coprincipal downset $\cQ$-modules induced by $\bb -
\cQ_+ \supseteq \aa - \cQ_+$.
\item%
The basin $\beta$ is said to \emph{$\xi$-attract}~$\alpha$ if, more
restrictively, $\bb \in \aa + \xi^\circ$.
\item%
The basin $\beta$ is \emph{$\xi$-secondary} if its image in~$\kk[\bb -
\cQ_+]$ is a $\xi$-secondary module.
\end{enumerate}
\end{defn}

\begin{example}\label{e:basin}
Because of quotients modulo faces in Definition~\ref{d:topr}, a basin
$$
  \wt{t}: \topr[\xi]\cM \to \kk[\wt\aa-\qrrp]
$$
takes values modulo~$\rho$.  This basin lifts canonically to a
homomorphism
$$
  (\kk[\rho]\otimes_{\cQ\times\nabro}\lx\cM)^\rho
  \to
  \kk[\aa - \cQ_+ + \RR\rho]
$$
that is not itself a basin but is induced by (perhaps many) basins
$$
  \kk[\rho]\otimes_{\cQ\times\nabro}\lx\cM
  \to
  \kk[\aa - \cQ_+]
$$
under applying the Matlis dual $(-)^\rho$ of localization
(Definition~\ref{d:cmr}).  Note that
$$
  \lx\cM \onto \kk[\rho]\otimes_{\cQ\times\nabro}\lx\cM
$$
by right-exactness of colimits.  (For the current purpose,
surjectivity of this last map is irrelevant, but it might be handy to
keep in mind for intuition.)  Composing these various lifts, the basin
\noheight{$\wt t$} lifts to (perhaps many)
\mbox{closure}~basins~\mbox{$t: \lx \to \kk[\aa - \cQ_+]$}.
\end{example}

\begin{defn}\label{d:nearby^}
Fix a module~$\cM$ over a real polyhedral group~$\cQ$.
\begin{enumerate}
\item\label{i:nbd-top}%
A \emph{neighborhood in~$\topr\cM$} of a basin \noheight{$\wt t$}
of~$\topr[\xi]\cM$ is $\topr(\beta \cM)$ for a $\xi$-secondary
basin~$\beta$ that $\xi$-attracts a lifted closure basin~$t$
of~$\lr\cM$ (Example~\ref{e:basin}).

\item\label{i:dense-top}%
A surjection $\topr\cM \onto T_\rho$ of $(\qrr\times\nabro)$-modules
is \emph{dense} if for all $\xi \supseteq \rho$, every neighborhood of
every basin of~$\topr[\xi]\cM$ has nonzero image in~$T_\rho$.

\item\label{i:quotient-functor}%
A quotient functor $\topr \onto \ctr$ from modules over~$\cQ$ to
modules over $\qrr\times\nabro$ is \emph{dense} if $\topr\kk[U] \onto
\ctr\kk[U]$ is dense for all faces $\rho$ and upsets $U \subseteq
\cQ$.
\end{enumerate}
\end{defn}

\begin{remark}\label{r:nearby^}
Definition~\ref{d:nearby^} skips the duals to notions of ``nearby''
and ``vicinity'' from Definition~\ref{d:nearby} because the work of
defining ``coprimary''---and hence ``secondary'', by taking
duals---has already been done in Section~\ref{s:hulls}.  The contrast
between Definitions~\ref{d:nearby^} and~\ref{d:nearby} is simple:
principal primary submodules become coprincipal secondary quotients.
Definition~\ref{d:nearby^}.\ref{i:quotient-functor} is Matlis dual to
Definition~\ref{d:dense-subfunctor}.
\end{remark}

\begin{thm}\label{t:dense-top}
Fix quotient functors $\topr \onto \cT_\rho$ for all faces~$\rho$ of a
real polyhedral group.
\hspace{-2.5pt}Theorem~\hspace{-1.6pt}\ref{t:surjection-RR} holds with
$\cT\!$ instead of $\top$ if and only if $\topr \!\onto\!
\ctr$~is~dense~for~all~$\rho$.
\end{thm}
\begin{proof}
Apply the exact Matlis duality functor to the statement of
Theorem~\ref{t:dense-subfunctor} in the presence of the finiteness hypothesis in
Theorem~\ref{t:surjection-RR}.
\end{proof}

The straightforward dualization of primary decomposition in
Sections~\ref{b:minimal-primary} and~\ref{s:discrete} to secondary
decomposition is omitted.

\section{Minimal presentations over discrete or real polyhedral groups}\label{s:min}

Algebra of modules over arbitrary posets \cite{hom-alg-poset-mods} and
primary decomposition over partially ordered abelian groups
\cite{prim-decomp-pogroup} lack a crucial aspect of noetherian
commutative algebra, namely minimality.  Much of the edifice of modern
commutative algebra is built on numerical, homological, combinatorial,
or geometric behavior whose quantification rests firmly on notions of
minimality: Betti numbers, Castelnuovo--Mumford regularity, primary
and irreducible decomposition, homological dimension, computational
complexity bounds---all of these depend on minimal resolutions, or
minimal decompositions, or minimal degrees of some nature.  When the
partially ordered group is a real vector space, earlier sections
rescue notions of minimality, perhaps with density amendments, for
generators and decompositions.  This section explores to what extent
minimality applies to presentations and resolutions.

\begin{defn}[{\cite[Definitions~3.16, 6.1, 6.4]{hom-alg-poset-mods}}]%
\label{d:presentation}%
Fix a module~$\cM$ over a partially ordered abelian group~$\cQ$.
\begin{enumerate}
\item\label{i:upset-presentation}%
An \emph{upset presentation} of~$\cM$ is an expression of~$\cM$ as the
cokernel of a homomorphism $F_1 \to F_0$ such that each $F_i$ is a
direct sum of upset modules.

\item\label{i:downset-copresentation}%
A \emph{downset copresentation} of~$\cM$ is an expression of~$\cM$ as
the kernel of a homomorphism $E^0 \to E^1$ such that each $E^i$ is a
direct sum of downset modules.

\item\label{i:fringe}%
A \emph{fringe presentation} of~$\cM$ is a direct sum~$F$ of upset
modules~$\kk[U]$, a direct sum~$E$ of downset modules~$\kk[D]$, and a
homomorphism $F \hspace{-.4pt}\to\hspace{-.4pt} E$ of
\mbox{$\cQ$-modules}~with
\begin{itemize}
  \item%
  image isomorphic to~$\cM$ and
  \item%
  components $\kk[U] \to \kk[D]$ that are connected
  (Definition~\ref{d:connected-homomorphism}).
\end{itemize}

\item%
An \emph{upset resolution} of~$\cM$ is a complex~$F_\spot$ of
$\cQ$-modules, each a direct sum of upset modules, whose differential
$F_i \to F_{i-1}$ decreases homological degrees and has only one
nonzero homology $H_0(F_\spot) \cong \cM$.

\item%
A \emph{downset resolution} of~$\cM$ is a complex~$E^\spot$ of
$\cQ$-modules, each a direct sum of downset modules, whose
differential $E^i \to E^{i+1}$ increases cohomological degrees and has
only one nonzero homology $H^0(E^\spot) \cong\nolinebreak \cM$.
\end{enumerate}
Any one of these is an \emph{indicator presentation} or
\emph{indicator resolution}.
\end{defn}

\begin{defn}\label{d:conditions}
Each indicator presentation or resolution in
Definition~\ref{d:presentation}
\begin{enumerate}
\item\label{i:finite}%
is \emph{finite} if it has only finitely many summands in total;

\item\label{i:dominate}%
\emph{dominates} a constant subdivision
(Definition~\ref{d:constant-subdivision}) or poset encoding
(Definition~\ref{d:encoding}) of~$\cM$ if the morphism or
differentials do (Definition~\ref{d:tame-morphism});

\item\label{i:semialg-or-PL}%
is \emph{semialgebraic} or \emph{PL} if the morphism has that type
(Definition~\ref{d:tame-morphism}).
\end{enumerate}
\end{defn}

\pagebreak[3]

\begin{defn}\label{d:morphism-min-dense}
Over a real polyhedral group, a module morphism $\phi: \cM \to \cN$ is
\begin{enumerate}
\item\label{i:soc-min-dense}%
\emph{injectively minimal} or \emph{injectively dense} if the
canonical inclusion $\image\phi \into \cN$ induces an isomorphism
$\soc(\image\phi) \simto \soc\cN$ or dense
inclusion $\soc(\image\phi) \into\nolinebreak \soc\cN$;

\item\label{i:top-min-dense}%
\emph{surjectively minimal} or \emph{surjectively dense} if the
canonical surjection $\cM \hspace{-1.1pt}\onto\hspace{-1.1pt}
\image\phi$~induces an isomorphism $\top\cM
\hspace{-1pt}\simto\hspace{-2pt} \top(\image\phi)$ or dense
surjection $\top\cM \hspace{-2pt}\onto\hspace{-1.2pt}
\top(\image\phi)$.
\end{enumerate}
Over a discrete polyhedral group the definition of \emph{injectively
minimal} and \emph{surjectively minimal} are unchanged.  In either the
real or discrete polyhedral setting, a complex of modules is
\emph{injectively} or \emph{surjectively} \emph{minimal} or
\emph{dense} if all of its differentials are.
\end{defn}

\begin{remark}\label{r:coimage}
Category-theoretically, injective minimality or density should
naturally be phrased in terms of the image morphism of~$\phi$, while
surjective minimality and density should be phrased in terms of the
coimage morphism of~$\phi$.
\end{remark}

\begin{remark}\label{r:min-or-dense}
The notion of minimal morphism makes sense in ordinary commutative
algebra much more generally: minimal resolutions and essential
submodules are transparently special cases.  Irredundant irreducible
decompositions $0 = \bigcap W_j$ in a module~$M$ also correspond to
also injectively minimal morphisms $M \into \bigoplus M/W_j$.  In
contrast, for historical reasons, a minimal primary decomposition $0 =
\bigcap P_j$ in a module~$M$ is usually defined to have a minimal
number of intersectands, a condition that need not induce an
injectively minimal morphism $M \into \bigoplus M/P_j$.  Consequently,
minimal primary decompositions by this definition suffer from annoying
non-uniqueness.  For example, the $\pp$-primary component in one
minimal primary decomposition can strictly contain the $\pp$-primary
component in another.  Defining a primary decomposition to be
\emph{minimal} precisely when it induces an injectively minimal
morphism would rectify this containment problem and other defects.
\end{remark}

\begin{defn}\label{d:min-or-dense}
Fix a module~$\cM$ over a real or discrete polyedral group~$\cQ$.
\begin{enumerate}
\item%
A downset copresentation or resolution $E^\spot$ of~$\cM$ is
\emph{minimal} or \emph{dense} if the~exact augmented complex $0
\hspace{-1pt}\to\hspace{-1.43pt} \cM
\hspace{-1.43pt}\to\hspace{-1.43pt} E^\spot$ is correspondingly
injecitvely minimal or~dense.

\item%
An upset presentation or resolution $F^\spot$ of~$\cM$ is
\emph{minimal} or \emph{dense} if the exact augmented complex $0 \from
\cM \from F^\spot$ is correspondingly surjectively minimal~or~dense.

\item%
A fringe presentation $F \to E$ of~$\cM$ is \emph{minimal} or
\emph{dense} if it is the composite of a correspondingly minimal or
dense upset cover and downset hull of~$\cM$.
\end{enumerate}
\end{defn}

\begin{thm}\label{t:presentations-dense}
A module over a real polyhedral group~$\cQ$ is tame if and only if
it~admits
\begin{enumerate}
\item\label{i:fringe'}%
a dense finite fringe presentation; or
\item\label{i:upset-presentation'}%
a dense finite upset presentation; or
\item\label{i:downset-copresentation'}%
a dense finite downset copresentation.
\end{enumerate}
Over a discrete polyedral group these presentations can be chosen
minimal instead of dense.  When the module is semialgebraic or PL
these presentations can all be chosen semialgebraic or PL,
respectively.
\end{thm}
\begin{proof}
In both the real and discrete cases, any one of these presentations
is, in particular, finite, so the existence of any of them implies
that the module is tame by the syzygy theorem
\cite[Theorem~6.12]{hom-alg-poset-mods}.  It is the other direction
that requires the theory in this paper.

In the real polyhedral case, any finite downset hull can be densitized
by Theorem~\ref{t:dense-hull-M} and Remark~\ref{r:filtration}.  The
Matlis dual of this statement says that any finite upset cover can be
densitized, as well.  Composing these from a given finite fringe
presentations yields a dense finite fringe presentation.  In addition,
the cokernel of any downset hull (dense or otherwise) of a tame module
is tame by Proposition~\ref{p:abelian-category}, so the cokernel has a
dense finite downset hull by Theorem~\ref{t:dense-hull-M} again.  That
yields a dense finite downset copresentation.  The Matlis dual of a
dense finite downset copresentation of the Matlis dual~$\cM^\vee$ is a
dense upset presentation of~$\cM$ by Theorem~\ref{t:top=socvee} (which
applies unfettered to tame modules by
Example~\ref{e:infinitesimally-Q-finite}).

The minimal discrete polyhedral case follows the parallel proof, using
Theorem~\ref{t:discrete-hull-M} and Remark~\ref{r:discrete-filtration}
instead of Theorem~\ref{t:dense-hull-M} and Remark~\ref{r:filtration}.

If $\cM$ is semialgebraic, then the densitization procedure in
Theorem~\ref{t:dense-hull-M} and Remark~\ref{r:filtration} is
semialgebraic by induction on the number~$k$ of summands there, the
base case being the canonical primary decomposition of a semialgebraic
interval in Theorem~\ref{t:hull-I}, which is semialgebraic by
Theorem~\ref{t:soct-tame}.
\end{proof}

\begin{remark}\label{r:presentations-minimal}
If minimal instead of dense presentations are desired in the real
polyhedral setting, then they can be achieved by combining
Definitions~\ref{d:interval-hull} and~\ref{d:presentation} to form
\emph{interval copresentations} instead of downset copresentations, or
the Matlis dual \emph{interval presentations} instead of upset
presentations.  Splicing these yields \emph{interval fringe
presentations} instead of fringe presentations.
Theorem~\ref{t:interval-hull} and the interval version of
Remark~\ref{r:filtration} forms the basis for a minimalizing version
of the proof of Theorem~\ref{t:presentations-dense}.
\end{remark}

\begin{remark}\label{r:missing-items}
Comparing Theorem~\ref{t:presentations-dense} to the syzygy theorem
for tame modules over arbitrary posets
\cite[Theorem~6.12]{hom-alg-poset-mods}, various items are missing.
\begin{enumerate}
\item%
Theorem~\ref{t:presentations-dense} makes no claim concerning whether
the presentations can be densitized if a poset encoding $\pi: \cQ \to
\cP$ (Definition~\ref{d:encoding}) has been specified beforehand.  It
is a~priori possible that deleting redundant generators of upsets and
cogenerators of downsets could prevent an indicator summand from being
constant on~fibers~of~$\pi$.

\item%
Theorem~\ref{t:presentations-dense} makes no claim concerning finite
poset encodings dominating any one of the three presentations there,
but as each of these presentations is finite, existence is already
implied by the syzygy theorem for tame modules
\cite[Theorem~6.12]{hom-alg-poset-mods}, including semialgebraic
and~PL considerations.

\item\label{i:res}%
Theorem~\ref{t:presentations-dense} makes no claim concerning
finiteness of minimal or dense indicator resolutions.  Dense
resolutions of tame modules over real polyhedral groups (or minimal
ones in the discrete polyhedral setting) can be constructed from
scratch by Theorem~\ref{t:dense-hull-M},
Theorem~\ref{t:discrete-hull-M}, and their Matlis duals, but there is
no a~priori guarantee that such resolutions must terminate after
finitely many steps.
\end{enumerate}
\end{remark}

\begin{conj}\label{conj:finite-dense}
Every tame, semialgebraic, or PL module~$\cM$ over a real polyhedral
group~$\cQ$ has finite dense downset and upset resolutions of the
corresponding type.
\end{conj}

\begin{conj}\label{conj:finite-minimal}
Every tame module~$\cM$ over a discrete polyhedral group~$\cQ$ has
finite minimal downset and upset resolutions of the corresponding
type.
\end{conj}

\begin{remark}\label{r:syzygy}
Remark~\ref{r:missing-items}.\ref{i:res} raises an intriguing point
about indicator resolutions: the bound on the length in the syzygy
theorem over arbitrary posets \cite[Theorem~6.12]{hom-alg-poset-mods}
comes from the order dimension of an encoding poset, which is more or
less unrelated to the dimension of the real or discrete polyhedral
group.  It seems plausible that the geometry of the polyhedral group
asserts control to prevent the lengths from going too high, just as it
does to prevent the cohomological dimension of an affine semigroup
ring from going too high via Ishida complexes to compute local
cohomology \cite[Section~13.3.1]{cca}.  This points to potential value
of developing a derived functor side of the top-socle / birth-death /
generator-cogenerator story for indicator resolutions to solve
Conjecture~\ref{conj:indicator-dimension-n}, which would be an even
tighter indicator analogue of the Hilbert Syzygy Theorem.
\end{remark}

\begin{defn}\label{d:indicator-dimension}
Fix a module~$\cM$ over a poset~$\cQ$.
\begin{enumerate}
\item%
The \emph{downset-dimension} of~$\cM$ is the smallest length of a
downset resolution of~$\cM$.

\item%
The \emph{upset-dimension} of~$\cM$ is the smallest length of an upset
resolution of~$\cM$.

\item%
The \emph{indicator-dimension} of~$\cM$ is maximum of its downset- and
upset-dimensions.

\item%
The \emph{indicator-dimension} of~$\cQ$ is the maximum of the
indicator-dimensions of its tame modules.
\end{enumerate}
\end{defn}

\begin{conj}\label{conj:indicator-dimension-n}
The indicator-dimension of any real or discrete polyhedral group~$\cQ$
is bounded above by the rank of~$\cQ$ (as an $\RR$-vector space or
abelian group, respectively).
\end{conj}

\begin{remark}\label{r:conj}
No uniform bound on the lengths of finite upset and downset
resolutions over an arbitrary posets~$\cQ$ is known when~$\cQ$ has
quotients with unbounded order~dimension.  It is already open to find
a module over~$\RR^2$ whose indicator-dimension is provably as high
as~$2$.  It would not be shocking if the rank of~$\cQ$ were a strict upper
bound---that is, if the indicator-dimension in
Conjecture~\ref{conj:indicator-dimension-n} were always strictly less
than the rank: the use of upset modules
instead of free modules could prevent the final syzygies that, in
finitely generated situations, come from elements supported at the
origin by local duality.
\end{remark}

\begin{remark}\label{r:relative-hom-alg}
Conjecture~\ref{conj:indicator-dimension-n} asks for bounds on
arbitrary resolutions of modules by indicator modules.  In contrast,
relative homological algebra asks for resolutions
relative to a nonstandard exact structure.  When the indecomposable
projectives are
the indicator modules for connected intervals (or \emph{spreads}),
this leads to a different notion of \emph{spread-global dimension};
see \cite{blanchette-desrochers-hanson-scoccola2025} and the
references therein for details.  In particular, the spread-global
dimension of $\ZZ^2$-modules is infinite
\cite[Ex.\,7.24.3]{blanchette-brüstle-hanson2025}.
\end{remark}

\addtocontents{toc}{\protect\setcounter{tocdepth}{2}}

\end{document}